\documentclass[final,1p,times]{elsarticle}
\usepackage{amssymb}
 \usepackage{amsthm}
 \usepackage{caption}
\usepackage{amsmath,amssymb,amsopn,amsfonts,mathrsfs,amsbsy,amscd}
\usepackage{longtable}
\usepackage{multirow}
\usepackage[latin1]{inputenc}
\newcommand{\tr}{\mathrm{tr}}

\newcommand{\T}{\mathfrak{t}}
\newcommand{\R}{\mathbb{R}}

\newcommand{\G}{{\mathfrak{g}}}

\newcommand{\h}{{\mathfrak{h}}}

\newcommand{\e}{\check{e}}
\newtheorem{Def}{Definition}
\newtheorem{theo}{Theorem}
\newtheorem{pr}{Proposition}

\newtheorem{co}{Corollary}

\newtheorem{remark}{Remark}

\usepackage{adjustbox}
\usepackage{longtable}
\usepackage{geometry}
\makeatletter
\def\ps@pprintTitle{%
   \let\@oddhead\@empty
   \let\@evenhead\@empty
   \let\@oddfoot\@empty
   \let\@evenfoot\@oddfoot}
\makeatother

\begin{document}
\begin{frontmatter}
\title{ Eight-Dimensional Symplectic Nilpotent Lie Groups with Lagrangian Normal Subgroups: A Complete Classification}
\author[]{T. A\"it Aissa, M. W. Mansouri  }
\address{Department of Mathematics, Faculty of Sciences, Ibn Tofail University\\
	Analysis, Geometry and Applications Laboratory $($LAGA$)$\\ Kenitra, Morocco\\e-mail:
	tarik.aitaissa @uit.ac.ma\\	mansourimohammed.wadia@uit.ac.ma}
\begin{abstract}
We investigate symplectic nilpotent Lie groups with Lagrangian normal subgroups. We show that there exists a bijection between the isomorphism classes of nilpotent Lie groups with Lagrangian normal subgroups and the isomorphism classes  of geodesically complete, flat, nilpotent Lie groups with  Lagrangian extension cohomology class. Finally, we provide a complete classification of eight-dimensional symplectic nilpotent Lie groups with Lagrangian normal subgroups,  identifying exactly ninety-five such groups. As a consequence, we obtain a complete classification of eight-dimensional symplectic filiform real Lie groups.
\end{abstract}

        \end{frontmatter}
\section{Introduction}\label{sec1}
 Let $G$ be a connected and simply connected Lie group with Lie algebra $\G$, which we identify with the space of left-invariant vector fields on $G$. Denote by $\mathfrak{g}^\ast$ the dual vector space of $\mathfrak{g}$, 
regarded as the space of left-invariant one-forms on $G$. Let $TG$ be the tangent bundle of $G$ and let $T^*G$ denote its cotangent bundle. For a differential one-form $\mu$ on $G$, we denote its value at a point $g \in G$ by
\begin{align*}
\mu_g \in T_g^*G = (T_g G)^*.
\end{align*}
Henceforward, the group $G$ is identified with the zero section in $T^\ast G$, and $\G^\ast$ is identified with the fiber over the neutral element $e$ of $G$ in $T^\ast G$. Therefore, we will identify $ T_{(e,0)}T^*G$ with $\G \times \G^\ast$. Let $\Omega$ denote the canonical skew-symmetric bilinear form on $\G \times \G^\ast$, defined explicitly by the relation
\begin{align}\label{Canoform}
\Omega\big((x,\alpha),(y,\beta) \big) = \alpha(y)-\beta(x)
\end{align}
for all $(x,\alpha),(y,\beta) \in \G\times \G^\ast $. Let us endow $\G^\ast$ with its structure of a vector Lie group. $T^\ast G$ carries a Lie group structure whose composition law is given as:
\begin{align}\label{loi}
(X,\alpha)\cdot(Y,\beta)=\big(X Y,\mathrm{L}^\ast_{X^{-1}}\beta+\mathrm{R}^\ast_{Y^{-1}}\alpha\big)
\end{align}
In the group structure defined by equation $(\ref{loi})$, the maps $\mathrm{L}_X$ and $\mathrm{R}_Y$ represent left and right translation by $X$ and $Y$, respectively. A classical result states that the 2-form $\Omega$ is left-invariant and closed with respect to this Lie group structure $(\ref{loi})$ if and only if $G$ is abelian. This naturally leads to the so-called \textit{cotangent extension problem} \cite{Boyom} does there exist a Lie group structure on $T^\ast G$ for which the left-invariant two-form induced by $\Omega$ is closed?

 More precisely, let $G$, $\G$, $\G^\ast$, and $\Omega$ be as defined above. The problem is to find all Lie group structures on $T^\ast G$ (or equivalently, all Lie algebra structures on $\G \oplus \G^\ast$) which satisfy the following conditions:
\begin{enumerate}
    \item[$(i)$] The natural sequence of Lie groups
 \[
          0 \hookrightarrow \G^\ast \hookrightarrow T^\ast G\longrightarrow \G\longrightarrow 1
          \]
 is exact, where $\G^\ast$ is endowed with its vector group  structure.
    \item[$(ii)$] The left-invariant 2-form on $T^\ast G$ induced by $\Omega$ is closed (i.e., it is a left-invariant symplectic form). 
\end{enumerate}
The total space of the cotangent bundle $T^\ast G$ of a Lie group $G$, equipped with its canonical symplectic form $\Omega$, is a symplectic manifold. In this setting, the subgroup $\G$ is Lagrangian, and $\G^\ast$ is a normal Lagrangian subgroup. A Lie algebra resulting from this construction is called a solution to the cotangent extension problem. Boyom $\cite{Boyom}$ used this method to construct models of symplectic Lie algebras.

Solutions to the cotangent extension problem on $T^\ast G$ are characterized by certain invariants of Koszul-Vinberg structures on the Lie algebra $\G$ of $G$. Recall that a Koszul-Vinberg structure on $\G$ (or on $G$) is a bilinear product
\begin{align*}
\G\times\G\longrightarrow\G
\end{align*}
which satisfy the following axioms:
\begin{enumerate}
\item[$KV_1$] $x\cdot y-y\cdot x=[x,y]_\G$,
\item[$KV_2$] $(x,y,z)=(y,x,z)$,
\end{enumerate}
where, $(x,y,z)$ is the associator
\begin{align*}
(x,y,z)=(x\cdot y)\cdot z-x\cdot(y\cdot z).
\end{align*}
There is a one-to-one correspondence between the set $\mathrm{KV}(\G)$ of Koszul-Vinberg structures on $\G$ and the set of left-invariant affine structures on the Lie group $G$. Koszul-Vinberg structures are known under many different names, such as flat Lie algebras, pre-Lie algebras, and left-symmetric algebras (LSA for short). For more details on left-symmetric algebras, we refer the reader to the survey article~\cite{Burd2} and the references therein.

The cotangent extension problem described above is closely related to the problem of existence for left-invariant symplectic 2-forms on  Lie groups. The following are two extreme cases in which this question can be answered easily:
\begin{enumerate}
\item If the Lie group $G$ is commutative, the Lie group structure defined by $(\ref{loi})$ provides a solution to the cotangent extension problem for $G$.
\item The cotangent extension admits no solution for a semisimple Lie group $G$, as shown by an argument due to Y. Matsushima \cite{Matsu}. This argument relies on the vanishing of the first Lie algebra cohomology for finite-dimensional representations of semisimple Lie algebras.
\end{enumerate}

 The complete solution to the cotangent extension problem is provided by Baues and Cort\'es in~\cite{Bause}, see, for instance, Theorem~$\ref{corresp}$.

The discussion now proceeds to the concept of a \emph{flat affine Lie group}, or equivalently, a \emph{flat Lie algebra}. Consider a simply connected Lie group $ G $ with its associated Lie algebra $\G$.  A \emph{connection} on a Lie algebra $\G$ is a bilinear map $\nabla : \G \times \G \to \G$, denoted $(x, y) \mapsto \nabla_x y$. Its \textit{torsion} $\mathrm{T}^\nabla$ and  $\mathrm{R}^\nabla$ are the bilinear and trilinear maps defined respectively for all $x, y, z \in \G$ by:
\begin{align*}
\mathrm{T}^\nabla(x, y) &= \nabla_x y - \nabla_y x - [x, y]_\G, \\
\mathrm{R}^\nabla(x, y)z &= \nabla_x \nabla_y z - \nabla_y \nabla_x z - \nabla_{[x, y]_\G} z.
\end{align*}
The connection $\nabla$ is said to be \textit{torsion-free} if $\mathrm{T}^\nabla = 0$. The curvature $\mathrm{R}^\nabla$ vanishes identically if and only if the induced map $\rho^\nabla : \G \to \mathrm{End}(\G)$, given by $x \mapsto \nabla_x$, is a representation of $\G$ on itself. In this case, $\nabla$ is called a \textit{flat connection}. A flat connection $\nabla$ is torsion-free if and only if the identity map $\mathrm{Id} : \G \to \G$ is a 1-cocycle in the cohomology space $Z^1_{\rho^\nabla}(\G, \G)$; that is, $\mathrm{Id} \in Z^1_{\rho^\nabla}(\G, \G)$. A \textit{flat Lie algebra} is a pair $(\G, \nabla)$ consisting of a Lie algebra $\G$ equipped with a connection $\nabla$ that is both torsion-free and flat.

It is known that every perfect Lie algebra $\G$ over a field $\mathbb{K}$ (i.e., one satisfying $[\G, \G] = \G$) admits no flat, torsion-free connection~(\cite{Hel}, p.31). 
In contrast, symplectic Lie algebras are naturally equipped with a flat, torsion-free connection, which makes them examples of flat Lie algebras. It is known~\cite{Chu,DM,MR} that for a symplectic Lie algebra $(\G, \omega)$, the bilinear map $\nabla$ defined implicitly by the formula
\begin{equation}\label{sympconnection}
    \omega(\nabla^\omega_x y, z) = -\omega(y, [x, z]) \quad \textit{for all }\, x, y, z \in \G
\end{equation}
yields a flat, torsion-free connection on $\G$.

Let $(\G, \nabla)$ be a flat Lie algebra. Set $\nabla_x y = x \cdot y$ for all $x, y \in \G$. Let $\varrho_x$ denote the right-multiplication operator by $x$ in the Koszul-Vinberg algebra $(\G, \cdot)$. We will finish this introduction with some results on the completeness of left-invariant flat affine connections. It is well-known that a left-invariant flat affine connection $\nabla$ is geodesically complete if and only if $\tr(\varrho_x) = 0$ for all $x \in \G$ (see~\cite{Hel}). For the proofs of the following theorems, we refer to~\cite{Hosh}. 
\begin{theo}(Scheuneman)\label{theo1} Let $(\h,\nabla)$ be a flat Lie algebra. Then,
 if $(\h,\nabla)$ is a complete flat Lie algebra whose associated
Lie algebra is nilpotent, then the endomorphism $\nabla_x$ is nilpotent for all $x\in\h$.
\end{theo}

\begin{theo}\label{theo2} Let $(\h,\nabla)$ be a flat Lie algebra.  If $\nabla_x$ is nilpotent for all $x\in\h$, then its associated Lie algebra is nilpotent and the right
multiplication $\varrho_x$ is nilpotent for all $x\in\h$.
\end{theo}

This paper follows the following structure. In Section~$\ref{sec2}$, we introduce the fundamental properties of symplectic Lagrangian reduction and derive some preliminary results. Section~\ref{sec3} investigates symplectic nilpotent Lie algebras with a Lagrangian ideal. We establish a one-to-one correspondence between such algebras and geodesically complete  flat Lie algebras characterized by the Lagrangian extension cohomology group (Corollary~\ref{one-to-one}). Furthermore, we provide a complete classification of eight-dimensional symplectic nilpotent Lie algebras that admit a Lagrangian ideal (Proposition~\ref{Symbracket}, Table~$\ref{Symbracket}$), including their symplectic forms (Proposition~\ref{Symform}, Table~$\ref{Symform}$). As a consequence, we also obtain a complete classification of eight-dimensional symplectic filiform real Lie algebras (Theorem~$\ref{Filifrom}$, Table~$\ref{SympFilifrom}$).\\\\
\textbf{Notation and conventions.} For our results on symplectic Lie algebras, we work over a fixed field $\mathbb{K}$ of characteristic zero, unless stated otherwise. Global geometric interpretations for simply connected Lie groups are naturally given over the field $\mathbb{K} = \R$ of real numbers, which is therefore of principal interest in our investigations. Let $\G$ be a Lie algebra with basis $\{e_i\}_{i=1}^n$, and let $\{e^i\}_{i=1}^n$ denote the corresponding dual basis in $\G^\ast$. We denote by $e^{ij}$ the $2$-form $e^i \wedge e^j \in \bigwedge^2 \G^\ast$. By $[\alpha]$ (respectively, $[\alpha]_{L}$), we denote the cohomology class of $\alpha \in Z^2_\rho(\h, \h^\ast)$ in the group \(H^2_\rho(\h, \h^\ast)$ (respectively, in the Lagrangian extension cohomology group $H^2_{L,\rho}(\h, \h^\ast)$).

\section{Lagrangian extension of flat   Lie algebras}\label{sec2}
In this section, we summarize key definitions and results concerning the Lagrangian extension of flat Lie algebras, as established in \cite{Bause}. We then present an adaptation of these results to the nilpotent case, which leads to our first main result.

At first, we define symplectic Lie algebras.
\begin{Def}
A symplectic Lie algebra $(\G, \omega)$ is a Lie algebra $\G$ endowed with a non-degenerate, skew-symmetric bilinear form $\omega : \G \times \G \to \R$ that is closed. This closure condition means $\mathrm{d}\omega = 0$, where $\mathrm{d}$ denotes the Chevalley-Eilenberg differential. Explicitly, for all $x, y, z \in \G$, the following equation holds$:$
\begin{align*}
\mathrm{d}\omega(x, y, z) = \omega(x, [y, z]) + \omega(y, [z, x]) + \omega(z, [x, y]) = 0.
\end{align*}
\end{Def}
 A symplectic Lie algebra is one-to-one with a simply connected Lie group with left-invariant forms. An isomorphism between two symplectic Lie algebras $(\G_1, \omega_1)$ and $(\G_2, \omega_2)$ is a Lie algebra isomorphism $\varphi : \G_1 \to \G_2$ that preserves the symplectic forms, i.e., $\varphi^*\omega_2 = \omega_1$. For low-dimensional symplectic Lie algebras, classifications have been established through several approaches: for the general case \cite{Co1,Ova}; for the nilpotent case \cite{Burd1,Goze,KGM} (see also \cite{Fis}); and for the nonsolvable case-which includes exact symplectic Lie algebras \cite{AM,Co2}.

\subsection{Symplectic reduction}
Given a symplectic Lie algebra $(\G, \omega)$, an ideal $\mathfrak{j}$ of $\G$ is called  \textit{isotropic} if it is an isotropic subspace with respect to the symplectic form, i.e., $\omega(x, y) = 0$ for all $x, y \in \mathfrak{j}$. If, in addition to being isotropic, the symplectic orthogonal $\mathfrak{j}^{\perp_\omega}$ is also an ideal in $\G$, then $\mathfrak{j}$ is called a \textit{normal isotropic} ideal. Lastly, an isotropic ideal is said to be \textit{Lagrangian} if it is a maximum isotropic subspace.

Let $(\G, \omega)$ be a symplectic Lie algebra with an isotropic ideal $\mathfrak{j}$. The symplectic orthogonal $\mathfrak{j}^{\perp_\omega}$ is a subalgebra containing $\mathfrak{j}$, and consequently, $\omega$ induces a symplectic form $\overline{\omega}$ on the quotient Lie algebra
\begin{align*}
\overline{\G}=\mathfrak{j}^{\perp_\omega}/\mathfrak{j}.
\end{align*}
The symplectic Lie algebra $(\overline{\G},\overline{\omega})$ is defined as the \textit{symplectic reduction} of $(\G, \omega)$ with respect to the isotropic ideal $\mathfrak{j}$. An isotropic ideal $\mathfrak{j}$ is termed normal if $\mathfrak{j}^{\perp_\omega}$ is an ideal in $\G$. In this case, the symplectic reduction is also called normal. Furthermore, this reduction is known as a \textit{Lagrangian reduction} if the ideal $\mathfrak{j}$ is Lagrangian.

Let $\mathfrak{j}$ be a normal  ideal in the symplectic Lie algebra $(\G,\omega)$ and let  
\begin{align*}
\h=\G/\mathfrak{j}^{\perp_\omega}
\end{align*}
be the induced quotient Lie algebra. The symplectic form $\omega$ induces a non-degenerate bilinear pairing $\omega_{\h}$ between $\h$ and $\mathfrak{j}$, defined by
\begin{align}
\omega_{\h}(\bar{x}, u) = \omega(x, u), \quad \text{for all } \bar{x} \in \h, u \in \mathfrak{j},
\end{align}
where for any $x \in \G$, we denote its equivalence class in $\h$ by $\bar{x}$. We now recall the following key result:
\begin{pr}\textsc{\cite{Bause}}\label{Flatindeuced}
Let $\mathfrak{j}$ be a normal ideal in $(\G,\omega)$. Then the following hold:
\begin{enumerate}
\item  The homomorphism $\omega_\h \in\mathrm{Hom}(\h,\mathfrak{j}^\ast)$, $\overline{x} \mapsto\omega_\h(\overline{x},\cdot)$, is an isomorphism $\h \longrightarrow
\mathfrak{j}^\ast$.
\item The Lie algebra $\h=\G/\mathfrak{j}^{\perp_\omega}$ carries a torsion-free flat connection $\overline{\nabla}=\nabla^\omega$ which
is defined by the equation
\begin{align}
\omega_\h\big(\overline{\nabla}_{\overline{x}}\overline{y},u\big)=-\omega\big(y,[x,u]_\G\big),\quad\text{for all } \overline{x},\overline{y}\in\h,~~u\in\mathfrak{j}.
\end{align}
\end{enumerate}
\end{pr}

\subsection{Lagrangian extension of flat  Lie algebras}
  In this section, we briefly outline the theory of Lagrangian extensions of flat Lie algebras and their relation to Lagrangian reduction \cite{Bause}. On the infinitesimal level, Lagrangian extensions generalize the structure of symplectic cotangent Lie groups.

 Let $(\h, \nabla)$ be a flat Lie algebra; that is, a Lie algebra $\h$ equipped with a flat, torsion-free connection $\nabla$.  Given that $\nabla$ is a flat connection, the map $x \mapsto \nabla_x$ defines a representation $\h \to \mathrm{End}(\h)$. We denote by $\rho : \h \to \mathrm{End}(\h^\ast)$ the dual representation, which satisfies
\begin{align}\label{dualrepre}
\rho(x)\xi:=-\nabla_x^\ast\xi=-\xi\circ\nabla_x,\quad x\in\h,~~\xi\in\h^\ast.
\end{align}
Define $Z^2_\nabla(\h,\h^\ast)=Z^2_\rho(\h,\h^\ast)$. Each such cocycle $\alpha\in Z^2_\rho(\h,\h^\ast)$ induces a Lie algebra extension
\begin{align*}
0\longrightarrow\h^\ast\longrightarrow\G_{\nabla,\alpha}\longrightarrow\h\longrightarrow 0,
\end{align*}
where the Lie algebra structure on the the vector space $\mathfrak{h}\oplus \mathfrak{h}^*$ is given by the following relations:
\begin{align}
[x,y]&=[x,y]_\h+\alpha(x,y),&\text{for all } x,y\in\h,\label{Bra1}\\
[x,\xi]&=\rho(x)\xi,&\text{for all } x,y\in\h,~~\xi\in\h^\ast.\label{Bra2}
\end{align}
Let $\omega$ be the non-degenerate, alternating two-form on $\G$ defined by the dual pairing of $\h$ and $\h^\ast$, which implies that both $\h$ and $\h^\ast$ are $\omega$-isotropic. Explicitly, for all $x \in \h$ and $\xi \in \h^\ast$, we have $\omega(\xi, x) = -\omega(x, \xi) = \xi(x)$.  Imposing the requirement that $\omega$ be closed on $\G_{\nabla,\alpha}$ leads to the following verifiable condition:
\begin{pr}\textsc{\cite{Bause}}
The form $\omega$ is symplectic for the Lie-algebra $\G_{\nabla,\alpha}$ if and only
if
\begin{align}
\mathop{\resizebox{1.3\width}{!}{$\sum$}}\limits_{\mathrm{cycl}}\alpha(x,y)(z)&=0\quad(\text{Bianchi identity})
\end{align}
for all $x,y,z\in\h$
\end{pr}

We therefore call the symplectic Lie algebra $(\G_{\nabla,\alpha}, \omega)$ the \textit{Lagrangian extension} of the flat Lie algebra $(\h, \nabla)$ by the cocycle $\alpha\in Z^2_\rho(\h,\h^\ast)$. Observe that every flat Lie algebra $(\h, \nabla)$ admits a canonical Lagrangian extension, namely the one defined by the trivial cocycle $\alpha \equiv 0$. This yields the \textit{semi-direct product extension} $\h \oplus_{\nabla} \h^\ast$ with the symplectic form $\omega$ induced by the dual pairing. This special case is also called the \textit{cotangent Lagrangian extension}. It shows that every flat Lie algebra is a quotient of a Lagrangian reduction.

Let $\rho: \h\longrightarrow \mathrm{End}(\h^\ast)$ be the representation of $\h$ on its dual space $\h^\ast$ induced by the flat, torsion-free connection $\nabla$. The elements of the spaces $C^1(\h, \h^\ast) = \mathrm{Hom}(\h, \h^\ast)$ and $C^2(\h, \h^\ast) = \mathrm{Hom}(\bigwedge^2 \h, \h^\ast)$ are called, respectively, \textit{symplectic $1$-cochains} and \textit{symplectic $2$-cochains} of $\h$ with coefficients in $(\h^\ast, \rho)$. The \textit{coboundary operators}
\begin{align*}
\partial^i_{\rho} : C^i(\h, \h^\ast) \to C^{i+1}(\h, \h^\ast)
\end{align*}
are defined by the usual formula in Lie algebra cohomology for the representation $\rho$. We primarily need the coboundary operator $\partial^i_\rho$ in low degrees. Its explicit action is given by the following formulas for all $x, y, z \in \h$:
\begin{align}
(\partial_\rho^1\sigma)(x,y)&=\rho(x)\sigma(y)-\rho(y)\sigma(x)-\sigma\big([x,y]_\h\big),\label{rho1}\\
(\partial_\rho^2\sigma)(x,y,z)&=\rho(x)\alpha(y, z) + \rho(y)\alpha(z, x) + \rho(z)\alpha(x, y)
+ \alpha(x, [y, z]) + \alpha(z, [x, y]) + \alpha(y, [z, x]).\label{rho2}
\end{align}
We now define the symplectic extension cohomology group for the flat Lie algebra $(\h, \nabla)$ as
\begin{align*}
H^2_\rho(\h,\h^\ast)=\frac{Z^2_\rho(\h,\h^\ast)}{\partial_\rho C^1(\h,\h^\ast)}.
\end{align*}

To classify Lagrangian extensions of flat Lie algebras, we need to determine when two extensions are equivalent. Based on Corollary $4.3.3$ in \cite{Bause}, we have the following:
\begin{pr}\label{cohooro}
Let $(\h, \nabla)$ be a flat Lie algebra. Two Lagrangian extensions $(\G_{\nabla,\alpha}, \omega)$ and $(\G_{\nabla,\beta}, \omega)$ over $\h$ are isomorphic as  Lie algebras if and only if the cocycles $\alpha$ and $\beta$ are cohomologous in $H^2_{\rho}(\h, \h^\ast)$, i.e., $[\alpha] = [\beta]$.
\end{pr}
\begin{proof}
Let $(\h, \nabla)$ be a flat Lie algebra. Consider two Lagrangian extensions $(\G_{\nabla,\alpha}, \omega)$ and $(\G_{\nabla,\beta}, \omega)$ over $\h$ corresponding to cocycles $\alpha, \beta \in Z^2_\rho(\h, \h^\ast)$. If $\alpha$ and $\beta$ are cohomologous, i.e., $\beta=\alpha- \partial_\nabla \sigma$ for some $\sigma \in C^1(\h, \h^\ast)$, then $[\alpha] = [\beta]$ in $H^2_\rho(\h, \h^\ast)$. We now define the map $\Psi: \G_{\nabla,\alpha} \rightarrow\G_{\nabla,\beta}$ of Lie algebras. Since both algebras have the underlying vector space $\h \oplus_\rho \h^\ast$, we define $\Psi$ as follows:
\begin{align}
\Psi:\G_{\nabla,\alpha}\longrightarrow\G_{\nabla,\beta},~~(x,\xi)\mapsto(x,\xi+\sigma(x))
\end{align}
where $\sigma \in C^1(\h, \h^\ast)$ is the cochain such that $ \beta  = \alpha-\partial_\nabla \sigma$.  Thus, the map $\Psi$ is the required isomorphism of  Lie algebras. Indeed, we denote the Lie brackets on $\G_{\nabla,\alpha}$ and $\G_{\nabla,\beta}$ by $[\cdot, \cdot]_\alpha$ and $[\cdot, \cdot]_\beta$, respectively. Then, for all $x,y\in\h$, we have
\begin{align*}
\Psi\big([x,y]_\alpha\big)&=
\Psi\big([x,y]_\h+\alpha(x,y)\big)\\
&=[x,y]_\h+\alpha(x,y)+\sigma\big([x,y]_\h\big),
\end{align*}
and
\begin{align*}
[\Psi(x),\Psi(y)]_\beta&=[x+\sigma(x),y+\sigma(y)]_\beta
\\
&=[x,y]_\h+\beta(x,y)+\rho(x)\sigma(y)-\rho(y)\sigma(x).
\end{align*}
On the other hand, the equality $\Psi\big([x, \xi]_{\alpha}\big) = [\Psi(x), \Psi(\xi)]_{\beta}$ for all $x \in \h$, $\xi \in \h^\ast$ follows directly from the definition of the bracket and the map $\Psi$. Consequently, The assumption that $\alpha$ and $\beta$ are cohomologous, i.e., $\beta=\alpha  - \partial_\rho^1 \sigma$, implies that the map $\Psi$ preserves the Lie bracket.

\end{proof}

We now establish a characterization for the nilpotency of a Lagrangian extension $(\G_{\nabla,\alpha}, \omega)$. Specifically, we will show that $\G_{\nabla,\alpha}$ is nilpotent if and only if the base algebra $\h$ is nilpotent and the cocycle $\alpha$ together with the connection $\nabla$ satisfy a certain compatibility condition.

\begin{pr}
Let $(\h, \nabla)$ be an $n$-dimensional flat Lie algebra. The Lagrangian extension of $(\h, \nabla)$ is a nilpotent Lie algebra if and only if  the following conditions hold$:$
\begin{enumerate}
\item $\h$ is nilpotent;
    \item $(\h,\nabla)$ is geodesically complete;
    \item $\mathop{\resizebox{1.3\width}{!}{$\sum^{p-1}$}}\limits_{j=0} \rho^{j}(x)\alpha\big(x, \mathrm{ad}_x^{p-1-j} (y)\big) = 0$;
\end{enumerate}
for all $x, y \in \h$ and for some integer $p \in \mathbb{N}^\ast$.
\end{pr}
\begin{proof}
Let $(\G_{\nabla,\alpha}, \omega)$ be the Lagrangian extension of the flat Lie algebra $(\mathfrak{h}, \nabla)$ with respect to a 2-cocycle $\alpha \in Z^2_\rho(\h, \h^\ast)$. Suppose $\G_{\nabla,\alpha}$ is nilpotent; that is, there exists a positive integer $p \in \mathbb{N}^\ast$ such that $\mathcal{C}^p(\G_{\nabla,\alpha}) = 0$. Then, for all $x \in \h$ and $\xi \in \h^\ast$, we have:
\begin{equation*}
    \mathrm{ad}_x \xi = -\xi \circ \nabla_x,
\end{equation*}
and
\begin{equation*}
    \mathrm{ad}_x^2 \xi = \xi \circ \nabla_x \circ \nabla_x = \xi \circ \nabla_x^2.
\end{equation*}
By induction, we obtain
\begin{equation}\label{complete}
    \mathrm{ad}_x^k \xi = (-1)^k \xi \circ \nabla_x^k \quad \text{for all } k \in \mathbb{N}^\ast.
\end{equation}
This implies that the endomorphism  $\nabla_x\colon \h \to \h$ is nilpotent. On the other hand, for all $x,y\in\h$, we have
\begin{align*}
\mathrm{ad}_xy=\mathrm{ad}_{|_\h x}y+\alpha(x,y),
\end{align*}
and
\begin{align*}
\mathrm{ad}^2_xy=\mathrm{ad}_{|_\h x}^2y+\alpha(x,\mathrm{ad}_{|_\h x}y)+\rho(x)\alpha(x,y),
\end{align*}
moreover,
\begin{align*}
\mathrm{ad}^3_xy=\mathrm{ad}_{|_\h x}^3y+\alpha(x,\mathrm{ad}^2_{|_\h x}y)+\rho(x)\alpha(x,\mathrm{ad}_{|_\h x}y)+\rho(x)^2\alpha(x,y).
\end{align*}
By induction, we obtain
\begin{align*}
\mathrm{ad}^k_xy&=\mathrm{ad}_{|_\h x}^ky+\mathop{\resizebox{1.3\width}{!}{$\sum^{k-1}$}}\limits_{j=0}\rho(x)^j\alpha\big(x, \mathrm{ad}_{|_\h x}^{k-1-j} (y)\big),\quad\text{for all } x\in\h,~k\in\mathbb{N}^\ast.
\end{align*}
Therefore, since $\G_{\nabla,\alpha}$ is nilpotent, there exists an integer $p \in \mathbb{N}^\ast$ such that
\begin{align}\label{nilpotent}
\mathrm{ad}_{|\h x}^p = 0 \quad \text{and} \quad \mathop{\resizebox{1.3\width}{!}{$\sum^{p-1}$}}\limits_{j=0} \rho(x)^j \alpha\big(x, \mathrm{ad}_{|_\h x}^{p-1-j}(y)\big) = 0
\end{align}
for all $x, y \in \h$. This shows that $\h$ is a nilpotent Lie algebra. Consequently, the combination of conditions $(\ref{complete})$ and $(\ref{nilpotent})$ yields a necessary and sufficient condition for the nilpotency of the Lagrangian extension, as stated in the proposition.
\end{proof}
\subsection{Lagrangian extension cohomology group }
We briefly recall the construction of Lagrangian extension cohomology group of a flat Lie algebra and its main consequences \cite{Bause}. We now construct, for any flat Lie algebra $(\h, \nabla)$, a cohomology group that describes all Lagrangian extensions of $\h$.

We begin by defining Lagrangian $1$-cochains and $2$-cochains on $\h$.
\begin{align*}
C_L^1(\h,\h^\ast)&=\big\{\mu\in C^1(\h,\h^\ast)|~~\mu(x)(y)-\mu(y)(x)=0\big\},\\
C_L^2(\h,\h^\ast)&=\big\{\alpha\in C^2(\h,\h^\ast)|~~\mathop{\resizebox{1.3\width}{!}{$\sum$}}\limits_{\mathrm{cycl}}\alpha(x,y)(z)=0\big\}.
\end{align*}
Furthermore, let $\rho = \rho^{\nabla}$ be the representation of $\h$ on $\h^\ast$ induced by $\nabla$, as given in \eqref{dualrepre}. We denote the corresponding coboundary operators for the cohomology with coefficients in $\rho$ by $\partial_{\rho}=\partial_{\rho}^i$. Using Relation~(\ref{rho1}), it is straightforward to show that for all $\mu\in C^1_L(\h,\h^\ast)$, we have $\partial_\rho\mu\in C^2_L(\h,\h^\ast)$. This means that the coboundary operator $\partial_\rho : C^1(\h, \h^\ast) \longrightarrow C^2(\h, \h^\ast)$ maps the subspace $C_L^1(\h, \h^\ast)$ into $C_L^2(\h, \h^\ast) \cap Z_\rho^2(\h, \h^\ast)$.

Let $Z^2_{L,\rho}(\h, \h^\ast) = C^2_L(\h, \h^\ast) \cap Z^2_{\rho}(\h, \h^\ast)$ denote the space of Lagrangian 2-cocycles. We define the associated Lagrangian extension cohomology group of the flat Lie algebra $(\h, \nabla)$ by
\begin{align*}
H^2_{L,\rho}(\h, \h^\ast)=\frac{Z^2_{L,\rho}(\h, \h^\ast)}{\partial_\rho C^1_L(\h, \h^\ast)}.
\end{align*}
The construction of the Lagrangian extension cohomology group yields a natural map
\[
H^2_{L,\rho}(\mathfrak{h}, \h^\ast) \to H^2_{\rho}(\h, \h^\ast)
\]
to the ordinary second cohomology group. It is important to note that this map is generally not injective.

The following theorem establishes that the isomorphism classes of Lagrangian extensions of a flat Lie algebra are in one-to-one correspondence with the cohomology group $H^2_{L,\rho}(\h, \h^\ast)$.

\begin{theo}\textsc{\cite{Bause}}\label{corresp}
The correspondence which associates to a symplectic Lie algebra with Lagrangian ideal $(\G, \omega, \mathfrak{j})$ the extension triple $(\h, \nabla, [\alpha]_L)$ induces a bijection
between isomorphism classes of symplectic Lie algebras with Lagrangian ideal and
isomorphism classes of flat Lie algebras with Lagrangian extension cohomology class.

\end{theo}
\begin{remark}\label{remarksymp}
Two Lagrangian extensions $(\G_{\nabla,\alpha_1}, \omega_1)$ and $(\G_{\nabla,\alpha_2}, \omega_2)$ of the flat Lie algebra $(\h, \nabla)$ are symplectically isomorphic if and only if they define the same cohomology class in $H^2_{L,\rho}(\h, \h^\ast)$ $($see Theorem $4.4.4$ \textsc{\cite{Bause}}$)$.
\end{remark}

The isomorphism classes of symplectic forms with a Lagrangian ideal on Lagrangian extensions of $(\h, \nabla)$ are given by the following result. Let $(\h, \nabla)$ be a flat Lie algebra. For any cohomology class $[\alpha] \in H^2_{\rho}(\h, \h^\ast)$, one constructs the associated Lagrangian extension $(\G_{\nabla, [\alpha]}, \omega)$ of $(\h, \nabla)$. Set $\overline{\alpha} = \alpha - \partial_\rho \sigma$ and $\widehat{\alpha} = \alpha - \partial_\rho \sigma_L$ for some $\sigma \in C^1(\h, \h^\ast)$ and some $\sigma_L \in C^1_L(\h, \h^\ast)$. We therefore have:

\begin{pr}\label{isosymp}
Let $(\mathfrak{h}, \nabla)$ be a flat Lie algebra. For any Lagrangian extension $(\G_{\nabla,\overline{\alpha}},\omega)$ of $(\h, \nabla)$, the isomorphism class of the symplectic form $\omega_{[\alpha]_L}$ with Lagrangian ideal  is uniquely determined by the cohomology class $[\alpha]_L \in H^2_{L,\rho}(\h, \h^\ast)$. Moreover, the symplectic form can be expressed as:
\begin{align*}
\omega_{[\alpha]_L}(x + \xi, y) &=\Lambda(x,y)+\omega (x,\xi) \quad \text{for all } x,y \in \h,\ \xi\in \h^\ast,
\end{align*}
where $\Lambda = (\sigma - \sigma_L) - {}^t(\sigma - \sigma_L)\in\mathrm{Hom}(\bigwedge^2\h,\R)$, and where $\overline{\alpha} = \alpha - \partial_\rho \sigma$ and $\widehat{\alpha} = \alpha - \partial_\rho \sigma_L$ for some $\sigma \in C^1(\h, \h^\ast)$ and some $\sigma_L \in C^1_L(\h, \h^\ast)$.
\end{pr}
\begin{proof}
Let $(\G_{\nabla,\alpha}, \omega)$ be the Lagrangian extension of the flat Lie algebra $(\h, \nabla)$. To obtain the isomorphism class of Lagrangian extensions over $\h$, we set $\overline{\alpha} = \alpha - \partial_\rho \sigma$ for some $\sigma \in C^1(\h, \h^\ast)$ (see Proposition~\ref{cohooro}). According to Remark~\ref{remarksymp}, it suffices to calculate the cohomology class of $\alpha$ in the Lagrangian extension cohomology group $H^2_{L,\rho}(\h, \h^\ast)$. To this end, let $\widehat{\alpha} = \alpha - \partial_\rho \sigma_L$ for some $\sigma_L \in C^1_L(\h, \h^\ast)$. Using $(\ref{Bra1})$ and $(\ref{Bra2})$, we will show that the map
\begin{align*}
\Psi:(\G_{\nabla,\overline{\alpha}},\omega_{[\alpha]_L})\longrightarrow(\G_{\nabla,\widehat{\alpha}},\omega),~~(x,\xi)\mapsto\big(x,\xi-(\sigma-\sigma_L)(x)\big)\end{align*}
is the required Lie algebra isomorphism. Indeed, for all $x,y\in\h$, we have
\begin{align*}
\Psi\big([x,y]_{\overline{\alpha}}\big)&=\Psi\big([x,y]_\h+\overline{\alpha}(x,y)\big)\\
&=[x,y]_\h+\overline{\alpha}(x,y)-(\sigma-\sigma_L)([x,y]_\h),
\end{align*} 
and
\begin{align*}
[\Psi(x),\Psi(y)]_{\widehat{\alpha}}&=[x-(\sigma-\sigma_L)(x),y-(\sigma-\sigma_L)(y)]_{\widehat{\alpha}}\\
&=[x,y]_\h+\widehat{\alpha}(x,y)-\rho(x)(\sigma-\sigma_L)(y)+\rho(y)(\sigma-\sigma_L)(x).
\end{align*}
Therefore,
\begin{align*}
\Psi\big([x,y]_{\overline{\alpha}}\big)-[\Psi(x),\Psi(y)]_{\widehat{\alpha}}&=\overline{\alpha}(x,y)-\widehat{\alpha}(x,y)+\rho(x)(\sigma-\sigma_L)(y)-\rho(y)(\sigma-\sigma_L)(x)-(\sigma-\sigma_L)([x,y]_\h)\\
&=\big(\overline{\alpha}+\partial_\rho\sigma\big)(x,y)-\big(\widehat{\alpha}+\partial_\rho\sigma_L\big)(x,y)\\
&=\alpha-\alpha=0.
\end{align*}
Moreover, for all $x,y\in\h$, $\xi\in\h^\ast$, we have
\begin{align*}
\omega_{[\alpha]_L}(x,y+\xi)&=(\Psi^\ast\omega)(x,y+\xi)=\omega\big(\Psi(x),\Psi(y+\xi)\big)\\
&=\omega\big(x+(\sigma-\sigma_L)(x),y+\xi+(\sigma-\sigma_L)(y)\big)\\
&=\omega(x,\xi)+\omega((\sigma-\sigma_L)(x),y)+\omega(x,(\sigma-\sigma_L)(y))\\
&=\omega(x,\xi)+\omega((\sigma-\sigma_L)(x),y)+\omega((\sigma-\sigma_L)^\ast(x),y)\\
&=\omega(x,\xi)+(\sigma-\sigma_L)(x)(y)-{}^t(\sigma-\sigma_L)(y)(x)\\
&=\Lambda(x,y)+\omega(x,\xi),
\end{align*}
where, $\Lambda=(\sigma-\sigma_L)-{}^t(\sigma-\sigma_L)\in\mathrm{Hom}(\bigwedge^2\h,\R)$.
\end{proof}

\subsection{Four-dimensional geodesically complete flat   nilpotent real Lie algebras}

It is well known that, up to isomorphism, there exist three 4-dimensional nilpotent Lie algebras (see, for example, \cite{chao} or \cite{dix}). We adopt the following notation for them:
\begin{enumerate}
\item $\mathfrak{a}\cong\R^4=\langle e_1,e_2,e_3,e_4|-\rangle$;\quad abelian
\item $\ell=\langle e_1,e_2,e_3,e_4|~~[e_1,e_2]=e_3\rangle$;\quad $Heisenberg\oplus\R$,
\item $\mathfrak{t}=\langle e_1,e_2,e_3,e_4|~~[e_1,e_4]=-e_2,~~[e_2,e_4]=-e_3\rangle$;
\end{enumerate}
\begin{remark}\label{Diffalgebra}
The classification of 4-dimensional geodesically complete flat real Lie algebras for nilpotent Lie algebras was established in \textsc{\cite{Hosh}} using an extension theory of flat Lie algebras. Since the original classification was given modulo isomorphism (with different but isomorphic algebras appearing as distinct cases), we present here a refined classification where isomorphic algebras are properly identified, yielding a unique representative for each isomorphism class.

Recall the following definition:

\begin{Def}\label{Defofnabla}
	Two flat Lie algebras $(\h_1, \nabla^1)$ and $(\h_2, \nabla^2)$ are said to be isomorphic if and only if there exists a Lie algebra isomorphism 
	\[
	\Psi: \h_1 \longrightarrow \h_2
	\] 
	such that 
	\begin{align}\label{relationequiva}
		\nabla_x^2 = \Psi \circ \nabla^1_{\Psi^{-1}(x)} \circ \Psi^{-1}, \quad \text{for all } x \in \h_2.
	\end{align}
\end{Def}

Based on this definition, and since we are in low dimension $(\dim \h = 4)$, we observe that the flat, torsion-free connections labeled $\mathbf{15}$, $\mathbf{23}$, $\mathbf{26}$, and $\mathbf{25}$ in \textsc{\cite{Hosh}} are, in fact, isomorphic to those represented by $\mathbf{13}$, $\mathbf{22}$, $\mathbf{24}$, and $\mathbf{24}$, respectively. 
The remaining connections are not isomorphic under the equivalence relation~\eqref{relationequiva}.
\end{remark}

\begin{pr}\label{4flat}
Every four-dimensional geodesically complete flat nilpotent Lie algebra is isomorphic to one of those listed in Table~$\ref{4flatcomplete}$.

\begin{center}
{\renewcommand*{\arraystretch}{1.5}
\captionof{table}{Four-dimensional geodesically complete flat nilpotent Lie algebras.}\label{4flatcomplete}
\setcounter{table}{0}
\begin{small} % Plus petit que small
\setlength{\tabcolsep}{0.5pt} % RÃ©duire l'espace entre colonnes
\begin{longtable}{@{}cllllllc@{}} % @{} rÃ©duit les marges
			\hline
		Algebra&&&&&&Remarks\\
			\hline
$\mathfrak{a}_1$&$\nabla_{e_2}e_2=e_1$&$\nabla_{e_3}e_3=e_1$&$\nabla_{e_4}e_4=e_1$&&&\\
$\mathfrak{a}_2$&$\nabla_{e_2}e_2=e_1$&$\nabla_{e_3}e_3=e_1$&$\nabla_{e_4}e_4=-e_1$&&&\\
$\mathfrak{a}_3$&$\nabla_{e_2}e_4=e_1$&$\nabla_{e_3}e_3=e_1$&$\nabla_{e_3}e_4=e_2$&$\nabla_{e_4}e_2=e_1$&$\nabla_{e_4}e_3=e_2$&\\
&$\nabla_{e_4}e_4=e_3$&&&&&\\
$\mathfrak{a}_4$&$\nabla_{e_3}e_3=e_1$&$\nabla_{e_4}e_4=e_1$&&&&\\
$\mathfrak{a}_5$&$\nabla_{e_3}e_3=e_1$&$\nabla_{e_4}e_4=-e_1$&&&&\\
$\mathfrak{a}_6$&$\nabla_{e_3}e_4=e_2$&$\nabla_{e_4}e_3=e_2$&$\nabla_{e_4}e_4=e_1$&&&\\
$\mathfrak{a}_7$&$\nabla_{e_3}e_3=e_1$&$\nabla_{e_4}e_4=e_2$&&&&&\\
$\mathfrak{a}_8$&$\nabla_{e_3}e_3=e_1$&$\nabla_{e_3}e_4=e_2$&$\nabla_{e_4}e_3=e_2$&&&\\
$\mathfrak{a}_9$&$\nabla_{e_3}e_4=e_1$&$\nabla_{e_4}e_3=e_1$&$\nabla_{e_4}e_4=e_3$&&&\\
$\mathfrak{a}_{10}$&$\nabla_{e_4}e_4=e_1$&&&&&\\\hline
$\ell_1$&$\nabla_{e_1}e_1=e_3$&$\nabla_{e_2}e_1=-e_3$&$\nabla_{e_2}e_4=e_3$&$\nabla_{e_4}e_2=e_3$&&\\
$\ell_2$&$\nabla_{e_2}e_1=-e_3$&$\nabla_{e_2}e_4=e_3$&$\nabla_{e_4}e_2=e_3$&&&\\
$\ell_3$&$\nabla_{e_1}e_1=e_3$&$\nabla_{e_2}e_1=-e_3$&$\nabla_{e_2}e_2=\frac{t+1}{4}e_3$&$\nabla_{e_4}e_4=e_3$&&$t\in\R^{+}$\\
$\ell_4$&$\nabla_{e_1}e_1=e_3$&$\nabla_{e_2}e_1=-e_3$&$\nabla_{e_2}e_2=te_3$&$\nabla_{e_4}e_4=-e_3$&&$t\in\R^{+}$\\
$\ell_5$&$\nabla_{e_1}e_2=e_3$&$\nabla_{e_2}e_2=e_1$&$\nabla_{e_4}e_4=e_1$&&&\\
$\ell_6$&$\nabla_{e_1}e_2=e_3$&$\nabla_{e_2}e_2=e_1$&$\nabla_{e_4}e_4=e_1+e_3$&&&\\
$\ell_7$&$\nabla_{e_1}e_2=e_3$&$\nabla_{e_2}e_2=e_1$&$\nabla_{e_2}e_4=e_3$&$\nabla_{e_4}e_2=e_3$&$\nabla_{e_4}e_4=e_1+te_3$&$t\in\R^{+}$\\
$\ell_8$&$\nabla_{e_1}e_2=e_3$&$\nabla_{e_2}e_2=e_1$&$\nabla_{e_4}e_4=-e_1$&&&\\
$\ell_9$&$\nabla_{e_1}e_2=e_3$&$\nabla_{e_2}e_2=e_1$&$\nabla_{e_2}e_4=e_1$&$\nabla_{e_4}e_2=e_1$&&\\
$\ell_{10}$&$\nabla_{e_1}e_2=e_3$&$\nabla_{e_2}e_2=e_1$&$\nabla_{e_2}e_4=e_1$&$\nabla_{e_4}e_2=e_1$&$\nabla_{e_4}e_4=-e_3$&\\
$\ell_{11}$&$\nabla_{e_1}e_2=e_3$&$\nabla_{e_2}e_2=e_1$&$\nabla_{e_4}e_4=-e_1+e_3$&&&\\
$\ell_{12}$&$\nabla_{e_1}e_2=e_3$&$\nabla_{e_2}e_2=e_1$&$\nabla_{e_2}e_4=e_3$&$\nabla_{e_4}e_2=e_3$&$\nabla_{e_4}e_4=-e_1+te_3$&$t\in\R$\\
$\ell_{13}$&$\nabla_{e_1}e_2=e_3$&$\nabla_{e_4}e_4=e_1$&&&&\\
$\ell_{14}$&$\nabla_{e_1}e_2=e_3$&$\nabla_{e_2}e_4=e_3$&$\nabla_{e_4}e_2=e_3$&$\nabla_{e_4}e_4=e_1$&&\\
$\ell_{15}$&$\nabla_{e_1}e_2=e_3$&$\nabla_{e_2}e_2=e_1$&$\nabla_{e_4}e_4=e_3$&&&\\
$\ell_{16}$&$\nabla_{e_1}e_1=e_2$&$\nabla_{e_1}e_4=e_3$&$\nabla_{e_2}e_1=-e_3$&$\nabla_{e_4}e_1=e_3$&&\\
$\ell_{17}$&$\nabla_{e_1}e_1=e_2$&$\nabla_{e_1}e_2=e_3$&$\nabla_{e_1}e_4=te_2$&$\nabla_{e_2}e_4=te_3$&$\nabla_{e_4}e_1=te_2$&$t\in\R$\\
&$\nabla_{e_4}e_2=te_3$&$\nabla_{e_4}e_4=t^2e_2+e_3$&&&&\\
$\ell_{18}$&$\nabla_{e_1}e_4=te_3$&$\nabla_{e_2}e_1=-e_3$&$\nabla_{e_2}e_2=e_1$&$\nabla_{e_2}e_4=-te_1+e_3$&$\nabla_{e_4}e_1=te_3$&$t\in\R$\\
&$\nabla_{e_4}e_2=-te_1+e_3$&$\nabla_{e_4}e_4=t^2e_1-3te_3$&&&&\\
$\ell_{19}$&$\nabla_{e_1}e_1=e_2$&$\nabla_{e_1}e_3=e_4$&$\nabla_{e_2}e_1=-e_3$&$\nabla_{e_2}e_2=-2e_4$&$\nabla_{e_3}e_1=e_4$&\\
$\ell_{20}$&$\nabla_{e_2}e_1=-e_3$&$\nabla_{e_2}e_2=e_4$&$\nabla_{e_2}e_4=e_1$&$\nabla_{e_4}e_2=e_1$&$\nabla_{e_4}e_4=-e_3$&\\
$\ell_{21}$&$\nabla_{e_1}e_1=e_3$&$\nabla_{e_1}e_2=e_4$&$\nabla_{e_2}e_1=-e_3+e_4$&$\nabla_{e_2}e_2=e_1$&$\nabla_{e_2}e_4=e_3$&\\
&$\nabla_{e_4}e_2=e_3$&&&&&\\
$\ell_{22}$&$\nabla_{e_1}e_1=e_4$&$\nabla_{e_1}e_2=\mu e_3$&$\nabla_{e_1}e_4=e_2+(1-\mu)e_3$&$\nabla_{e_2}e_1=(\mu-1)e_3$&$\nabla_{e_4}e_1=e_2+(1-\mu)e_3$&$\mu\in\R$\\
&$\nabla_{e_4}e_4=\mu e_3$&&&&&\\
$\ell_{23}$&$\nabla_{e_1}e_1=e_2$&$\nabla_{e_1}e_2=\frac{\mu}{\mu-1}e_3$&$\nabla_{e_1}e_3=(\mu-1)e_4$&$\nabla_{e_2}e_1=\frac{1}{\mu-1}e_3$&$\nabla_{e_2}e_2=(2-\mu)e_4$&$\mu\in\R$,~$\mu\neq1$\\
&$\nabla_{e_3}e_1=(\mu-1)e_4$&&&&&\\
$\ell_{24}$&$\nabla_{e_1}e_1=e_2$&$\nabla_{e_1}e_2=e_3$&$\nabla_{e_1}e_3=e_4$&$\nabla_{e_2}e_2=-e_4$&$\nabla_{e_3}e_1=e_4$&\\
$\ell_{25}$&$\nabla_{e_1}e_1=e_2$&$\nabla_{e_1}e_2=e_3+e_4$&$\nabla_{e_1}e_3=e_4$&$\nabla_{e_2}e_1=e_4$&$\nabla_{e_2}e_2=-e_4$&\\
&$\nabla_{e_3}e_1=e_4$&&&&&\\
$\ell_{26}$&$\nabla_{e_1}e_2=\frac{1}{2}e_3$&$\nabla_{e_2}e_1=-\frac{1}{2}e_3$&&&&\\
$\ell_{27}$&$\nabla_{e_1}e_1=e_3$&$\nabla_{e_1}e_2=\frac{1}{2}e_3$&$\nabla_{e_2}e_1=-\frac{1}{2}e_3$&&&\\
$\ell_{28}$&$\nabla_{e_1}e_2=\frac{1}{2}e_3$&$\nabla_{e_2}e_1=-\frac{1}{2}e_3$&$\nabla_{e_2}e_2=e_4$&&&\\
$\ell_{29}$&$\nabla_{e_1}e_1=e_4$&$\nabla_{e_1}e_2=\frac{1}{2}e_3$&$\nabla_{e_2}e_2=-\frac{1}{2}e_3$&$\nabla_{e_2}e_2=e_4$&&\\
$\ell_{30}$&$\nabla_{e_1}e_1=-e_4$&$\nabla_{e_1}e_2=\frac{1}{2}e_3$&$\nabla_{e_2}e_2=-\frac{1}{2}e_3$&$\nabla_{e_2}e_2=e_4$&&\\
$\ell_{31}$&$\nabla_{e_1}e_1=\mu e_3$&$\nabla_{e_1}e_2=\frac{1}{2}e_3$&$\nabla_{e_2}e_1=-\frac{1}{2}e_3$&$\nabla_{e_2}e_2=e_3$&&$\mu\in\R^{\ast+}$&\\
$\ell_{33}$&$\nabla_{e_1}e_1=\mu e_3$&$\nabla_{e_1}e_2=\frac{1}{2}e_3$&$\nabla_{e_2}e_1=-\frac{1}{2}e_3$&$\nabla_{e_2}e_2=-e_3$&&$\mu\in\R^{\ast+}$\\
$\ell_{34}$&$\nabla_{e_1}e_1=e_4$&$\nabla_{e_1}e_2=\mu e_3$&$\nabla_{e_2}e_1=(\mu-1)e_3$&$\nabla_{e_2}e_2=e_4$&&$\mu>\frac{1}{2}$\\
$\ell_{35}$&$\nabla_{e_1}e_1=e_4$&$\nabla_{e_1}e_2=\frac{1}{2}e_3$&$\nabla_{e_2}e_1=-\frac{1}{2}e_3$&$\nabla_{e_2}e_2=\frac{1}{2}e_3-te_4$&&$t\in\R^{\ast+}$\\
$\ell_{36}$&$\nabla_{e_1}e_2=e_4$&$\nabla_{e_2}e_1=-e_3+e_4$&$\nabla_{e_2}e_2=e_3$&&&\\
$\ell_{37}$&$\nabla_{e_1}e_2=\mu e_3$&$\nabla_{e_2}e_1=(\mu-1)e_3$&$\nabla_{e_2}e_2=e_4$&&&$\mu<\frac{1}{2}$\\
$\ell_{38}$&$\nabla_{e_1}e_1=e_2$&$\nabla_{e_1}e_2=e_3$&&&&\\
$\ell_{39}$&$\nabla_{e_1}e_1=e_2$&$\nabla_{e_1}e_2=e_4$&$\nabla_{e_2}e_1=-e_3+e_4$&&&\\
$\ell_{40}$&$\nabla_{e_1}e_1=e_2$&$\nabla_{e_1}e_2=\mu e_3$&$\nabla_{e_2}e_1=(\mu-1)e_3$&&&$\mu\in\R$\\\hline
$\mathfrak{t}_1$&$\nabla_{e_1}e_3=e_4$&$\nabla_{e_2}e_2=-e_4$&$\nabla_{e_3}e_1=e_4$&$\nabla_{e_4}e_1=e_2$&$\nabla_{e_4}e_2=e_3$&\\
$\mathfrak{t}_2$&$\nabla_{e_1}e_1=-e_3$&$\nabla_{e_1}e_3=e_4$&$\nabla_{e_2}e_2=-e_4$&$\nabla_{e_3}e_1=e_4$&$\nabla_{e_4}e_1=e_2$&&\\
&$\nabla_{e_4}e_2=e_3$&&&&&\\
$\T_3$&$\nabla_{e_1}e_1=(\mu+9)e_2$&$\nabla_{e_1}e_2=4e_3$&$\nabla_{e_1}e_4=-2e_2$&$\nabla_{e_2}e_1=4e_3$&$\nabla_{e_2}e_4=-e_3$&$\mu\in\R$\\
&$\nabla_{e_4}e_1=-e_2$&$\nabla_{e_4}e_4=\frac{1}{4}e_2$&&&&\\
$\T_4$&$\nabla_{e_1}e_1=(\mu+9)e_2+e_3$&$\nabla_{e_1}e_2=4e_3$&$\nabla_{e_1}e_4=-2e_2$&$\nabla_{e_2}e_1=4e_3$&$\nabla_{e_2}e_4=-e_3$&$\mu\in\R$\\
&$\nabla_{e_4}e_1=-e_2$&$\nabla_{e_4}e_4=\frac{1}{4}e_2$&&&&\\
$\T_5$&$\nabla_{e_1}e_1=(\mu+9)e_2+\mu_1e_3$&$\nabla_{e_1}e_2=4e_3$&$\nabla_{e_1}e_4=-2e_2$&$\nabla_{e_2}e_1=4e_3$&$\nabla_{e_2}e_4=-e_3$&$\mu,\mu_1\in\R$\\
&$\nabla_{e_4}e_1=-e_2$&$\nabla_{e_4}e_4=\frac{1}{4}e_2+e_3$&&&&\\
$\T_6$&$\nabla_{e_1}e_4=-\frac{1}{2}e_2$&$\nabla_{e_2}e_4=-\frac{1}{2}e_3$&$\nabla_{e_4}e_1=\frac{1}{2}e_2$&$\nabla_{e_4}e_2=\frac{2}{3}e_3$&&\\
$\T_7$&$\nabla_{e_1}e_1=e_3$&$\nabla_{e_1}e_4=-\frac{1}{2}e_2$&$\nabla_{e_2}e_4=-\frac{1}{2}e_3$&$\nabla_{e_4}e_1=\frac{1}{2}e_2$&$\nabla_{e_4}e_2=\frac{2}{3}e_3$&\\
$\T_8$&$\nabla_{e_1}e_1=e_4$&$\nabla_{e_4}e_1=e_2$&$\nabla_{e_4}e_2=e_3$&&&\\
$\T_9$&$\nabla_{e_1}e_1=e_4$&$\nabla_{e_1}e_4=e_3$&$\nabla_{e_4}e_1=e_2+e_3$&$\nabla_{e_4}e_2=e_3$&&\\
$\T_{10}$&$\nabla_{e_1}e_1=e_4$&$\nabla_{e_1}e_4=-e_3$&$\nabla_{e_4}e_1=e_2-e_3$&$\nabla_{e_4}e_2=e_3$&&\\
$\T_{11}$&$\nabla_{e_1}e_4=-e_2$&$\nabla_{e_2}e_4=-\frac{1}{2}e_3$&$\nabla_{e_4}e_2=\frac{1}{2}e_3$&$\nabla_{e_4}e_4=e_1$&&\\
$\T_{12}$&$\nabla_{e_1}e_1=\frac{1}{3}e_3$&$\nabla_{e_1}e_4=-e_2+\frac{1}{3}e_3$&$\nabla_{e_2}e_4=-\frac{2}{3}e_3$&$\nabla_{e_4}e_1=\frac{1}{3}e_3$&$\nabla_{e_4}e_2=\frac{1}{3}e_3$&\\
&$\nabla_{e_4}e_4=e_1$&&&&&\\
$\T_{13}$&$\nabla_{e_1}e_1=\mu e_3$&$\nabla_{e_1}e_4=-e_2$&$\nabla_{e_2}e_4=-\frac{\mu+1}{2}e_3$&$\nabla_{e_4}e_2=\frac{1-\mu}{2}e_3$&$\nabla_{e_4}e_4=e_1$&$\mu\in\R^\ast$\\
$\T_{14}$&$\nabla_{e_1}e_1=e_4$&$\nabla_{e_1}e_2=e_3$&$\nabla_{e_1}e_4=(\mu+2)e_3$&$\nabla_{e_2}e_1=e_3$&$\nabla_{e_4}e_1=e_2+(\mu+2)e_3$&$\mu\in\R$\\
&$\nabla_{e_4}e_2=e_3$&$\nabla_{e_4}e_4=2e_3$&&&&\\
$\T_{15}$&$\nabla_{e_1}e_1=\mu e_3$&$\nabla_{e_1}e_4=\mu e_2$&$\nabla_{e_4}e_1=(\mu+1)e_2$&$\nabla_{e_4}e_2=e_3$&$\nabla_{e_4}e_4=e_1$&$\mu\in\R^\ast$\\
$\T_{16}$&$\nabla_{e_1}e_1=\frac{\mu}{3}(2\mu+1)e_3$&$\nabla_{e_1}e_4=\mu e_2-\frac{2\mu^3}{3}e_3$&$\nabla_{e_2}e_4=\frac{2\mu}{3}e_3$&$\nabla_{e_4}e_1=(\mu+1)e_2-\frac{2\mu^3}{3}e_3$&$\nabla_{e_4}e_2=(\frac{2\mu}{3}+1)e_3$&$\mu\in\R^\ast$\\
&$\nabla_{e_4}e_4=e_1$&&&&&\\
$\T_{17}$&$\nabla_{e_1}e_1=\mu_1e_3$&$\nabla_{e_1}e_4=\mu e_2$&$\nabla_{e_4}e_2=\frac{\mu_1-\mu}{\mu-1}e_3$&$\nabla_{e_4}e_1=(\mu+1)e_2$&$\nabla_{e_4}e_2=\frac{\mu_1-1}{\mu-1}e_3$&$\mu,\mu_1\in\R$\\
&$\nabla_{e_4}e_4=e_1$&$\mu\neq0,1$,~$\mu_1\neq\frac{\mu(2\mu+1)}{3}$&&&&\\
$\T_{18}$&$\nabla_{e_4}e_1=e_2$&$\nabla_{e_4}e_2=e_3$&$\nabla_{e_4}e_4=e_1$&&&&\\
$\T_{19}$&$\nabla_{e_1}e_4=e_3$&$\nabla_{e_4}e_1=e_2+e_3$&$\nabla_{e_4}e_2=e_3$&$\nabla_{e_4}e_4=e_1$&&&\\
$\T_{20}$&$\nabla_{e_1}e_1=\mu e_3$&$\nabla_{e_2}e_3=-\mu e_3$&$\nabla_{e_4}e_1=e_2$&$\nabla_{e_4}e_2=(1-\mu)e_3$&$\nabla_{e_4}e_4=e_1$&$\mu\in\R^\ast$&\\
$\T_{21}$&$\nabla_{e_1}e_1=\mu e_3$&$\nabla_{e_1}e_4=-\mu e_3$&$\nabla_{e_2}e_4=-\mu e_3$&$\nabla_{e_4}e_1=e_2-\mu e_3$&$\nabla_{e_4}e_2=(1-\mu)e_3$&$\mu\in\R^\ast$&\\
&$\nabla_{e_4}e_4=e_1$&&&&&&
\\\hline
\end{longtable}
			\end{small}	
			}	
			\end{center}
\end{pr}
\begin{proof}
As noted in Remark~\ref{Diffalgebra}, the classification of four-dimensional geodesically complete flat real Lie algebras associated with nilpotent Lie algebras was obtained in \textsc{\cite{Hosh}} using the extension theory of flat Lie algebras. However, the original classification was formulated up to isomorphism, meaning that distinct cases may represent isomorphic Lie algebras. In this proof, we therefore present a refined classification in which isomorphic algebras are explicitly identified, resulting in a unique representative for each isomorphism class.

First, since the abelian Lie algebra has vanishing brackets, we fix its associated flat torsion-free connection and denote it by $(\mathfrak{a}_j,\nabla)$, instead of $\mathrm{A}$, which is the notation used in \textsc{\cite{Hosh}}.

Next, we consider the second Lie algebra, namely the Heisenberg Lie algebra. Its flat torsion-free connections are presented in \textsc{\cite{Hosh}} under different Lie algebra structures, all of which are isomorphic to the Heisenberg Lie algebra. For this reason, we fix the Heisenberg Lie algebra, denoted by $\ell$, and focus on the first case, numbered $5$ in \textsc{\cite{Hosh}}. In the basis $\{e_1,e_2,e_3,e_4\}$, the associated flat torsion-free connection is given by
\begin{align}
	\nabla_{e_2}e_2&=e_1,\quad\nabla_{e_2}e_3=e_1,\quad\nabla_{e_3}e_2=-e_1,\quad\nabla_{e_3}e_4=e_1,\quad\nabla_{e_4}e_3=e_1.
\end{align}
	The underlying Lie algebra structure associated with $\nabla$ is given by
	\begin{align}
		\mathrm{H}_5:&~~[e_2,e_3]=2\,e_1.
	\end{align}
	We observe that the algebra $\mathrm{H}_5$ is in fact isomorphic to the Heisenberg algebra $\ell$, whose only nonvanishing Lie bracket is given by $[e_1,e_2]=e_3$. We define the map
	\[
	\Psi : 	\mathrm{H}_5 \longrightarrow \ell
	\]
	by
	\[
	e_1 \mapsto e_3, \quad e_2 \mapsto e_1+\tfrac{1}{2}e_4, \quad e_3 \mapsto 2\,e_2, \quad e_4 \mapsto \tfrac{1}{2}e_4.
	\]
	It is straightforward to verify that this map defines an isomorphism of Lie algebras. 
	Applying Definition~\ref{Defofnabla}, we then recover the flat torsion-free connection corresponding to the Heisenberg Lie algebra $\ell$. 
	Finally, we obtain
		\begin{align}\label{H5}
	\nabla_{e_1}e_1&=e_3,\quad\nabla_{e_2}e_1=-e_3,\quad\nabla_{e_2}e_4=e_3,\quad\nabla_{e_4}e_2=e_3.
	\end{align}
This flat torsion-free connection is listed in Table~\ref{4flatcomplete} under the algebra $\ell_1$.	We now consider the second example, numbered $8_t$ in \textsc{\cite{Hosh}}.
	The corresponding flat torsion-free connection is given by \begin{align}
		\nabla_{e_2}e_2&=e_1,\quad\nabla_{e_2}e_3=e_1,\quad\nabla_{e_3}e_3=t\,e_1,\quad\nabla_{e_4}e_4=-e_1,\quad t\in\R^+.
	\end{align}
	and the underlying Lie algebra associated with this connection is determined by the following nonvanishing Lie brackets: 
	\begin{align}
			\mathrm{H}_{8_t}:~~[e_2,e_3]=e_1.
	\end{align}
	As mentioned previously, this algebra is in fact isomorphic to the Heisenberg Lie algebra $\ell$. 
	One easily verifies that the map
	\[
	\Psi : 	\mathrm{H}_{8_t} \longrightarrow \ell
	\]
	defined by
	\[
	e_1 \mapsto e_3, \quad e_2 \mapsto e_1, \quad e_3 \mapsto e_1+e_2, \quad e_4 \mapsto e_4
	\]
	is a Lie algebra isomorphism. It therefore remains to recover the corresponding flat torsion-free connection on $\ell$ by applying Definition~\ref{Defofnabla}. We obtain
	\begin{align}\label{H8t}
		\nabla_{e_1}e_1&=e_3,\quad\nabla_{e_2}e_1=-e_3,\quad\nabla_{e_2}e_2=te_3,\quad\nabla_{e_4}e_4=-e_3,\quad t\in\R^+.
	\end{align}
	The flat torsion-free connection corresponding to the algebra $\ell_4$ is given in Table~\ref{4flatcomplete}. 	Observe that the Lie brackets defining the Lie algebra $	\mathrm{H}_5$ differ from those of $	\mathrm{H}_{8_t}$. 
	For this reason, we identify both Lie algebras with the Heisenberg Lie algebra $\ell$. 
	In order to compare these structures and determine whether they are isomorphic as flat Lie algebras, 
	we make use of the automorphism group of $\ell$.
	
	Using the automorphisms given in Appendix~6, one can easily verify that the two connections defined in (\ref{H5}) and (\ref{H8t}) are not isomorphic. The remaining cases are handled similarly, and the appendix provides a complete list of all flat Lie algebras that are isomorphic, as referenced above Definition~\ref{Defofnabla}.
	
\end{proof}

\section{Symplectic  nilpotent Lie algebras with Lagrangian ideal}\label{sec3}
In the first part of this section  we investigate symplectic nilpotent Lie algebras endowed with a Lagrangian ideal. 
    As proved by Baues and Cort\'es (Theorem~7.2.1, \cite{Bause}), every filiform 
    symplectic Lie algebra is a Lagrangian extension of a certain flat filiform Lie algebra 
    $(\mathfrak{h}, \nabla)$. Here, we precisely determine the type of these flat 
    Lie algebras in a more general setting.

\begin{pr}
Let $(\G, \omega)$ be a   symplectic  nilpotent Lie algebra of dimension $2n$, and let $\mathfrak{j}$ be a Lagrangian ideal of $(\G, \omega)$. Then, the quotient Lie algebra $\h=\G/\mathfrak{j}$ is a geodesically complete flat   nilpotent Lie algebra.
\end{pr}
\begin{proof}
Let $(\G, \omega)$ be a symplectic nilpotent Lie algebra. Suppose that $\mathfrak{j}$ is a Lagrangian ideal of $(\G, \omega)$. Then, the quotient Lie algebra $\h = \G/\mathfrak{j}$ is nilpotent and, moreover, flat. Recall that by Proposition~\ref{Flatindeuced}, the induced flat torsion-free connection $\nabla$ on $\h$ satisfies
\begin{align}
    \omega_{\h}\big(\nabla_x y, u\big) = -\omega\big(\tilde{y}, [\tilde{x}, u]\big), \quad \text{for all } x, y \in \h, ~u \in \mathfrak{j},
\end{align}
where $\tilde{x}, \tilde{y} \in \G$ denote lifts of $x, y \in \h$, and where $\omega_\h$ is a non-degenerate bilinear pairing between $\h$ and $\mathfrak{j}$, defined by
\begin{align*}
\omega_\h\big(x, u\big) = \omega(\tilde{x}, u) \quad \text{for all } x \in \h, \, u \in \mathfrak{j}.
\end{align*}
We denote by $``\cdot"$ the product associated with the flat, torsion-free connection $\nabla^\omega$  (defined by formula~$(\ref{sympconnection})$), i.e., $\nabla^\omega_{a}b= a \cdot b$ for all $a, b \in \G$. Set $\nabla_x y = x y$. First, observe that for all $x, y \in \h$ and $u \in \mathfrak{j}$, we have
\begin{align*}
\omega\big(\tilde{x}\cdot\tilde{y},u\big)&=-\omega\big(\tilde{y},[\tilde{x},u]\big)=\omega_\h
\big(\nabla_xy,u\big),
\end{align*}
and 
\begin{align*}
\omega\big(\widetilde{xy},u\big)&
=\omega_\h\big(xy,u\big)=-\omega\big(\tilde{y},[\tilde{x},u]\big)=\omega\big(\tilde{x}\cdot\tilde{y},u\big).
\end{align*}
Then,
\begin{align*}
\omega_\h
\big(\nabla^2_xy,u\big)&=\omega_\h
\big(\nabla_x(xy),u\big)\\
&=-\omega\big(\widetilde{xy},[\tilde{x},u]\big)\\
&=-\omega\big(\tilde{x}\cdot\tilde{y},[\tilde{x},u]\big)\\
&=\omega\big(\tilde{y},\mathrm{ad}_{\tilde{x}}^2u\big).
\end{align*}
By induction, we obtain
\begin{align}
\omega_\h
\big(\nabla^k_xy,u\big)&=(-1)^k\omega\big(\tilde{y},\mathrm{ad}_{\tilde{x}}^ku\big),\quad \text{for all } x,y\in\h,~u\in\mathfrak{j},~k\in
\mathbb{N}^\ast.
\end{align}
Since $\G$ is nilpotent, there exists an integer $p \in \mathbb{N}^\ast$ such that $\mathrm{ad}_x^p = 0$ for all $x \in \G$. This implies that
\begin{align}
\omega_{\h} \big( \nabla^p_x y, u \big) = 0 \quad \text{for all } x, y \in \h, \; u \in \mathfrak{j}.
\end{align}
Since $\omega_{\h}$ is nondegenerate, it follows that $\nabla^p_x = 0$ for all $x \in \h$. According to Theorem~\ref{theo2}, given that the endomorphism $\nabla_x \colon \h \to \h$ is nilpotent, 
it follows that $\h$ is nilpotent and the right-multiplication $\varrho_x \colon \h \to \h$, 
$y \mapsto yx$, is nilpotent. Consequently, the flat Lie algebra $(\h, \nabla)$ is geodesically complete.

\end{proof}

Note that any flat quotient $(\h, \nabla)$ of a nilpotent Lie algebra $\G$ is   geodesically complete.
In view of the correspondence theorem~$\ref{corresp}$, the previous proposition shows
that  symplectic nilpotent Lie algebras with Lagrangian ideal arise as Lagrangian extensions of certain flat Lie
algebras $(\h,\nabla)$, where $\h$ is nilpotent. Let $(\h, \nabla, \alpha)$ be a Lagrangian extension triple, where $\h$ is nilpotent and $\nabla$ is geodesically complete. We call $(\h, \nabla, \alpha)$ a \textit{geodesically complete triple} if, in addition to these conditions, the extension $[\h, \nabla, \alpha]$ is geodesically complete (ignoring here the specific requirements on $\alpha$). This leads to the following important consequence:

\begin{co}\label{one-to-one}
The correspondence which associates to a  symplectic nilpotent Lie algebra the extension triple $(\h, \nabla, [\alpha])$,  induces a bijection between isomorphism classes of
 symplectic nilpotent Lie algebras  with Lagrangian ideal and
isomorphism classes of geodesically complete flat nilpotent Lie algebras with Lagrangian extension cohomology class.

\end{co}

Using Proposition~$\ref{cohooro}$, we obtain a complete classification of the eight-dimensional nilpotent Lie algebras that admit a Lagrangian ideal. We refer to $\cite{AEM}$ for an illustration of the calculation in low dimensions.

\begin{pr}\label{prSymbracket}
Let $(\G_{\nabla,\alpha},\omega)$ be the lagrangian extension of a four-dimensional geodesically complete   flat nilpotent Lie algebra $(\h,\nabla)$ with respect to $\alpha\in Z^2_\rho(\h,\h^\ast)$. Then,  $(\G,\omega)$ is one of the following Lie algebras$:$
\newgeometry{left=0.2cm,right=0.2cm,top=3cm,bottom=2.8cm}

{\renewcommand*{\arraystretch}{1.5}
\captionof{table}{Lagrangian extension of geodesically complete Flat nilpotent Lie algebras.}\label{Symbracket}
\setcounter{table}{1}
\begin{small} % Plus petit que small
\setlength{\tabcolsep}{2pt} % RÃ©duire l'espace entre colonnes
\begin{longtable}{@{}cllllllc@{}} % @{} rÃ©duit les marges
			\hline
		Algebra&&&\\
			\hline
$\G_{8,1}$&$[e_1,e_2]=ae_7+be_8$&$[e_1,e_3]=ae_6+(c-d)e_7+\lambda_1 e_8$&$[e_1,e_4]=be_6+\lambda_1 e_7+(c-\lambda_2)e_8$\\&$[e_2,e_3]=(\lambda_3-\lambda_4)e_8$&$[e_2,e_4]=-\lambda_4 e_7+\eta_1 e_8$&$[e_3,e_4]=-\lambda_3 e_6+\eta_2 e_7+\eta_3 e_8$\\
&$[e_2,e_5]=-e_6$&$[e_3,e_5]=-e_7$&$[e_4,e_5]=-e_8$\\
&$a,b,c,d\in\R$, $\lambda_j,\eta_j\in\R$&$\mathcal{C}^2(\G)=0$&\\
$\G_{8,2}$&$[e_1,e_2]=ae_7+be_8$&$[e_1,e_3]=ae_6+(c-d)e_7+\lambda_1 e_8$&$[e_1,e_4]=be_6+\lambda_1 e_7-(c+\lambda_2)e_8$\\
&$[e_2,e_3]=(\lambda_3-\lambda_4)e_8$
&$[e_2,e_4]=-\lambda_4 e_7+\eta_1 e_8$&$[e_3,e_4]=-\lambda_3 e_6+\eta_2 e_7+\eta_3 e_8$\\
&$[e_2,e_5]=-e_6$&$[e_3,e_5]=-e_7$&$[e_4,e_5]=e_8$\\
&$a,b,c,d\in\R$, $\lambda_j,\eta_j\in\R$&$\mathcal{C}^2(\G)=0$&\\
$\G_{8,3}$&$[e_1,e_2]=ae_7$&$[e_1,e_3]=ae_6+be_7$&$[e_1,e_4]=ae_5+be_6+ce_7$\\
&$[e_2,e_3]=be_6$
&$[e_2,e_4]=be_5+ce_6$&$[e_3,e_4]=ce_5$\\
&$[e_2,e_5]=-e_8$&$[e_3,e_5]=-e_7$&$[e_3,e_6]=-e_8$\\
&$[e_4,e_5]=-e_6$&$[e_4,e_6]=-e_7$&$[e_4,e_8]=-e_8$\\
&$a,b,c\in\R$&$\mathcal{C}^5(\G)=0$&&\\
$\G_{8,4}$&$[e_1,e_2]=ae_6+be_7+ce_8$&$[e_1,e_3]=be_6+(\lambda_1-d)e_8$&$[e_1,e_4]=ce_6-de_7+\lambda_2e_8$\\
&$[e_2,e_3]=\lambda_3e_6+\lambda_4e_8$
&$[e_2,e_4]=\eta_1 e_6+(\eta_2+\lambda_4)e_7+\eta_3e_8$&$[e_3,e_4]=-\lambda_1e_5+\eta_2e_6$\\
&$[e_3,e_5]=-e_7$&$[e_4,e_5]=-e_8$&$a,b,c,d\in\R$,~$\lambda_j,\eta_j\in\R$\\
&$\mathcal{C}^3(\G)=0$&&\\
$\G_{8,5}$&$[e_1,e_2]=ae_6+be_7+ce_8$&$[e_1,e_3]=be_6+(\lambda_1-d)e_8$&$[e_1,e_4]=ce_6-de_7+\lambda_2e_8$\\
&$[e_2,e_3]=\lambda_3e_6+\lambda_4e_8$
&$[e_2,e_4]=\eta_1 e_6+(\eta_2+\lambda_4)e_7+\eta_3e_8$&$[e_3,e_4]=-\lambda_1e_5+\eta_2e_6$\\
&$[e_3,e_5]=-e_7$&$[e_4,e_5]=e_8$&$a,b,c,d\in\R$,~$\lambda_j,\eta_j\in\R$\\
&$\mathcal{C}^3(\G)=0$&&\\
$\G_{8,6}$&$[e_1,e_2]=ae_7+be_8$&$[e_1,e_3]=(b-c)e_5+de_7$&$[e_1,e_4]=\lambda_1e_5+(b-c)e_6+\lambda_2e_7$\\
&$[e_2,e_3]=-ae_5+\lambda_3e_7$
&$[e_2,e_4]=-ce_5-ae_6+\lambda_4e_7$&$[e_3,e_4]=\lambda_2e_5+\lambda_4e_6$\\
&$[e_3,e_6]=-e_8$&$[e_4,e_5]=-e_8$&$[e_4,e_6]=-e_7$\\
&$a,b,c,d\in\R$,~$\lambda_j\in\R$&$\mathcal{C}^3(\G)=0$\\
$\G_{8,7}$&$[e_1,e_2]=ae_7+be_8$&$[e_1,e_3]=ce_5+de_8$&$[e_1,e_4]=be_6+(d-\lambda_1)e_7$\\
&$[e_2,e_3]=-ae_5+\lambda_2e_8$
&$[e_2,e_4]=\lambda_3e_6+(\lambda_2-\lambda_4)e_7$&$[e_3,e_4]=-\lambda_1e_5-\lambda_4e_6$\\
&$[e_3,e_5]=-e_7$&$[e_4,e_6]=-e_8$&$a,b,c,d\in\R$,~$\lambda_j\in\R$\\
&$\mathcal{C}^3(\G)=0$&&\\
$\G_{8,8}$&$[e_1,e_2]=ae_7+be_8$&$[e_1,e_3]=ce_5+(a+d)e_6$&$[e_1,e_4]=(a+d)e_5+\lambda_1e_7+\lambda_2e_8$\\
&$[e_2,e_3]=de_5-be_6$
&$[e_2,e_4]=-be_5+\lambda_3e_7+\lambda_4e_8$&$[e_3,e_4]=\lambda_1e_5+\lambda_3e_6$\\
&$[e_3,e_5]=-e_7$&$[e_3,e_6]=-e_8$&$[e_4,e_6]=-e_7$\\
&$a,b,c,d\in\R$,~$\lambda_j\in\R$&$\mathcal{C}^3(\G)=0$&\\
$\G_{8,9}$&$[e_1,e_2]=ae_6+be_8$&$[e_1,e_3]=ce_7$&$[e_1,e_4]=ce_5+(b+d)e_6+\lambda_1e_7$\\
&$[e_2,e_3]=\lambda_2e_6+de_7$
&$[e_2,e_4]=de_5+\lambda_3e_6+\lambda_4e_7$&$[e_3,e_4]=\lambda_1e_5+\lambda_4e_6$\\
&$[e_3,e_5]=-e_8$&$[e_4,e_5]=-e_7$&$[e_4,e_7]=-e_8$\\
&$a,b,c,d\in\R$,~$\lambda_j\in\R$&$\mathcal{C}^4(\G)=0$&\\
$\G_{8,10}$&$[e_1,e_2]=ae_6+be_7+ce_8$&$[e_1,e_3]=be_6+de_7+\lambda_1e_8$&$[e_1,e_4]=\lambda_2e_5+(c+\lambda_3)e_6+(\lambda_1+\lambda_4)e_7$\\
&$[e_2,e_3]=\eta_1e_6+\eta_2e_7+\eta_3e_8$
&$[e_2,e_4]=\lambda_3e_5+\eta_4e_6+(\eta_3+\gamma_1)e_7$&$[e_3,e_4]=\lambda_4e_5+\gamma_1e_6+\gamma_2e_7$\\
&$[e_4,e_5]=-e_8$&$a,b,c,d\in\R$,~$\lambda_j,\gamma_j\in\R$&$\mathcal{C}^3(\G)=0$\\
$\G_{8,11}$&$[e_1,e_2]=e_3$&$[e_1,e_3]=(a+b)e_5+(c-a)e_6+de_8$&$[e_1,e_4]=(d+\lambda_1)e_7+\lambda_2e_8$\\
&$[e_2,e_3]=(c-a)e_5+\lambda_3e_6-(d+2\lambda_1)e_8$
&$[e_2,e_4]=\lambda_4 e_5-(d+2\lambda_1)e_7+\eta e_8$&$[e_3,e_4]=\lambda_1e_5+(d+\lambda_1)e_8$\\
&$[e_1,e_7]=-e_5$&$[e_2,e_7]=e_5-e_8$&$[e_4,e_7]=-e_6$\\
&$a,b,c,d,\eta\in\R$,~$\lambda_j\in\R$&$\mathcal{C}^3(\G)=0$&\\
$\G_{8,12}$&$[e_1,e_2]=e_3$&$[e_1,e_3]=ae_5+be_6+2ce_8$&$[e_1,e_4]=de_5+ce_7+\lambda_1e_8$\\
&$[e_2,e_3]=be_5+\lambda_2e_6+\lambda_3e_8$&
$[e_2,e_4]=\lambda_4e_5+\lambda_3e_7$&$[e_3,e_4]=-ce_5+ce_8$\\
&$[e_2,e_7]=e_5-e_8$&$[e_4,e_7]=-e_6$&$a,b,c,d\in\R$,~$\lambda_j\in\R$\\
&$\mathcal{C}^3(\G)=0$&&&&\\
$\G_{8,13}$&$[e_1,e_2]=e_3$&$[e_1,e_3]=ae_5+be_6+ce_8$&$[e_1,e_4]=de_6+(c+\lambda_1)e_7$\\
&$[e_2,e_3]=be_5+\lambda_2e_6-\frac{1}{4}\big((t+9)\lambda_1+(t+5)c\big)e_8$
&$[e_2,e_4]=\lambda_3e_5+\lambda_4e_6-(c+2\lambda_1)e_7$&$[e_3,e_4]=\lambda_1e_5+\frac{1}{4}(c+\lambda_1)(1+t)e_6$\\
&$[e_1,e_7]=-e_5$&$[e_2,e_7]=e_5-\frac{1}{4}(1+t)e_6$&$[e_4,e_7]=-e_8$\\
&$a,b,c,d\in\R$,~$\lambda_j\in\R$,~$t\in\R^{+}$&$\mathcal{C}^3(\G)=0$&\\
$\G_{8,14}$&$[e_1,e_2]=e_3$&$[e_1,e_3]=ae_5+be_6+ce_8$&$[e_1,e_4]=de_6+(c+\lambda_1)e_7$\\
&$[e_2,e_3]=be_5+\lambda_2e_6-\big((t+2)\lambda_1+(t+1)c\big)e_8$
&$[e_2,e_4]=\lambda_3e_5+\lambda_4e_6-(c+2\lambda_1)e_7$&$[e_3,e_4]=\lambda_1e_5+t(c+\lambda_1)e_6$\\
&$[e_1,e_7]=-e_5$&$[e_2,e_7]=e_5-te_6$&$[e_4,e_7]=e_8$\\
&$a,b,c,d\in\R$,~$\lambda_j\in\R$,~$t\in\R^\ast$&$\mathcal{C}^3(\G)=0$&\\
$\G_{8,15}$&$[e_1,e_2]=e_3$&$[e_1,e_3]=ae_7+be_8$&$[e_1,e_4]=ce_5+be_7$\\
&$[e_2,e_3]=de_6+(2\lambda_1-c)e_8$
&$[e_2,e_4]=\lambda_2e_5+\lambda_1e_7$&$[e_3,e_4]=(c-\lambda_1)e_6$\\
&$[e_1,e_7]=-e_6$&$[e_2,e_5]=-e_6$&$[e_4,e_5]=-e_8$\\
&$a,b,c,d\in\R$,~$\lambda_j\in\R$&$\mathcal{C}^4(\G)=0$&\\
$\G_{8,16}$&$[e_1,e_2]=e_3$&$[e_1,e_3]=a(e_5-e_6-e_7)+be_8$&$[e_1,e_4]=ce_5+be_7$\\
&$[e_2,e_3]=-ae_5+de_6+ae_7+(2\lambda_1-c)e_8$
&$[e_2,e_4]=\lambda_2e_5+\lambda_1e_7$&$[e_3,e_4]=(c-\lambda_1)e_6$\\
&$[e_1,e_7]=-e_6$&$[e_2,e_5]=-e_6$&$[e_4,e_5]=-e_8$\\
&$[e_4,e_7]=-e_8$&$a,b,c,d\in\R$,~$\lambda_j\in\R$&$\mathcal{C}^4(\G)=0$\\
$\G_{8,17}$&$[e_1,e_2]=e_3$&$[e_1,e_3]=ae_8$&$[e_1,e_4]=be_5+ae_7$\\
&$[e_2,e_3]=ce_6+(2d-b)e_8$
&$[e_2,e_4]=\lambda_1e_5+de_7$&$[e_3,e_4]=(b-d)e_6+ae_8$\\
&$[e_1,e_7]=-e_6$&$[e_2,e_5]=-e_6$&$[e_2,e_7]=-e_8$\\
&$[e_4,e_5]=-e_8$&$[e_4,e_7]=-e_6-te_8$&$a,b,c,d\in\R$,~$\lambda_j\in\R$,~$t\in\R^{+}$\\
&$\mathcal{C}^3(\G)=0$&&\\
$\G_{8,18}$&$[e_1,e_2]=e_3$&$[e_1,e_3]=ae_7+be_8$&$[e_1,e_4]=ce_5+be_7$\\
&$[e_2,e_3]=de_6+(2\lambda_1-c)e_8$
&$[e_2,e_4]=\lambda_2e_5+\lambda_1e_7$&$[e_3,e_4]=(c-\lambda_1)e_6$\\
&$[e_1,e_7]=-e_6$&$[e_2,e_5]=-e_6$&$[e_4,e_5]=e_8$\\
&$a,b,c,d\in\R$,~$\lambda_j\in\R$&$\mathcal{C}^4(\G)=0$&\\
$\G_{8,19}$&$[e_1,e_2]=e_3$&$[e_1,e_3]=ae_7+be_8$&$[e_1,e_4]=ce_5+be_7+de_8$\\
&$[e_2,e_3]=\lambda_1e_6+(2\lambda_2-c)e_8$
&$[e_2,e_4]=\lambda_3e_5+\lambda_2e_7$&$[e_3,e_4]=(c-\lambda_2)e_6+ce_8$\\
&$[e_1,e_7]=-e_6$&$[e_2,e_5]=-e_6-e_8$&$[e_4,e_5]=-e_6$&\\
&$a,b,c,d\in\R$,~$\lambda_j\in\R$&$\mathcal{C}^4(\G)=0$&\\
$\G_{8,20}$&$[e_1,e_2]=e_3$&$[e_1,e_3]=ae_8$&$[e_1,e_4]=be_5+ae_7$\\
&$[e_2,e_3]=ce_6+(2d-b)e_8$
&$[e_2,e_4]=\lambda e_5+de_7$&$[e_3,e_4]=(b-d)e_6+be_8$\\
&$[e_1,e_7]=-e_6$&$[e_2,e_5]=-e_6-e_8$&$[e_4,e_5]=-e_6$&\\
&$[e_4,e_7]=e_8$&$a,b,c,d,\lambda\in\R$&$\mathcal{C}^3(\G)=0$\\
$\G_{8,21}$&$[e_1,e_2]=e_3$&$[e_1,e_3]=a(e_5+e_6+e_7)+be_8$&$[e_1,e_4]=ce_5+be_7$\\
&$[e_2,e_3]=ae_5+de_6+ae_7+(2\lambda_1-c)e_8$
&$[e_2,e_4]=\lambda_2e_5+\lambda_1e_7$&$[e_3,e_4]=(c-\lambda_1)e_6$\\
&$[e_1,e_7]=-e_6$&$[e_2,e_5]=-e_6$&$[e_4,e_5]=e_8$\\
&$[e_4,e_6]=-e_8$&$a,b,c,d\in\R$,~$\lambda_j\in\R$&$\mathcal{C}^4(\G)=0$\\
$\G_{8,22}$&$[e_1,e_2]=e_3$&$[e_1,e_3]=ae_8$&$[e_1,e_4]=ce_5+ae_7$\\
&$[e_2,e_3]=de_6+(2\lambda_1-c)e_8$
&$[e_2,e_4]=\lambda_2e_5+\lambda_1e_7$&$[e_3,e_4]=(c-\lambda_1)e_6+ae_8$\\
&$[e_1,e_7]=-e_6$&$[e_2,e_5]=-e_6$&$[e_2,e_7]=-e_8$\\
&$[e_4,e_5]=e_8$&$[e_4,e_7]=-e_6-te_8$&$a,b,c,d\in\R$,~$\lambda_j\in\R$,~$t\in\R$\\
&$\mathcal{C}^3(\G)=0$&&\\
$\G_{8,23}$&$[e_1,e_2]=e_3$&$[e_1,e_3]=ae_7+be_8$&$[e_1,e_4]=ce_5+be_7$\\
&$[e_2,e_3]=de_6+2\lambda_1e_8$
&$[e_2,e_4]=\lambda_2e_5+\lambda_3e_6+\lambda_1e_7$&$[e_3,e_4]=-\lambda_1e_6$\\
&$[e_1,e_7]=-e_6$&$[e_4,e_5]=-e_8$&$a,b,c,d\in\R$,~$\lambda_j\in\R$\\
&$\mathcal{C}^4(\G)=0$&&\\
$\G_{8,24}$&$[e_1,e_2]=e_3$&$[e_1,e_3]=ae_8$&$[e_1,e_4]=be_5+ae_7$\\
&$[e_2,e_3]=ce_6+2de_8$
&$[e_2,e_4]=\lambda e_5+de_7$&$[e_3,e_4]=-de_6+ae_8$\\
&$[e_1,e_7]=-e_6$&$[e_2,e_7]=-e_8$&$[e_4,e_5]=-e_8$\\
&$[e_4,e_6]=-e_6$&$a,b,c,d,\lambda\in\R$&$\mathcal{C}^3(\G)=0$\\
$\G_{8,25}$&$[e_1,e_2]=e_3$&$[e_1,e_3]=ae_6+be_8$&$[e_1,e_4]=ce_5+be_7$\\
&$[e_2,e_3]=ae_5+de_6+(2\lambda_1-c)e_8$
&$[e_2,e_4]=\lambda_2e_5+\lambda_1e_7$&$[e_3,e_4]=(c-\lambda_1)e_6$\\
&$[e_1,e_7]=-e_6$&$[e_2,e_5]=-e_6$&$[e_4,e_7]=-e_8$\\
&$a,b,c,d\in\R$,~$\lambda_j\in\R$&$\mathcal{C}^4(\G)=0$&\\
$\G_{8,26}$&$[e_1,e_2]=e_3$&$[e_1,e_3]=ae_5+be_6+ce_8$&$[e_1,e_4]=de_6+\lambda_1e_7$\\
&$[e_2,e_3]=be_5+\lambda_2e_8$
&$[e_2,e_4]=(c-2\lambda_1)e_6+\lambda_2e_7+\lambda_3e_8$&$[e_3,e_4]=(\lambda_1-c)e_5-\lambda_2e_8$\\
&$[e_1,e_6]=-e_5$&$[e_1,e_7]=-e_8$&$[e_2,e_7]=e_5$\\
&$[e_4,e_7]=-e_5$&$a,b,c,d\in\R$,~$\lambda_j\in\R$&$\mathcal{C}^4(\G)=0$\\
$\G_{8,27}$&$[e_1,e_2]=e_3$&$[e_1,e_3]=ae_6+be_8$&$[e_1,e_4]=ce_7$\\
&$[e_2,e_3]=ae_5+2tae_8$
&$[e_2,e_4]=de_5+(b-c)e_6+tae_7+\lambda e_8$&$[e_3,e_4]=(c-b)e_5-tae_6+t(2c-b)e_8$\\
&$[e_1,e_6]=-e_5-te_8$&$[e_1,e_7]=-e_6$&$[e_2,e_7]=-te_8$\\
&$[e_4,e_6]=-te_5-t^2e_8$&$[e_4,e_7]=-te_6-e_8$&$a,b,c,d,\lambda,t\in\R$\\
&$\mathcal{C}^4(\G)=0$&&\\
$\G_{8,28}$&$[e_1,e_2]=e_3$&$[e_1,e_3]=ae_6-2tae_8$&$[e_1,e_4]=-tae_7$\\
&$[e_2,e_3]=ae_5+be_6+ce_8$
&$[e_2,e_4]=ce_7+de_8$&$[e_3,e_4]=tae_5-t(a+c)e_8$\\
&$[e_1,e_7]=-te_8$&$[e_2,e_5]=-e_6+te_8$&$[e_2,e_7]=e_5-e_8$\\
&$[e_4,e_5]=te_6-t^2e_8$&$[e_4,e_7]=-te_5-e_6+3te_8$&$a,b,c,d\in\R$,~$t\in\R^\ast$
\\
&$\mathcal{C}^4(\G)=0$&&\\
$\G_{8,29}$&$[e_1,e_2]=e_3$&$[e_1,e_3]=ae_6+2be_8$&$[e_1,e_4]=ce_5+be_7+de_8$\\
&$[e_2,e_3]=ae_5+\lambda_1e_6+be_7+(\lambda_2-c)e_8$
&$[e_2,e_4]=\lambda_2e_7+\lambda_3e_8$&$[e_3,e_4]=-be_5+ce_6+be_8$\\
&$[e_2,e_5]=-e_6$&$[e_2,e_7]=e_5-e_8$&$[e_4,e_7]=-e_6$\\
&$a,b,c,d\in\R$,~$\lambda_j\in\R$&$\mathcal{C}^5(\G)=0$\\
$\G_{8,30}$&$[e_1,e_2]=e_3$&$[e_1,e_3]=ae_6+be_7$&$[e_1,e_4]=ce_6$\\
&$[e_2,e_3]=ae_5-be_6+de_7$
&$[e_2,e_4]=de_8$&$[e_3,e_4]=-de_7$\\
&$[e_1,e_6]=-e_5$&$[e_1,e_8]=-e_7$&$[e_2,e_7]=e_5$\\
&$[e_2,e_8]=2e_6$&$[e_3,e_8]=-e_5$&$a,b,c,d\in\R$\\
&$\mathcal{C}^4(\G)=0$&&\\
$\G_{8,31}$&$[e_1,e_2]=e_3$&$[e_1,e_3]=ae_6$&$[e_1,e_4]=-ae_5$\\
&$[e_2,e_3]=ae_5+be_6+ce_8$
&$[e_2,e_4]=ce_7$&$[e_3,e_4]=-ae_8$\\
&$[e_2,e_5]=-e_8$&$[e_2,e_7]=e_5$&$[e_2,e_8]=-e_6$\\
&$[e_4,e_5]=-e_6$&$[e_4,e_7]=e_8$&$a,b,c\in\R$\\
&$\mathcal{C}^5(\G)=0$&&\\
$\G_{8,32}$&$[e_1,e_2]=e_3$&$[e_1,e_3]=ae_5+be_8$&$[e_1,e_4]=ce_8$\\
&$[e_2,e_3]=be_7+ae_8$&$[e_2,e_4]=de_5+(b+c)e_7+(a-b-c)e_8$&$[e_3,e_4]=-be_5+(-a+b+c)e_6$\\
&$[e_1,e_7]=-e_5$&$[e_1,e_8]=-e_6$&$[e_2,e_5]=-e_6$\\
&$[e_2,e_7]=e_5-e_8$&$[e_2,e_8]=-e_5$&$[e_4,e_7]=-e_6$\\
&$a,b,c,d\in\R$&$\mathcal{C}^6(\G)=0$&\\
$\G_{8,33}$&$[e_1,e_2]=e_3$&$[e_1,e_3]=ae_6$&$[e_1,e_4]=be_7$\\
&$[e_2,e_3]=ae_5$
&$[e_2,e_4]=ce_5+ae_6+b(\mu-2)e_8$&$[e_3,e_4]=be_5-ae_8$\\
&$[e_1,e_6]=-e_8$&$[e_1,e_7]=-\mu e_6+(\mu-1)e_8$&$[e_1,e_8]=-e_5$\\
&$[e_2,e_7]=(1-\mu)e_5$&$[e_4,e_6]=-e_5$&$[e_4,e_7]=(\mu-1)e_5-\mu e_8$
\\&$a,b,c\in\R$,~$\mu\in\R^\ast$&$\mathcal{C}^5(\G)=0$&\\
$\G_{8,34}$&$[e_1,e_2]=e_3$&$[e_1,e_3]=ae_6+be_7+ce_8$&$[e_1,e_4]=de_6+\lambda_1e_7$\\
&$[e_2,e_3]=ae_5+\lambda_2(e_6-e_7)-be_8$
&$[e_2,e_4]=(a+2b)e_6-be_7+(c-2\lambda_1)e_8$&$[e_3,e_4]=(\lambda_1-c)e_5-(a+b)e_8$\\
&$[e_1,e_6]=-e_8$&$[e_1,e_7]=-e_8$&$[e_1,e_8]=-e_5$\\
&$[e_2,e_7]=e_5$&$[e_4,e_6]=-e_5$&$[e_4,e_7]=-e_5$\\
&$a,b,c,d\in\R$,~$\lambda_j\in\R$
&$\mathcal{C}^5(\G)=0$&&\\
$\G_{8,35}$&$[e_1,e_2]=e_3$&$[e_1,e_3]=ae_7+be_8$&$[e_1,e_4]=be_7$\\
&$[e_2,e_3]=\frac{a}{\mu-1}e_6+\frac{b}{\mu}(2-\mu)(\mu-1)e_7$
&$[e_2,e_4]=ce_5+\frac{b}{\mu-1}e_6$&$[e_1,e_6]=-e_5$\\
&$[e_1,e_7]=\frac{\mu}{1-\mu}e_6$&$[e_1,e_8]=(1-\mu)e_7$
&$[e_2,e_7]=\frac{1}{1-\mu}e_5$
\\
&$[e_2,e_8]=(\mu-2)e_6$&$[e_3,e_8]=(1-\mu)e_5$
&$a,b,c\in\R$,~$\mu\in\R^\ast$,~$\mu\neq1$\\
&$\mathcal{C}^6(\G)=0$&&\\
$\G_{8,36}$&$[e_1,e_2]=e_3$&$[e_1,e_3]=ae_6+be_7$&$[e_1,e_4]=ce_6$\\
&$[e_2,e_3]=ae_5-be_6+de_7$
&$[e_2,e_4]=\lambda e_8$&$[e_3,e_4]=\lambda e_7$\\
&$[e_1,e_6]=-e_5$&$[e_1,e_8]=e_7$&$[e_2,e_7]=e_5$\\
&$[e_2,e_8]=-2e_6$&$[e_3,e_8]=e_5$&$a,b,c,d,\lambda\in\R$\\
&$\mathcal{C}^4(\G)=0$&&\\
$\G_{8,37}$&$[e_1,e_2]=e_3$&$[e_1,e_3]=ae_7+be_8$&$[e_2,e_3]=-be_7$\\
&$[e_2,e_4]=ce_5+be_6$
&$[e_3,e_4]=-be_5$&$[e_1,e_6]=-e_5$\\
&$[e_1,e_7]=-e_6$&$[e_1,e_8]=-e_7$&$[e_2,e_8]=e_6$\\
&$[e_3,e_8]=-e_5$&$a,b,c\in\R$&$\mathcal{C}^6(\G)=0$\\
$\G_{8,38}$&$[e_1,e_2]=e_3$&$[e_1,e_3]=ae_7+be_8$&$[e_2,e_3]=b(e_6-e_7)$\\
&$[e_2,e_4]=ce_5+be_6$
&$[e_3,e_4]=-be_5$&$[e_1,e_6]=-e_5$\\
&$[e_1,e_7]=-e_6$&$[e_1,e_8]=-e_6-e_7$&$[e_2,e_8]=-e_5+e_6$\\
&$[e_3,e_8]=-e_5$&$a,b,c\in\R$&$\mathcal{C}^6(\G)=0$\\
$\G_{8,39}$&$[e_1,e_2]=e_3$&$[e_1,e_3]=ae_5+be_6+3ce_8$&$[e_1,e_4]=de_5+2ce_7+\lambda_1e_8$\\
&$[e_2,e_3]=be_5+\lambda_2e_6+3\lambda_3e_8$
&$[e_2,e_4]=\lambda_4e_5+\eta_1e_6+2\lambda_3e_7+\eta_2e_8$&$[e_3,e_4]=-ce_5-\lambda_3e_6$\\
&$[e_1,e_7]=-\frac{1}{2}e_6$&$[e_2,e_7]=\frac{1}{2}e_5$&$a,b,c,d\in\R$,~$\lambda_j,\eta_j\in\R$\\
&$\mathcal{C}^3(\G)=0$&&\\
$\G_{8,40}$&$[e_1,e_2]=e_3$&$[e_1,e_3]=ae_5+be_6+ce_8$&$[e_1,e_4]=de_6+\lambda_1e_7+\lambda_2e_8$\\
&$[e_2,e_3]=be_5+\lambda_3e_6+\frac{3}{4}(2c-3\lambda_1)e_8$
&$[e_2,e_4]=\lambda_4e_5+\eta_1e_6+\frac{1}{2}(2c-3\lambda_1)e_7+\eta_2e_8$&$[e_3,e_4]=(\lambda_1-c)e_5+\frac{1}{4}(3\lambda_1-2c)e_6$\\
&$[e_1,e_7]=-e_5-\frac{1}{2}e_6$&$[e_2,e_7]=\frac{1}{2}e_5$&$a,b,c,d\in\R$,~$\lambda_j,\eta_j\in\R$\\
&$\mathcal{C}^3(\G)=0$&&\\
$\G_{8,41}$&$[e_1,e_2]=e_3$&$[e_1,e_3]=ae_5+be_6+3ce_8$&$[e_1,e_4]=de_5+2ce_7+\lambda_1e_8$\\
&$[e_2,e_3]=be_5+\lambda_2e_6+6ce_7+\lambda_4e_8$
&$[e_2,e_4]=\eta_1e_5+\frac{2}{3}(\lambda_1+\lambda_4)e_7+\eta_2e_8$&$[e_3,e_4]=-ce_5+\frac{1}{3}(2\lambda_1-\lambda_4)e_6$\\
&$[e_1,e_7]=-\frac{1}{2}e_6$&$[e_2,e_7]=\frac{1}{2}e_5$&$[e_2,e_8]=-e_6$\\
&$a,b,c,d\in\R$,~$\lambda_j,\eta_j\in\R$&$\mathcal{C}^4(\G)=0$&\\
$\G_{8,42}$&$[e_1,e_2]=e_3$&$[e_1,e_3]=ae_5+be_6+ce_7+de_8$&$[e_1,e_4]=\lambda_1e_7+\lambda_2e_8$\\
&$[e_2,e_3]=be_5+\lambda_3e_6+2de_7-\frac{1}{2}ce_8$
&$[e_2,e_4]=\lambda_3e_5+\lambda_4e_6+\frac{1}{3}(2\lambda_2-c)e_7+(d-\frac{3}{2}\lambda_1)e_8$&$[e_3,e_4]=(\lambda_1-d)e_5+\frac{1}{6}(c+4\lambda_2)e_6$\\
&$[e_1,e_7]=-\frac{1}{2}e_6$&$[e_1,e_8]=-e_5$&$[e_2,e_7]=\frac{1}{2}e_5$\\
&$[e_2,e_8]=-e_6$&$a,b,c,d\in\R$,~$\lambda_j,\in\R$&$\mathcal{C}^4(\G)=0$\\
$\G_{8,43}$&$[e_1,e_2]=e_3$&$[e_1,e_3]=ae_5+be_6+ce_7+de_8$&$[e_1,e_4]=\lambda_1e_7+\lambda_2e_8$\\
&$[e_2,e_3]=be_5+\lambda_3e_6+2de_7+\frac{1}{2}ce_8$&$[e_2,e_4]=\lambda_4e_5+\eta e_6+\frac{1}{3}(2\lambda_2+c)e_7+(\frac{3}{2}\lambda_1-d)e_8$&$[e_3,e_4]=(\lambda_1-d)e_5+\frac{1}{6}(4\lambda_2-c)e_6$\\
&$[e_1,e_7]=-\frac{1}{2}e_6$&$[e_1,e_8]=e_5$&$[e_2,e_7]=\frac{1}{2}e_5$\\
&$[e_2,e_8]=-e_6$&$a,b,c,d,\eta\in\R$,~$\lambda_j\in\R$&$\mathcal{C}^4(\G)=0$\\
$\G_{8,44}$&$[e_1,e_2]=e_3$&$[e_1,e_3]=ae_5+be_6+ce_8$&$[e_1,e_4]=de_5+\lambda_1e_7+\lambda_2e_8$\\
&$[e_2,e_3]=be_5+\lambda_3e_6-\frac{1}{4\mu}\big((4\mu+9)\lambda_1-6c\big)e_8$&$[e_2,e_4]=\eta_1e_5+\eta_2e_6+\frac{1}{2\mu}(2c-3\lambda_1)e_7+\eta_3e_8$&$[e_3,e_4]=(\lambda_1-c)e_5+\frac{1}{4\mu}\big((4\mu+3)\lambda_1-2c\big)e_6$\\
&$[e_1,e_7]=-\mu e_5-\frac{1}{2}e_6$&$[e_2,e_7]=\frac{1}{2}e_5-e_6$&$a,b,c,d\in\R$,~$\lambda_j,\eta_j\in\R$,~$\mu\in\R^{\ast+}$\\
&$\mathcal{C}^3(\G)=0$&&\\
$\G_{8,45}$&$[e_1,e_2]=e_3$&$[e_1,e_3]=ae_5+be_6+ce_8$&$[e_1,e_4]=de_5+\lambda_1e_7+\lambda_2e_8$\\
&$[e_2,e_3]=be_5+\lambda_3e_6+\frac{1}{4\mu}\big((4\mu-9)\lambda_1+6c\big)e_8$&$[e_2,e_4]=\lambda_4e_5+\eta_1e_6+\frac{1}{2\mu}(2c-3\lambda_1)e_7+\eta_2e_8$&$[e_3,e_4]=(\lambda_1-c)e_5+\frac{1}{4\mu}\big((3-4\mu)\lambda_1-2c\big)e_6$\\
&$[e_1,e_7]=-\mu e_5-\frac{1}{2}e_6$&$[e_2,e_7]=\frac{1}{2}e_5+e_6$&$a,b,c,d\in\R$,~$\lambda_j,\eta_j\in\R$,~$\mu\in\R^{\ast+},~~\mu\neq\frac{1}{4}$\\
&$\mathcal{C}^3(\G)=0$&&\\
$\G_{8,46}$&$[e_1,e_2]=e_3$&$[e_1,e_3]=ae_5+be_6+ce_7+de_8$&$[e_1,e_4]=\lambda_1e_5+\lambda_2e_7+\lambda_3e_8$\\
&$[e_2,e_3]=be_5+\lambda_4e_6-2ce_7+(6d-8\lambda_2)e_8$&$[e_2,e_4]=\eta_1e_5+\eta_2e_6+(4d-6\lambda_2)e_7+\eta_3e_8$&$[e_3,e_4]=(\lambda_2-d)(e_5+2e_6)$\\
&$[e_1,e_7]=-\frac{1}{4}e_5-\frac{1}{2}e_6$&$[e_2,e_7]=\frac{1}{2}e_5+e_6$&$a,b,c,d\in\R$,~$\lambda_j,\eta_j\in\R$\\
&$\mathcal{C}^4(\G)=0$&&\\
$\G_{8,47}$&$[e_1,e_2]=e_3$&$[e_1,e_3]=ae_5+be_6+ce_7+\mu de_8$&$[e_1,e_4]=\lambda_1e_7+\lambda_2e_8$\\
&$[e_2,e_3]=be_5+\lambda_3e_6+de_7+c(\mu-1)e_8$&$[e_2,e_4]=\lambda_4e_5+\eta_1e_6+\eta_2e_7+\big((\mu-2)\lambda_1+\mu d\big)e_8$&$[e_3,e_4]=(\lambda_1-\mu d)e_5+\frac{(1-\mu)\mu c+\lambda_2}{\mu+1}e_6$\\
&$[e_1,e_7]=-\mu e_6$&$[e_1,e_8]=-e_5$&$[e_2,e_7]=(1-\mu)e_5$\\
&$[e_2,e_8]=-e_6$&$a,b,c,d\in\R$, $\lambda_j,\eta_j\in\R$,~$\mu>\frac{1}{2}$,~$\eta_2=\frac{c(\mu-1)+\lambda_2}{\mu+1}$&$\mathcal{C}^4(\G)=0$\\
$\G_{8,48}$&$[e_1,e_2]=e_3$&$[e_1,e_3]=ae_5+be_6+ce_7+de_8$&$[e_1,e_4]=\lambda_1e_7+\lambda_2e_8$\\
&$[e_2,e_3]=be_5+\lambda_3e_6+(c-2td)e_7-\frac{1}{2}ce_8$&$[e_2,e_4]=\lambda_4e_5+\eta_1e_6+\eta_2e_7+(d-\frac{3}{2}\lambda_1)e_8$&$[e_3,e_4]=(\lambda_1-d)e_5+\frac{1}{3}(\lambda_1+\frac{1}{2}c-2t\lambda_2)e_6$\\
&$[e_1,e_7]=-\frac{1}{2}e_6$&$[e_1,e_8]=-e_5$&$[e_2,e_7]=\frac{1}{2}e_5-\frac{1}{2}e_6$\\
&$[e_2,e_8]=te_6$&$a,b,c,d\in\R$,~$\lambda_j,\eta_j\in\R$,~$t\in\R^{\ast+}$&$\eta_2=\frac{1}{3}(\lambda_1-2t\lambda_2-c)$\\
&$\mathcal{C}^4(\G)=0$&&\\
$\G_{8,49}$&$[e_1,e_2]=e_3$&$[e_1,e_3]=ae_5+be_6+c(e_7+e_8)$&$[e_1,e_4]=de_5+\lambda_1e_7+(2\lambda_1-c)e_8$\\
&$[e_2,e_3]=be_5+\lambda_2e_6+\lambda_3e_7+ce_8$&$[e_2,e_4]=\lambda_4e_5+\eta e_7+(c+\lambda_1-\eta)e_8$&$[e_3,e_4]=(\lambda_1-c)e_5+(\eta-c)e_6$\\
&$[e_1,e_8]=-e_6$&$[e_2,e_7]=e_5-e_6$&$[e_2,e_8]=-e_5$\\
&$a,b,c,d,\eta\in\R$,~$\lambda_j\in\R$&$\mathcal{C}^4(\G)=0$&\\
$\G_{8,50}$&$[e_1,e_2]=e_3$&$[e_1,e_3]=ae_5+be_6+\mu ce_8$&$[e_1,e_4]=de_5+\lambda_1e_6-\frac{\mu c}{\mu-2}e_7+\lambda_2e_8$\\
&$[e_2,e_3]=be_5+\lambda_3e_6+ce_7-\big((\mu+1)\lambda_4+\lambda_2\big)e_8$&$[e_2,e_4]=-\lambda_4e_7+\eta e_8$&$[e_3,e_4]=\frac{\mu c(1-\mu)}{\mu-2}e_5+(\mu\lambda_4+\lambda_2)e_6$\\
&$[e_1,e_7]=-\mu e_6$&$[e_2,e_7]=(1-\mu)e_5$&$[e_2,e_8]=-e_6$\\
&$a,b,c,d,\eta\in\R$,~$\lambda_j\in\R$,~$\mu<\frac{1}{2}$&$\mathcal{C}^4(\G)=0$&\\
$\G_{8,51}$&$[e_1,e_2]=e_3$&$[e_1,e_3]=ae_6+be_7+ce_8$&$[e_1,e_4]=de_7+\lambda_1e_8$\\
&$[e_2,e_3]=ae_5$&$[e_2,e_4]=\lambda_2e_5+(c-d)e_6+\lambda_3e_8$&$[e_3,e_4]=(d-c)e_5$\\
&$[e_1,e_6]=-e_5$&$[e_1,e_7]=-e_6$&$a,b,c,d\in\R$,~$\lambda_j\in\R$\\
&$\mathcal{C}^5(\G)=0$&&\\
$\G_{8,52}$&$[e_1,e_2]=e_3$&$[e_1,e_3]=ae_6+be_7+ce_8$&$[e_1,e_4]=de_7+\lambda_1e_8$\\
&$[e_2,e_3]=ae_5+(c-b)e_6+\lambda_2e_7$&$[e_2,e_4]=(\lambda_1+c-2d)e_6$&$[e_3,e_4]=(d-c)e_5$\\
&$[e_1,e_6]=-e_5$&$[e_1,e_8]=-e_6$&$[e_2,e_7]=e_5$\\
&$[e_2,e_8]=-e_5$&$a,b,c,d\in\R$,~$\lambda_j\in\R$&$\mathcal{C}^5(\G)=0$\\
$\G_{8,53}$&$[e_1,e_2]=e_3$&$[e_1,e_3]=ae_6+be_7+ce_8$&$[e_1,e_4]=de_7+\lambda_1e_8$\\
&$[e_2,e_3]=ae_5+b(\mu-1)e_6$&$[e_2,e_4]=\lambda_2e_5+\big((\mu-2)d+c\big)e_6+\lambda_3e_8$&$[e_3,e_4]=(d-c)e_5$\\
&$[e_1,e_6]=-e_5$&$[e_1,e_7]=-\mu e_6$&$[e_2,e_7]=(1-\mu)e_5$\\
&$a,b,c,d\in\R$,~$\lambda_j\in\R$,~$\mu\in\R^\ast$&$\mathcal{C}^5(\G)=0$&\\
$\G_{8,54}$&$[e_1,e_2]=e_3$&$[e_1,e_3]=ae_5+be_6+ce_7+de_8$&$[e_1,e_4]=\lambda_1e_6+\lambda_2e_7+\lambda_3e_8$\\
&$[e_2,e_3]=be_5-ce_6+\lambda_4e_7+\eta_1e_8$&$[e_2,e_4]=(d-2\lambda_2)e_6+\eta_1e_7+\eta_2e_8$&$[e_3,e_4]=(\lambda_2-d)e_5$\\
&$[e_1,e_6]=-e_5$&$[e_2,e_7]=e_5$&$a,b,c,d\in\R$,~$\lambda_j,\eta_j\in\R$\\
&$\mathcal{C}^4(\G)=0$&&\\
$\G_{8,55}$&$[e_1,e_2]=ae_5$&$[e_1,e_4]=-e_2$&$[e_2,e_4]=-e_3$\\
&$[e_3,e_4]=be_7$&$[e_1,e_8]=-e_7$&$[e_2,e_8]=e_6$\\
&$[e_3,e_8]=-e_5$&$[e_4,e_6]=-e_5$&$[e_4,e_7]=-e_6$\\
&$a,b\in\R$&$\mathcal{C}^6(\G)=0$&\\
$\G_{8,56}$&$[e_1,e_3]=ae_6$&$[e_1,e_4]=-e_2-ae_8$&$[e_2,e_3]=ae_5$\\
&$[e_2,e_4]=-e_3$&$[e_3,e_4]=be_5+3ae_7$&$[e_1,e_7]=e_5$\\
&$[e_1,e_8]=-e_7$&$[e_2,e_8]=e_6$&$[e_3,e_8]=-e_5$\\
&$[e_4,e_6]=-e_5$&$[e_4,e_7]=-e_6$&$a,b\in\R$\\
&$\mathcal{C}^6(\G)=0$&&\\
$\G_{8,57}$&$[e_1,e_3]=ae_5$&$[e_1,e_4]=-e_2-\frac{2a}{\mu+1}e_7$&$[e_2,e_4]=-e_3-\frac{a}{\mu+1}e_6$\\
&$[e_3,e_4]=be_5+ce_8$&$[e_1,e_6]=-(\mu+9)e_5+2e_8$&$[e_1,e_7]=-4e_6$\\
&$[e_2,e_7]=-4e_5+e_8$&$[e_4,e_6]=e_5-\frac{1}{4}e_8$&$a,b,c\in\R$,~$\mu\in\R$,~$\mu\neq-1$\\
&$\mathcal{C}^4(\G)=0$&&\\
$\G_{8,58}$&$[e_1,e_2]=ae_5$&$[e_1,e_4]=-e_2+be_7$&$[e_2,e_4]=-e_3+\frac{b}{2}e_6$\\
&$[e_3,e_4]=ce_5+de_8$&$[e_1,e_6]=-8e_5+2e_8$&$[e_1,e_7]=-4e_6$\\
&$[e_2,e_7]=-4e_5+e_8$&$[e_4,e_6]=e_5-\frac{1}{4}e_8$&$a,b,c,d\in\R$\\
&$\mathcal{C}^4(\G)=0$&&\\
$\G_{8,59}$&$[e_1,e_3]=ae_6$&$[e_1,e_4]=-e_2+\frac{a}{2(\mu+1)}e_7$&$[e_2,e_3]=ae_5-\frac{a}{4}e_8$\\
&$[e_2,e_4]=-e_3+\frac{a}{4(\mu+1)}e_6-\frac{a}{4}e_7$&$[e_3,e_4]=be_5+ce_8$&$[e_1,e_6]=-(\mu+9)e_5+2e_8$\\
&$[e_1,e_7]=-e_5-4e_6$&$[e_2,e_7]=-4e_5+e_8$&$[e_4,e_6]=e_5-\frac{1}{4}e_8$\\
&$a,b,c\in\R$,~$\mu\in\R$,~$\mu\neq-1$&$\mathcal{C}^4(\G)=0$&\\
$\G_{8,60}$&$[e_1,e_2]=ae_5$&$[e_1,e_4]=-e_2+be_7$&$[e_2,e_4]=-e_3+\frac{b}{2}e_6$\\
&$[e_3,e_4]=ce_5+de_8$&$[e_1,e_6]=-8e_5+2e_8$&$[e_1,e_7]=-e_5-4e_6$\\
&$[e_2,e_7]=-4e_5+e_8$&$[e_4,e_6]=e_5-\frac{1}{4}e_8$&$a,b,c,d\in\R$\\
&$\mathcal{C}^4(\G)=0$&&\\
$\G_{8,61}$&$[e_1,e_3]=ae_5$&$[e_1,e_4]=-e_2-\frac{2a}{\mu+1}e_7$&$[e_2,e_4]=-e_3-\frac{a}{\mu+1}e_6$\\
&$[e_3,e_4]=be_5+ce_8$&$[e_1,e_6]=-(\mu+9)e_5+2e_8$&$[e_1,e_7]=-\mu_1e_5-4e_6$\\
&$[e_2,e_7]=-4e_5+e_8$&$[e_4,e_6]=e_5-\frac{1}{4}e_8$&$[e_4,e_7]=-e_8$\\
&$a,b,c\in\R$,~$\mu,\mu_1\in\R$,~$\mu\neq-1$&$\mathcal{C}^4(\G)=0$&\\
$\G_{8,62}$&$[e_1,e_2]=ae_5$&$[e_1,e_4]=-e_2+be_7$&$[e_2,e_4]=-e_3+\frac{b}{2}e_6$\\
&$[e_3,e_4]=ce_5+de_8$&$[e_1,e_6]=-8e_5+2e_8$&$[e_1,e_7]=-\mu_1e_5-4e_6$\\
&$[e_2,e_7]=-4e_5+e_8$&$[e_4,e_6]=e_5-\frac{1}{4}e_8$&$[e_4,e_7]=-e_8$\\
&$a,b,c,d\in\R$,~$\mu_1\in\R$&$\mathcal{C}^4(\G)=0$&\\
$\G_{8,63}$&$[e_1,e_2]=ae_5+2be_6$&$[e_1,e_3]=-be_5$&$[e_1,e_4]=-e_2$\\
&$[e_2,e_4]=-e_3$&$[e_3,e_4]=ce_5+de_8$&$[e_1,e_6]=\frac{1}{2}e_8$\\
&$[e_2,e_7]=\frac{1}{3}e_8$&$[e_4,e_6]=-\frac{1}{2}e_5$&$[e_4,e_7]=-\frac{2}{3}e_6$\\
&$a,b,c,d\in\R$&$\mathcal{C}^4(\G)=0$&\\
$\G_{8,64}$&$[e_1,e_2]=ae_6$&$[e_1,e_4]=-e_2+be_7$&$[e_2,e_3]=-\frac{a}{6}e_8$\\
&$[e_2,e_4]=-e_3+\frac{2b}{3}e_6-\frac{a}{2}e_7$&$[e_3,e_4]=ce_5-\frac{a}{3}e_6+de_8$&$[e_1,e_6]=\frac{1}{2}e_8$\\
&$[e_1,e_7]-e_5$&$[e_2,e_7]=\frac{1}{3}e_8$&$[e_4,e_6]=-\frac{1}{2}e_5$\\
&$[e_4,e_7]=-\frac{2}{3}e_6$&$a,b,c,d\in\R$&$\mathcal{C}^4(\G)=0$\\
$\G_{8,65}$&$[e_1,e_2]=ae_6+be_7+ce_8$&$[e_1,e_3]=-be_6$&$[e_1,e_4]=-e_2+de_8$\\
&$[e_2,e_3]=-2be_5$&$[e_2,e_4]=-e_3-ae_8$&$[e_3,e_4]=\lambda_1e_5+\lambda_2e_7+3be_8$\\
&$[e_1,e_8]=-e_5$&$[e_4,e_6]=-e_5$&$[e_4,e_7]=-e_6$\\
&$a,b,c,d\in\R$&$\mathcal{C}^6(\G)=0$&\\
$\G_{8,66}$&$[e_1,e_2]=ae_6+be_7$&$[e_1,e_3]=-be_6$&$[e_1,e_4]=-e_2+ce_7$\\
&$[e_2,e_3]=-2be_5$&$[e_2,e_4]=-e_3-(a+b)e_8$&$[e_3,e_4]=de_5+3be_8$\\
&$[e_1,e_7]=-e_8$&$[e_1,e_8]=-e_5$&$[e_4,e_6]=-e_5$\\
&$[e_4,e_7]=-e_5-e_6$&$a,b,c,d\in\R$&$\mathcal{C}^5(\G)=0$\\
$\G_{8,67}$&$[e_1,e_2]=ae_6+be_7$&$[e_1,e_3]=-be_6$&$[e_1,e_4]=-e_2+ce_7$\\
&$[e_2,e_3]=-2be_5$&$[e_2,e_4]=-e_3+(b-a)e_8$&$[e_3,e_4]=de_5+3be_8$\\
&$[e_1,e_7]=e_8$&$[e_1,e_8]=-e_5$&$[e_4,e_6]=-e_5$\\
&$[e_4,e_7]=e_5-e_6$&$a,b,c,d\in\R$&$\mathcal{C}^5(\G)=0$\\
$\G_{8,68}$&$[e_1,e_2]=ae_5+be_6$&$[e_1,e_4]=-e_2$&$[e_2,e_4]=-e_3+ae_6$\\
&$[e_3,e_4]=be_5+ce_8$&$[e_1,e_6]=e_8$&$[e_2,e_7]=\frac{1}{2}e_8$\\
&$[e_4,e_5]=-e_8$&$[e_4,e_7]=-\frac{1}{2}e_6$\\
&$a,b,c\in\R$&$\mathcal{C}^5(\G)=0$&\\
$\G_{8,69}$&$[e_1,e_2]=ae_6$&$[e_1,e_4]=-e_2+be_7$&$[e_2,e_4]=-e_3+\frac{2b}{3}e_6$\\
&$[e_3,e_4]=ce_5+de_8$&$[e_1,e_6]=e_8$&$[e_1,e_7]=-\frac{1}{3}e_5-\frac{1}{3}e_8$\\
&$[e_2,e_7]=\frac{2}{3}e_8$&$[e_4,e_5]=-e_8$&$[e_4,e_7]=-\frac{1}{3}e_5-\frac{1}{3}e_6$\\
&$a,b,c,d\in\R$&$\mathcal{C}^5(\G)=0$&\\
$\G_{8,70}$&$[e_1,e_2]=ae_6$&$[e_1,e_4]=-e_2+be_7$&$[e_2,e_4]=-e_3+\frac{b(\mu+1)}{2}e_6$\\
&$[e_3,e_4]=ce_5+de_8$&$[e_1,e_6]=e_8$&$[e_1,e_7]=-\mu e_5$\\
&$[e_2,e_7]=\frac{\mu+1}{2}e_8$&$[e_4,e_5]=-e_8$&$[e_4,e_7]=\frac{\mu-1}{2}e_6$\\
&$a,b,c,d\in\R$,~$\mu\in\R^\ast$,~$\mu\neq-\frac{1}{3},3$&$\mathcal{C}^5(\G)=0$&\\
$\G_{8,71}$&$[e_1,e_2]=ae_6+be_7$&$[e_1,e_3]=be_6$&$[e_1,e_4]=-e_2+ce_7$\\
&$[e_2,e_4]=-e_3+2ce_6$&$[e_3,e_4]=de_5+\lambda e_8$&$[e_1,e_6]=e_8$\\
&$[e_1,e_7]=-3e_5$&$[e_2,e_7]=2e_8$&$[e_4,e_5]=-e_8$\\
&$[e_4,e_7]=e_6$&$a,b,c,d,\lambda\in\R$&$\mathcal{C}^5(\G)=0$\\
$\G_{8,72}$&$[e_1,e_2]=ae_6+be_7$&$[e_1,e_3]=-\frac{2b}{3}e_6$&$[e_1,e_4]=-e_2+ce_7$\\
&$[e_2,e_3]=-\frac{5b}{3}e_5$&$[e_2,e_4]=-e_3+\frac{c}{3}e_6$&$[e_3,e_4]=de_5-5be_7+\lambda e_8$\\
&$[e_1,e_6]=e_8$&$[e_1,e_7]=\frac{1}{3}e_5$&$[e_2,e_7]=\frac{1}{3}e_8$\\
&$[e_4,e_5]=-e_8$&$[e_4,e_7]=-\frac{2}{3} e_6$&$a,b,c,d,\lambda\in\R$\\
&$\mathcal{C}^6(\G)=0$&&\\
$\G_{8,73}$&$[e_1,e_2]=ae_6+be_7$&$[e_1,e_3]=\mu be_8$&$[e_1,e_4]=-e_2+ce_7$\\
&$[e_2,e_3]=-be_5$&$[e_2,e_4]=-e_3-be_7-\big(a-c+(\mu+2)b\big)e_8$&$[e_3,e_4]=de_5-be_6+be_8$\\
&$[e_1,e_7]=-e_6-(\mu+2)e_8$&$[e_1,e_8]=-e_5$&$[e_2,e_7]=-e_5$\\
&$[e_4,e_6]=-e_5$&$[e_4,e_7]=-(\mu+2)e_5-e_6-2e_8$&$a,b,c,d,\lambda\in\R$,~$\mu\in\R$,~$\mu\neq-2$\\
&$\mathcal{C}^5(\G)=0$&&\\
$\G_{8,74}$&$[e_1,e_2]=ae_7+be_8$&$[e_1,e_3]=-2ae_8$&$[e_1,e_4]=-e_2+ce_7$\\
&$[e_2,e_3]=-ae_5+de_6$&$[e_2,e_4]=-e_3-ae_7+ce_8$&$[e_3,e_4]=\lambda e_5-ae_6-de_7+ae_8$\\
&$[e_1,e_7]=-e_6$&$[e_1,e_8]=-e_5$&$[e_2,e_7]=-e_5$\\
&$[e_4,e_6]=-e_5$&$[e_4,e_7]=-e_6-2e_8$&$a,b,c,d,\lambda\in\R$\\
&$\mathcal{C}^6(\G)=0$&&\\
$\G_{8,75}$&$[e_1,e_2]=ae_6$&$[e_1,e_4]=-e_2+be_7$&$[e_2,e_3]=-\frac{a(\mu+1)}{\mu-1}e_8$\\
&$[e_2,e_4]=-e_3-\frac{a(\mu+1)}{\mu}e_7$&$[e_3,e_4]=ce_5+\frac{a(\mu+1)}{\mu(\mu-1)}e_6+de_8$&$[e_1,e_6]=-\mu e_8$\\
&$[e_1,e_7]=-\mu e_5$&$[e_4,e_5]=-e_8$&$[e_4,e_6]=-(\mu+1)e_5$\\
&$[e_4,e_7]=-e_6$&$a,b,c,d\in\R$,~$\mu\in\R^\ast$,~$\mu\neq1$&$\mathcal{C}^6(\G)=0$\\
$\G_{8,76}$&$[e_1,e_4]=-e_2+ae_7$&$[e_2,e_3]=be_8$&$[e_2,e_4]=-e_3$\\
&$[e_3,e_4]=ce_5-be_6+de_8$&$[e_1,e_6]=-e_8$&$[e_1,e_7]=-e_5$\\
&$[e_4,e_5]=-e_8$&$[e_4,e_6]=-2e_5$&$[e_4,e_7]=-e_6$\\
&$a,b,c,d\in\R$&$\mathcal{C}^6(\G)=0$&\\
$\G_{8,77}$&$[e_1,e_2]=\big(a+\frac{2(c-b)\mu^2}{3}\big)e_5+\frac{2b(\mu-1)}{3(\mu+1)}e_6$&$[e_1,e_3]=be_5$&$[e_1,e_4]=-e_2+de_6$\\
&$[e_2,e_4]=-e_3-\frac{3a+2\mu^2c}{3\mu}e_6-\frac{b}{\mu}e_7$&$[e_3,e_4]=de_5-\frac{b}{\mu}e_6$&$[e_1,e_6]=-\mu e_8$\\
&$[e_1,e_7]=-\frac{\mu}{3}(2\mu+1)e_5+\frac{2\mu^3}{3}e_8$&$[e_2,e_7]=-\frac{2\mu}{3}e_8$&$[e_4,e_5]=-e_8$\\
&$[e_4,e_6]=-(\mu+1)e_5$&$[e_4,e_7]=\frac{2\mu^3}{3}e_5-(\frac{2\mu}{3}+1)e_6$&$a,b,c,d\in\R$,~$\mu\in\R^\ast$,~$\mu\neq-1,\frac{1}{2}$\\
&$\mathcal{C}^6(\G)=0$&&\\
$\G_{8,78}$&$[e_1,e_2]=ae_6$&$[e_1,e_4]=-e_2+be_7$&$[e_2,e_4]=-e_3+\frac{2b}{3}e_6$\\
&$[e_3,e_4]=ce_5+de_8$&$[e_1,e_6]=e_8$&$[e_1,e_7]=-\frac{1}{3}e_5-\frac{2}{3}e_8$\\
&$[e_2,e_7]=\frac{2}{3}e_8$&$[e_4,e_5]=-e_8$&$[e_4,e_7]=-\frac{2}{3}e_5-\frac{1}{3}e_6$\\
&$a,b,c,d\in\R$&$\mathcal{C}^5(\G)=0$&\\
$\G_{8,79}$&$[e_1,e_2]=ae_6$&$[e_1,e_4]=-e_2+be_7$&$[e_2,e_3]=\frac{9a}{2}e_8$\\
&$[e_2,e_4]=-e_3+(\frac{2b}{3}-\frac{3a}{4})e_6-\frac{9a}{2}e_7$&$[e_3,e_4]=ce_5-9ae_6+de_8$&$[e_1,e_6]=-\frac{1}{2}e_8$\\
&$[e_1,e_7]=-\frac{1}{3}e_5+\frac{1}{12}e_8$&$[e_2,e_7]=-\frac{1}{3}e_8$&$[e_4,e_5]=-e_8$\\
&$[e_4,e_6]=-\frac{3}{2}e_5$&$[e_4,e_7]=\frac{1}{12}e_5-\frac{4}{3}e_6$&$a,b,c,d\in\R$\\
&$\mathcal{C}^6(\G)=0$&&\\
$\G_{8,80}$&$[e_1,e_2]=ae_6+be_7$&$[e_1,e_3]=-\frac{b(\mu_1-1)}{\mu-1}e_6$&$[e_1,e_4]=-e_2+ce_6+de_7$\\
&$[e_2,e_3]=-\frac{b(\mu+1)(\mu_1-1)(\mu_1-\mu)}{(\mu\mu_1+\mu-2\mu_1)(\mu-1)}e_6-\frac{a\mu(\mu+1)}{\mu_1(\mu-1)}e_8$&$[e_2,e_4]=-e_3+\frac{d(\mu_1-\mu)}{\mu(\mu-1)}e_6-\frac{a(\mu+1)}{\mu_1}e_7$&$[e_3,e_4]=\lambda e_5+\frac{a(\mu+1)}{\mu_1(\mu-1)}e_6+\frac{b(\mu+1)(\mu_1-1)}{\mu\mu_1+\mu-2\mu_1}e_7$\\
&$[e_1,e_6]=-\mu e_8$&$[e_1,e_7]=-\mu_1e_5$&$[e_2,e_7]=\frac{\mu-\mu_1}{\mu-1}e_8$\\
&$[e_4,e_5]=-e_8$&$[e_4,e_6]=-(\mu+1)e_5$&$[e_4,e_7]=\frac{1-\mu_1}{\mu-1}e_6$\\
&$a,b,c,d,\lambda\in\R$,~$\mu,\mu_1\in\R^\ast$,~$\mu\neq\pm1$,~$\mu\neq\frac{2\mu_1}{\mu_1+1},$&$      \mu_1\neq-1$,~$\mu_1\neq\frac{\mu}{2\mu-1}$,~ $\mu=\frac{\pm(\mu_1+3)+\sqrt{3}\sqrt{(8\mu_1+3)(\mu_1-1)^2}}{4\mu_1}$&$\mathcal{C}^7(\G)=0$&\\
$\G_{8,81}$&$[e_1,e_2]=ae_6$&$[e_1,e_4]=-e_2+be_7$&$[e_2,e_4]=-e_3+\frac{b(\mu_1+1)}{2}e_6$\\
&$[e_3,e_4]=ce_5+de_8$&$[e_1,e_6]=e_8$&$[e_1,e_7]=-\mu_1e_5$\\
&$[e_2,e_7]=\frac{\mu_1+1}{2}e_8$&$[e_4,e_5]=-e_8$&$[e_4,e_7]=\frac{\mu_1-1}{2}e_6$\\
&$a,b,c,d\in\R$,~$\mu_1\in\R^\ast$,~$\mu_1\neq-\frac{1}{3},3,-1$&$\mathcal{C}^5(\G)=0$&\\
$\G_{8,82}$&$[e_1,e_2]=ae_6+be_7$&$[e_1,e_3]=be_6$&$[e_1,e_4]=-e_2+ce_7$\\
&$[e_2,e_4]=-e_3+2ce_6$&$[e_3,e_4]=de_5+\lambda e_8$&$[e_1,e_6]=e_8$\\
&$[e_1,e_7]=-3e_5$&$[e_2,e_7]=2e_8$&$[e_4,e_5]=-e_8$\\
&$[e_4,e_7]=e_6$&$a,b,c,d,\lambda\in\R$&$\mathcal{C}^5(\G)=0$\\
$\G_{8,83}$&$[e_1,e_2]=ae_6$&$[e_1,e_4]=-e_2+ce_6+de_7$&$[e_2,e_3]=-\frac{a\mu(\mu+1)}{\mu_1(\mu-1)}e_8$\\
&$[e_2,e_4]=-e_3+\frac{d(\mu_1-\mu)}{\mu(\mu-1)}e_6-\frac{a(\mu+1)}{\mu_1}e_7$&$[e_3,e_4]=\lambda e_5+\frac{a(\mu+1)}{\mu_1(\mu-1)}e_6$&$[e_1,e_6]=-\mu e_8$\\
&$[e_1,e_7]=-\mu_1e_5$&$[e_2,e_7]=\frac{\mu-\mu_1}{\mu-1}e_8$&$[e_4,e_5]=-e_8$\\
&$[e_4,e_6]=-(\mu+1)e_5$&$[e_4,e_7]=\frac{1-\mu_1}{\mu-1}e_6$&$a,b,c,d,\lambda\in\R$,~$\mu_1,\mu\in\R^\ast$,~$\mu,\mu_1\neq-1$,~$\mu\neq\frac{2\mu_1}{\mu_1+1}$\\
&$\mathcal{C}^6(\G)=0$&&\\
$\G_{8,84}$&$[e_1,e_2]=ae_6$&$[e_1,e_4]=-e_2+be_7$&$[e_2,e_4]=-e_3+\frac{b(\mu_1+1)}{2}e_6$\\
&$[e_3,e_4]=ce_5+de_8$&$[e_1,e_6]=e_8$&$[e_1,e_7]=-\mu_1e_5$\\
&$[e_2,e_7]=\frac{\mu_1+1}{2}e_8$&$[e_4,e_5]=-e_8$&$[e_4,e_7]=\frac{\mu_1-1}{2}e_6$\\
&$a,b,c,d\in\R$,~$\mu_1\in\R^\ast$,~$\mu_1\neq-1,\pm\frac{1}{3}$&$\mathcal{C}^5(\G)=0$&\\
$\G_{8,85}$&$[e_1,e_2]=ae_5+be_6$&$[e_1,e_3]=-b(\mu+1)e_5$&$[e_1,e_4]=-e_2$\\
&$[e_2,e_4]=-e_3-\frac{a}{\mu}e_6+\frac{b(\mu+1)}{\mu}e_7$&$[e_3,e_4]=ce_5+\frac{b(\mu+1)}{\mu}e_6+de_8$&$[e_1,e_6]=-\mu e_8$\\
&$[e_2,e_7]=\frac{\mu}{\mu-1}e_8$&$[e_4,e_5]=-e_8$&$[e_4,e_6]=-(\mu+1)e_5$\\
&$[e_4,e_7]=\frac{1}{\mu-1}e_6$&$a,b,c,d\in\R$, $\mu\in\R^\ast$, $\mu\neq1$&$\mathcal{C}^6(\G)=0$\\
$\G_{8,86}$&$[e_1,e_2]=ae_6$&$[e_1,e_4]=-e_2+be_7$&$[e_2,e_3]=-\frac{2a(3\mu_1+1)}{\mu_1^2-1}e_8$\\
&$[e_2,e_4]=-e_3+\frac{b(\mu_1+1)}{2}e_6-\frac{a(3\mu_1+1)}{\mu_1(\mu_1+1)}e_7$&$[e_3,e_4]=ce_5+\frac{a(3\mu_1+1)}{\mu_1(\mu_1-1)}e_6+de_8$&$[e_1,e_6]=-\frac{2\mu_1}{\mu_1+1}e_8$\\
&$[e_1,e_7]=-\mu_1e_5$&$[e_2,e_7]-\mu_1e_8$&$[e_4,e_5]=-e_8$\\
&$[e_4,e_6]=-\frac{3\mu_1+1}{\mu_1+1}e_5$&$[e_4,e_7]=-(\mu_1+1)e_6$&$a,b,c,d\in\R$,~$\mu_1\in\R^\ast$,~$\mu_1\neq-1,-\frac{1}{3}$\\
&$\mathcal{C}^6(\G)=0$&&\\
$\G_{8,87}$&$[e_1,e_2]=ae_6+be_7$&$[e_1,e_3]=-\frac{2b}{3}e_6$&$[e_1,e_4]=-e_2+ce_7$\\
&$[e_2,e_3]=-\frac{5b}{3}e_5$&$[e_2,e_4]=-e_3+\frac{c}{3}e_6$&$[e_3,e_4]=de_5-5be_7+\lambda e_8$\\
&$[e_1,e_6]=e_8$&$[e_1,e_7]=\frac{1}{3}e_5$&$[e_2,e_7]=\frac{1}{3}e_8$\\
&$ [e_4,e_5]=-e_8$&$[e_4,e_6]=-\frac{2}{3}e_6$&$a,b,c,d,\lambda\in\R$\\
&$\mathcal{C}^6(\G)=0$&&\\
$\G_{8,88}$&$[e_1,e_2]=ae_6$&$[e_1,e_4]=-e_2+ce_7$&$[e_2,e_3]=\frac{a\mu(\mu+1)}{\mu-1}e_8$\\
&$[e_2,e_4]=-e_3-\frac{c(\mu+1)}{\mu(\mu-1)}e_6+a(\mu+1)e_7$&$[e_3,e_4]=de_5-\frac{a(\mu+1)}{\mu-1}e_6+\lambda e_8$&$[e_1,e_6]=-\mu e_8$\\
&$[e_1,e_7]=e_5$&$[e_2,e_7]=\frac{1+\mu}{\mu-1}e_8$&$[e_4,e_5]=-e_8$\\
&$[e_4,e_6]=-(\mu+1)e_5$&$[e_4,e_7]=\frac{2}{\mu-1}e_6$&$a,b,c,d,\lambda\in\R$,~$\mu\in\R^\ast$,~$\mu\neq0,1,\pm\frac{1}{3}$\\
&$\mathcal{C}^6(\G)=0$&&\\
$\G_{8,89}$&$[e_1,e_2]=ae_6$&$[e_1,e_4]=-e_2+be_6+ce_7$&$[e_2,e_3]=\frac{a}{6}e_8$\\
&$[e_2,e_4]=-e_3-\frac{3c}{2}e_6+\frac{2a}{3}e_7$&$[e_3,e_4]=de_5+\frac{a}{2}e_6$&$[e_1,e_6]=\frac{1}{3}e_8$\\
&$[e_1,e_7]=e_5$&$[e_2,e_7]=-\frac{1}{2}e_8$&$[e_4,e_5]=-e_8$\\
&$[e_4,e_6]=-\frac{2}{3}e_5$&$[e_4,e_7]=-\frac{3}{2}e_6$&$a,b,c,d\in\R$\\
&$\mathcal{C}^6(\G)=0$&&\\
$\G_{8,90}$&$[e_1,e_2]=ae_5+be_6+ce_8$&$[e_1,e_3]=-be_5-ae_8$&$[e_1,e_4]=-e_2$\\
&$[e_2,e_3]=be_8$&$[e_2,e_4]=-e_3$&$[e_3,e_4]=de_5-be_6+\lambda e_7$\\
&$[e_4,e_5]=-e_8$&$[e_4,e_6]=-e_5$&$[e_4,e_7]=-e_6$\\
&$a,b,c,d,\lambda\in\R$&$\mathcal{C}^7(\G)=0$&\\
$\G_{8,91}$&$[e_1,e_2]=ae_5+be_6$&$[e_1,e_3]=-be_5$&$[e_1,e_4]=-e_2+ce_7$\\
&$[e_2,e_3]=de_8$&$[e_2,e_4]=-e_3-ae_7$&$[e_3,e_4]=\lambda e_5-(a+d)e_6+(b-d)e_7$\\
&$[e_1,e_7]=-e_8$&$[e_4,e_5]=-e_8$&$[e_4,e_6]=-e_5$\\
&$[e_4,e_7]=-e_5-e_6$&$a,b,c,d,\lambda\in\R$&$\mathcal{C}^7(\G)=0$\\
$\G_{8,92}$&$[e_1,e_2]=ae_6$&$[e_1,e_3]=be_8$&$[e_1,e_4]=-e_2-\frac{b}{\mu}e_7$\\
&$[e_2,e_4]=-e_3-\frac{a}{\mu}e_7$&$[e_3,e_4]=ce_5-\frac{a}{\mu}e_6+de_8$&$[e_1,e_7]=-\mu e_5$\\
&$[e_2,e_7]=\mu e_8$&$[e_4,e_5]=-e_8$&$[e_4,e_6]=-e_5$\\
&$[e_4,e_7]=(\mu-1)e_6$&$a,b,c,d\in\R$,~$\mu\in\R^\ast$,~$\mu\neq\frac{5}{3}$&$\mathcal{C}^6(\G)=0$\\
$\G_{8,93}$&$[e_1,e_2]=ae_6+be_7$&$[e_1,e_3]=\frac{2b}{3}e_6+ce_8$&$[e_1,e_4]=-e_2-\frac{3c}{5}e_7$\\
&$[e_2,e_3]=-\frac{b}{3}e_5$&$[e_2,e_4]=-e_3-\frac{3a}{5}e_7$&$[e_3,e_4]=de_5-\frac{3a}{5}e_6-\frac{b}{5}e_7+\lambda e_8$\\
&$[e_1,e_7]=-\frac{5}{3}e_5$&$[e_2,e_7]=\frac{5}{3}e_8$&$[e_4,e_5]=-e_8$\\
&$[e_4,e_6]=-e_5$&$[e_4,e_7]=\frac{2}{3} e_6$&$a,b,c,d,\lambda\in\R$\\
&$\mathcal{C}^7(\G)=0$&&\\
$\G_{8,94}$&$[e_1,e_2]=ae_6$&$[e_1,e_3]=be_8$&$[e_1,e_4]=-e_2-\frac{a+b}{\mu}e_7$\\
&$[e_2,e_4]=-e_3-\frac{a}{\mu}e_7$&$[e_3,e_4]=ce_5-\frac{a}{\mu}e_6+de_8$&$[e_1,e_7]=-\mu e_5+\mu e_8$\\
&$[e_2,e_7]=\mu e_8$&$[e_4,e_5]=-e_8$&$[e_4,e_6]=-e_5$\\
&$[e_4,e_7]=\mu e_5+(\mu-1)e_6$&$a,b,c,d\in\R$&$\mathcal{C}^6(\G)=0$\\
$\G_{8,95}$&$[e_1,e_2]=ae_6+be_7$&$[e_1,e_3]=\frac{2b}{3}e_6+ce_8$&$[e_1,e_4]=-e_2+\big(b-\frac{3}{5}(a+c)\big)e_7$\\
&$[e_2,e_3]=-\frac{b}{3}e_5-\frac{b}{3}e_8$&$[e_2,e_4]=-e_3+(b-\frac{3a}{5})e_7$&$[e_3,e_4]=de_5+(\frac{4b}{3}-\frac{3a}{5})e_6-\frac{b}{5}e_7+\lambda e_8$\\
&$[e_1,e_7]=-\frac{5}{3}e_5+\frac{5}{3}e_8$&$[e_2,e_7]=\frac{5}{3}e_8$&$[e_4,e_5]=-e_8$\\
&$[e_4,e_6]=-e_5$&$[e_4,e_7]=\frac{5}{3}e_5+\frac{2}{3}e_6$&$a,b,c,d,\lambda\in\R$\\
&$\mathcal{C}^7(\G)=0$&&
\\\hline
\end{longtable}
			\end{small}	
			}	
			
			\restoregeometry
\end{pr}

\begin{proof}
		This describes the general scheme of the proof. We consider a flat Lie algebra
		$(\h, \nabla)$ chosen among those obtained in
		Proposition~$\ref{4flat}$ (see Table~$\ref{4flatcomplete})$.
		This flat Lie algebra admits a basis $B = \{e_1, e_2, e_3,e_4\}$, and we denote by
		$\{e_1^\ast, e_2^\ast, e_3^\ast,e_4^\ast\}$ its dual basis.
		Let $(\G_{\nabla, \alpha}, \omega)$ be the Lagrangian extension of the
		flat Lie algebra $(\h, \nabla)$ with respect to
		$\alpha \in Z^2_{\rho}(\h, \h^\ast)$.

		Let $\sigma \in \mathcal{C}^{1}(\h, \h^\ast)$ and
		$\alpha' \in Z^{2}_{\rho}(\h, \h^\ast)$ be such that
		$\alpha' = \alpha - \partial_{\rho}\sigma$.
		In this case, $\G_{\nabla,\alpha} \cong \G_{\nabla,\alpha'}$ (see Proposition~\ref{cohooro}).
		Hence, the Lie algebra $\G_{\nabla,\alpha'}$ will be the object of our scheme.
		This Lie algebra is isomorphic to a family of Lie algebras described in
		Proposition~$\ref{prSymbracket}$, whose Lie brackets are given by equations~$(\ref{Bra1})$ and~$(\ref{Bra2})$. For a representation of Lie algebra structures as given in
		Proposition~\ref{prSymbracket}, we identify the basis
		$\{e_5, e_6, e_7,e_8\}$ with the dual basis of $\{e_1, e_2, e_3,e_4\}$.
		For illustrative purposes, we compute one example; the remaining cases are treated similarly. This example illustrates all the techniques required in every case. We now apply the above scheme to $(\mathfrak{a}_1,\nabla)$ (see Table~\ref{4flatcomplete}), endowed with
		the basis $B=\{e_1,e_2,e_3,e_4\}$, where the flat torsion-free connection $\nabla$ is defined by
		\begin{align}
	\nabla_{e_2}e_2&=e_1,\quad\nabla_{e_3}e_3=e_1,\quad\nabla_{e_4}e_4=e_1.
		\end{align}
	Next, we compute the ordinary Lie algebra cohomology group $H^2_\rho(\mathfrak{a}_1,\mathfrak{a}_1^\ast)$
	associated with the representation $\rho$.  
	
	Let $(\G_{\nabla,\alpha},\omega)$ be the Lagrangian extension of the flat Lie	algebra $(\mathfrak{a}_1,\nabla)$ with respect to
	$\alpha \in Z^2_{\rho}(\mathfrak{a}_1,\mathfrak{a}_1^\ast)$.
	Let $B^\ast=\{e^1,e^2,e^3,e^4\}$ be the dual basis of $B$, which generates
	$\mathfrak{a}_1^\ast$.
	Let $\alpha \in C^2(\mathfrak{a}_1,\mathfrak{a}_1^\ast)$ be given by
	\begin{align}
	\alpha(e_i,e_j)&=	\mathop{\resizebox{1.3\width}{!}{$\sum$}}^4\limits_{k=1}\alpha_{ijk}e^k,\quad 1\leq i<j\leq 4,~~\alpha_{ijk}\in\R.
	\end{align}
	Therefore, $\alpha \in Z^2_{\rho}(\mathfrak{a}_1,\mathfrak{a}_1^\ast)$
	if and only if $\alpha$ has the following form:
\begin{align}
	\alpha(e_i,e_j)&=	\mathop{\resizebox{1.3\width}{!}{$\sum$}}^4\limits_{k=1}\alpha_{ijk}e^k,\quad 1\leq i<j\leq 4,~~\alpha_{ijk}\in\R,
\end{align}
	such that $\alpha_{ij1}=0$ for all $1\leq i<j\leq 4$.  Moreover, $\alpha\in Z^2_\rho( \mathfrak{a}_1, \mathfrak{a}_1^\ast)$ satisfies the Jacobi identity; that is,
	\begin{align}
	\mathop{\resizebox{1.3\width}{!}{$\sum$}}\limits_{\mathrm{cycl}}\alpha(x,y)(z) = 0
	\end{align}
	for all $x,y,z \in \mathfrak{a}_1$, if and only if
	\begin{align}
		\alpha_{123}&=\alpha_{132},\quad\alpha_{124}=\alpha_{142},\quad\alpha_{134}=\alpha_{143},\quad\alpha_{234}=\alpha_{243}-\alpha_{342}.
	\end{align}
		Let now $\sigma \in \mathcal{C}^1(\mathfrak{a}_1,\mathfrak{a}_1^\ast)$ be given, with respect to the basis $\{e^1, e^2, e^3, e^4\}$, by
		\begin{align}
			\sigma(e_i)&=\mathop{\resizebox{1.3\width}{!}{$\sum$}}^4\limits_{j=1}\sigma_{ij}e^j,\quad\sigma_{ij}\in\R.
		\end{align}
	Next, we compute the differential of $\sigma$, denoted $\partial_\rho \sigma$, with respect to the dual representation $\rho$ of $\nabla$, that is,
	\begin{align}
		\big(\partial_\rho \sigma\big)(x,y)
		&= \rho(x)\sigma(y) - \rho(y)\sigma(x) - \sigma\big([x,y]\big),
		\quad \text{for all } x,y \in \mathfrak{a}_1.
	\end{align}
Therefore, we obtain
\begin{align*}
	\big(\partial_\rho\sigma\big)(e_1,e_2)&=\sigma_{11}e^2,&\big(\partial_\rho\sigma\big)(e_1,e_3)&=\sigma_{11}e^3,&\big(\partial_\rho\sigma\big)(e_1,e_4)&=\sigma_{11}e^4,\\
		\big(\partial_\rho\sigma\big)(e_2,e_3)&=-\sigma_{31}e^2+\sigma_{21}e^3,&\big(\partial_\rho\sigma\big)(e_2,e_4)&=-\sigma_{41}e^2+\sigma_{21}e^4,&\big(\partial_\rho\sigma\big)(e_3,e_4)&=-\sigma_{41}e^3+\sigma_{31}e^4.
\end{align*}
	Define $\alpha' :=\alpha-  \partial_\rho \sigma$. Then $\alpha'$ is given by
	\begin{align*}
		\alpha'(e_1,e_2)&=(\alpha_{122}-\sigma_{11})e^2+\alpha_{132}e^3+\alpha_{142}e^4, &\alpha'(e_1,e_3)&=\alpha_{132}e^2+(\alpha_{133}-\sigma_{11})e^3+\alpha_{143}e^4,\\
			\alpha'(e_1,e_4)&=\alpha_{142}e^2+\alpha_{143}e^3+(\alpha_{144}-\sigma_{11})e^4, &\alpha'(e_2,e_3)&=(\alpha_{232}+\sigma_{31})e^2+(\alpha_{233}-\sigma_{21})e^3+(\alpha_{243}-\alpha_{342})e^4,\\
			\alpha'(e_2,e_4)&=(\alpha_{242}+\sigma_{41})e^2+\alpha_{243}e^3+(\alpha_{244}-\sigma_{21})e^4, &
			\alpha'(e_3,e_4)&=\alpha_{342}e^2+(\alpha_{343}+\sigma_{41})e^3+(\alpha_{344}-\sigma_{31})e^4.
	\end{align*}
	Since $\sigma \in \mathcal{C}^1(\mathfrak{a}_1,\mathfrak{a}_1^\ast)$ can be chosen arbitrarily, we consider the following assumptions:
	\begin{align}\label{assmpp}
		\sigma_{11}&=\alpha_{122}, &\sigma_{31}&=-\alpha_{232}, &\sigma_{21}&=\alpha_{233}, &\sigma_{41}&=-\alpha_{242}.
	\end{align}
Hence, $\alpha'$ takes the form
\begin{align*}
	\alpha'(e_1,e_2)&=\alpha_{132}e^3+\alpha_{142}e^4, &\alpha'(e_1,e_3)&=\alpha_{132}e^2+(\alpha_{133}-\alpha_{122})e^3+\alpha_{143}e^4,\\
	\alpha'(e_1,e_4)&=\alpha_{142}e^2+\alpha_{143}e^3+(\alpha_{144}-\alpha_{122})e^4, &\alpha'(e_2,e_3)&=(\alpha_{243}-\alpha_{342})e^4,\\
	\alpha'(e_2,e_4)&=\alpha_{243}e^3+(\alpha_{244}-\alpha_{233})e^4, &
	\alpha'(e_3,e_4)&=\alpha_{342}e^2+(\alpha_{343}-\alpha_{242})e^3+(\alpha_{344}+\alpha_{232})e^4.
\end{align*}
In other words, we have shown that
\begin{align}
	H^2_\rho(\mathfrak{a}_1,\mathfrak{a}_1^\ast) = [\alpha'].
\end{align}
According to Proposition~\ref{cohooro}, the following map
\begin{align}
	\Psi: \G_{\nabla,\alpha} \longrightarrow \G_{\nabla,\alpha'}, \quad (x,\xi) \mapsto (x, \xi + \sigma(x))
\end{align}
defines an isomorphism of Lie algebras, where $\sigma$ is given under the assumptions in~\eqref{assmpp}.

It remains now to compute the Lie brackets associated with the cocycle $\alpha'$, using the relations~\eqref{Bra1} and~\eqref{Bra2}. Then, the Lie brackets of the Lie algebra $\G_{\nabla,\alpha'}$ are given by
\begin{align*}
	[e_1,e_2]&=\alpha_{132}e^3+\alpha_{142}e^4, &[e_1,e_3]&=\alpha_{132}e^2+(\alpha_{133}-\alpha_{122})e^3+\alpha_{143}e^4,\\
	[e_1,e_4]&=\alpha_{142}e^2+\alpha_{143}e^3+(\alpha_{144}-\alpha_{122})e^4, &[e_2,e_3]&=(\alpha_{243}-\alpha_{342})e^4,\\
	[e_2,e_4]&=\alpha_{243}e^3+(\alpha_{244}-\alpha_{233})e^4, &
	[e_3,e_4]&=\alpha_{342}e^2+(\alpha_{343}-\alpha_{242})e^3+(\alpha_{344}+\alpha_{232})e^4,\\
	[e_2,e^1]&=-e^2, &[e_3,e^1]&=-e^3, \\
	[e_4,e^1]&=-e^4,\quad \alpha_{ijk}\in\R.
\end{align*}
Now, using a change of variables, adopting the following notation
\begin{align*}
a&=\alpha_{132}, &b&=\alpha_{142}, &c&=\alpha_{133}, &d&=\alpha_{122}, &\lambda_1&=\alpha_{143}, &\lambda_2&=c+d-\alpha_{144},  \\
\lambda_3&=-\alpha_{342}, &\lambda_4&=-\alpha_{243}, &\eta_1&=\alpha_{244}-\alpha_{233}, &\eta_2&=\alpha_{343}-\alpha_{242}, &\eta_3&=\alpha_{232}+\alpha_{344}.
\end{align*}
and identifying the dual basis 
$\{e^1, e^2, e^3, e^4\}$ with $\{e_5, e_6, e_7, e_8\}$, the previous Lie algebra structures become
\begin{align*}
	[e_1,e_2]&=ae_7+be_8, &[e_1,e_3]&=ae_6+(c-d)e_7+\lambda_1 e_8, &[e_1,e_4]&=be_6+\lambda_1 e_7+(c-\lambda_2)e_8\\
	[e_2,e_3]&=(\lambda_3-\lambda_4)e_8, &[e_2,e_4]&=-\lambda_4 e_7+\eta_1 e_8, &[e_3,e_4]&=-\lambda_3 e_6+\eta_2 e_7+\eta_3 e_8,\\
	[e_2,e_5]&=-e_6, &[e_3,e_5]&=-e_7, &[e_4,e_5]&=-e_8.
\end{align*}
Here, $a,b,c,d \in \mathbb{R}$ and $\lambda_j, \eta_j \in \mathbb{R}$. Since 
$[\mathfrak{a}_1, \mathfrak{a}_1^\ast] \subset \mathfrak{a}_1^\ast$ and $\mathfrak{a}_1^\ast$ is an abelian ideal, it is clear that this algebra is $2$-step nilpotent. This algebra corresponds to the first algebra $\G_{8,1}$ in Table~\ref{Symbracket}.

For the remaining cases, the calculation is in fact the computation of the cohomology group $H^2_\rho(\h, \h^\ast)$, which is determined by the Lie brackets of each algebra presented in Table~\ref{Symbracket}. Considering the second algebra in Table~\ref{Symbracket}, $\G_{8,2}$, which is exactly the Lagrangian extension of the flat Lie algebra $(\mathfrak{a}_2, \nabla)$ given in Table~\ref{4flatcomplete}, its Lie brackets are given by
\begin{align*}
	[e_1,e_2]&=ae_7+be_8, &[e_1,e_3]&=ae_6+(c-d)e_7+\lambda_1 e_8, &[e_1,e_4]&=be_6+\lambda_1 e_7-(c+\lambda_2)e_8,\\
	[e_2,e_3]&=(\lambda_3-\lambda_4)e_8, &[e_2,e_4]&=-\lambda_4 e_7+\eta_1 e_8, &[e_3,e_4]&=-\lambda_3 e_6+\eta_2 e_7+\eta_3 e_8,\\
	[e_2,e_5]&=-e_6, &[e_3,e_5]&=-e_7, &[e_4,e_5]&=e_8.
\end{align*}
Here, $a,b,c,d\in\R$, and $\lambda_j,\eta_j\in\R$. It follows that
\begin{align}
	H^2_\rho(\mathfrak{a}_1, \mathfrak{a}_1^\ast) = [\alpha'],
\end{align}
where $\alpha'$ is defined by
\begin{align*}
	\alpha'(e_1,e_2)&=ae^3+be^4, &\alpha'(e_1,e_3)&=ae^2+(c-d)e^3+\lambda_1 e^4, &\alpha'(e_1,e_4)&=be^2+\lambda_1 e^3-(c+\lambda_2)e^4,\\
	\alpha'(e_2,e_3)&=(\lambda_3-\lambda_4)e^4, &\alpha'(e_2,e_4)&=-\lambda_4 e^3+\eta_1 e^4, &\alpha'(e_3,e_4)&=-\lambda_3 e^2+\eta_2 e^3+\eta_3 e^4.
\end{align*}
The proof is thus complete.

	\end{proof}

The following result shows the isomorphism classes of eight-dimensional symplectic nilpotent Lie algebras that possess a Lagrangian ideal. This is equivalent to a complete solution of the cotangent extension problem, as stated in the Introduction (Section~$\ref{sec1}$). Together with Proposition~\ref{isosymp}, this yields a refined classification. We also refer to $\cite{AEM}$ for an illustration of this calculation in low dimensions, as well as for the case of eight-dimensional symplectic non-solvable Lie algebras $\cite{AM}$.
 
\begin{pr}\label{prSymform}
The isomorphism classes of symplectic nilpotent Lie algebras with a Lagrangian ideal are summarized in Table~$\ref{Symform}$.
{\renewcommand*{\arraystretch}{1.5}
\captionof{table}{Isomorphism classes of symplectic nilpotent Lie algebras with Lagrangian ideal.}\label{Symform}

\begin{small} % Plus petit que small
\setlength{\tabcolsep}{10pt} 
\begin{longtable}{@{}cllllllc@{}} 
			\hline
		Algebra&Symplectic form&&Remarks\\
			\hline
$\G_{8,1}$&$\omega=e^{15}+e^{26}+e^{37}+e^{48}$&&\\
$\G_{8,2}$&$\omega=e^{15}+e^{26}+e^{37}+e^{48}$&&\\
$\G_{8,3}$&$\omega=\kappa_1 e^{12}+\kappa_2 e^{13}+e^{15}+\kappa_3 e^{23}+e^{26}+e^{37}+e^{48}$&&$\kappa_1,\kappa_2,\kappa_3\in\R$\\
$\G_{8,4}$&$\omega=e^{15}+e^{26}+e^{37}+e^{48}$&&\\
$\G_{8,5}$&$\omega=e^{15}+e^{26}+e^{37}+e^{48}$&&\\
$\G_{8,6}$&$\omega=\kappa e^{12}+ e^{15}+e^{26}+e^{37}+e^{48}$&&$\kappa\in\R$\\
$\G_{8,7}$&$\omega=\kappa e^{12}+ e^{15}+e^{26}+e^{37}+e^{48}$&&$\kappa\in\R$\\
$\G_{8,8}$&$\omega=\kappa e^{12}+ e^{15}+e^{26}+e^{37}+e^{48}$&&$\kappa\in\R$\\
$\G_{8,9}$&$\omega=\kappa e^{13}+ e^{15}+e^{26}+e^{37}+e^{48}$&&$\kappa\in\R$\\
$\G_{8,10}$&$\omega= e^{15}+e^{26}+e^{37}+e^{48}$&&\\
$\G_{8,11}$&$\omega=\kappa_1 e^{13}+ e^{15}+\kappa_2 e^{23}+e^{26}+\kappa_3 e^{34}+e^{37}+e^{48}$&&$\kappa_1,\kappa_2,(\kappa_3=-\lambda_4)\in\R$\\
$\G_{8,12}$&$\omega=\kappa_1 e^{13}+ e^{15}+\kappa_2 e^{23}+e^{26}+\kappa_3 e^{34}+e^{37}+e^{48}$&&$\kappa_1,\kappa_2,(\kappa_3=-\lambda_4)\in\R$\\
$\G_{8,13}$&$\omega=\kappa_1 e^{13}+ e^{15}+\kappa_2 e^{23}+e^{26}+\kappa_3 e^{34}+e^{37}+e^{48}$&&$\kappa_1,\kappa_2,(\kappa_3=d-\lambda_3)\in\R$\\
$\G_{8,14}$&$\omega=\kappa_1 e^{13}+ e^{15}+\kappa_2 e^{23}+e^{26}+\kappa_3 e^{34}+e^{37}+e^{48}$&&$\kappa_1,\kappa_2,(\kappa_3=d-\lambda_3)\in\R$\\
$\G_{8,15}$&$\omega= e^{15}+e^{26}+\kappa e^{34}+e^{37}+e^{48}$&&$(\kappa=-\lambda_2)\in\R$\\
$\G_{8,16}$&$\omega= e^{15}+e^{26}+\kappa e^{34}+e^{37}+e^{48}$&&$(\kappa=\lambda_2)\in\R$\\
$\G_{8,17}$&$\omega= e^{15}+e^{26}+\kappa e^{34}+e^{37}+e^{48}$&&$(\kappa=-\lambda_1)\in\R$\\
$\G_{8,18}$&$\omega= e^{15}+e^{26}+\kappa e^{34}+e^{37}+e^{48}$&&$(\kappa=-\lambda_2)\in\R$\\
$\G_{8,19}$&$\omega= e^{15}+e^{26}+\kappa e^{34}+e^{37}+e^{48}$&&$(\kappa=-\lambda_3)\in\R$\\
$\G_{8,20}$&$\omega=\kappa_1 e^{13}+ e^{15}+e^{26}+\kappa_2 e^{34}+e^{37}+e^{48}$&&$\kappa_1,(\kappa_2=-\lambda)\in\R$\\
$\G_{8,21}$&$\omega= e^{15}+e^{26}+\kappa e^{34}+e^{37}+e^{48}$&&$(\kappa=-\lambda_2)\in\R$\\
$\G_{8,22}$&$\omega= e^{15}+e^{26}+\kappa e^{34}+e^{37}+e^{48}$&&$(\kappa=-\lambda_2)\in\R$\\
$\G_{8,23}$&$\omega= e^{15}+e^{26}+\kappa e^{34}+e^{37}+e^{48}$&&$(\kappa=-\lambda_2)\in\R$\\
$\G_{8,24}$&$\omega= e^{15}+\kappa_1 e^{23}+e^{26}+\kappa_2 e^{34}+e^{37}+e^{48}$&&$\kappa_1,(\kappa_2=-\lambda)\in\R$\\
$\G_{8,25}$&$\omega=\kappa_1e^{13}+ e^{15}+e^{26}+\kappa_2 e^{34}+e^{37}+e^{48}$&&$(\kappa_1,\kappa_2=-\lambda_2)\in\R$\\
$\G_{8,26}$&$\omega= e^{15}+\kappa_1e^{23}+e^{26}+\kappa_2 e^{34}+e^{37}+e^{48}$&&$(\kappa_1,\kappa_2=d)\in\R$\\
$\G_{8,27}$&$\omega= e^{15}+\kappa_1e^{23}+e^{26}+\kappa_2 e^{34}+e^{37}+e^{48}$&&$(\kappa_1,\kappa_2=-d)\in\R$\\
$\G_{8,28}$&$\omega=\kappa_1e^{13}+ e^{15}+\kappa_2e^{23}+e^{26}+e^{37}+e^{48}$&&$\kappa_1,\kappa_2\in\R$\\
$\G_{8,29}$&$\omega=\kappa e^{13}+ e^{15}+e^{26}+e^{37}+e^{48}$&&$\kappa\in\R$\\
$\G_{8,30}$&$\omega= e^{15}+e^{26}+\kappa e^{34}+e^{37}+e^{48}$&&$(\kappa=c)\in\R$\\
$\G_{8,31}$&$\omega= \kappa e^{13}+e^{15}+e^{26}+e^{37}+e^{48}$&&$\kappa\in\R$\\
$\G_{8,32}$&$\omega= \kappa_1 e^{13}+\kappa_2 e^{14}+e^{15}+e^{26}+\kappa_3 e^{34}+e^{37}+e^{48}$&&$\kappa_1,\kappa_2,(\kappa_3=-d)\in\R$\\
$\G_{8,33}$&$\omega= \kappa_1 e^{13}+e^{15}+\kappa_2 e^{23}+e^{26}+\kappa_3 e^{34}+e^{37}+e^{48}$&&$\kappa_1,\kappa_2,(\kappa_3=-c)\in\R$\\
$\G_{8,34}$&$\omega= e^{15}+e^{26}+\kappa e^{34}+e^{37}+e^{48}$&&$(\kappa=d)\in\R$\\
$\G_{8,35}$&$\omega= e^{15}+\kappa_1 e^{24}+e^{26}+\kappa_2 e^{34}+e^{37}+e^{48}$&&$\kappa_1,(\kappa_2=-c)\in\R$\\
$\G_{8,36}$&$\omega= e^{15}+e^{26}+\kappa e^{34}+e^{37}+e^{48}$&&$(\kappa=-c)\in\R$\\
$\G_{8,37}$&$\omega= e^{15}+\kappa_1 e^{24}+e^{26}+\kappa_2 e^{34}+e^{37}+e^{48}$&&$\kappa_1,(\kappa_2=-c)\in\R$\\
$\G_{8,38}$&$\omega= e^{15}+\kappa_1 e^{24}+e^{26}+\kappa_2 e^{34}+e^{37}+e^{48}$&&$\kappa_1,(\kappa_2=-c)\in\R$\\
$\G_{8,39}$&$\omega= e^{15}+e^{26}+\kappa e^{34}+e^{37}+e^{48}$&&$(\kappa=-\lambda_4)\in\R$\\
$\G_{8,40}$&$\omega= e^{15}+e^{26}+\kappa e^{34}+e^{37}+e^{48}$&&$(\kappa=d-\lambda_4)\in\R$\\
$\G_{8,41}$&$\omega= e^{15}+e^{26}+\kappa e^{34}+e^{37}+e^{48}$&&$(\kappa=-\eta_1)\in\R$\\
$\G_{8,42}$&$\omega= e^{15}+e^{26}+\kappa e^{34}+e^{37}+e^{48}$&&$(\kappa=-\lambda_3)\in\R$\\
$\G_{8,43}$&$\omega= e^{15}+e^{26}+\kappa e^{34}+e^{37}+e^{48}$&&$(\kappa=-\lambda_4)\in\R$\\
$\G_{8,44}$&$\omega= e^{15}+e^{26}+\kappa e^{34}+e^{37}+e^{48}$&&$(\kappa=-\eta_1)\in\R$\\
$\G_{8,45}$&$\omega=\kappa_1 e^{13}+ e^{15}+e^{26}+\kappa_2 e^{34}+e^{37}+e^{48}$&&$\kappa_1,(\kappa_2=-\lambda_4)\in\R,~\kappa_1=0$ if $\mu\neq\frac{9}{4}$\\
$\G_{8,46}$&$\omega= e^{15}+e^{26}+\kappa e^{34}+e^{37}+e^{48}$&&$(\kappa=-\eta_1)\in\R$\\
$\G_{8,47}$&$\omega= e^{15}+e^{26}+\kappa e^{34}+e^{37}+e^{48}$&&$(\kappa=-\lambda_4)\in\R$\\
$\G_{8,48}$&$\omega= e^{15}+e^{26}+\kappa e^{34}+e^{37}+e^{48}$&&$(\kappa=-\lambda_4)\in\R$\\
$\G_{8,49}$&$\omega= e^{15}+e^{26}+\kappa e^{34}+e^{37}+e^{48}$&&$(\kappa=-\lambda_4)\in\R$\\
$\G_{8,50}$&$\omega= e^{15}+e^{26}+\kappa e^{34}+e^{37}+e^{48}$&&$(\kappa=\lambda_1)\in\R$\\
$\G_{8,51}$&$\omega= e^{15}+\kappa_1 e^{23}+e^{26}+\kappa_2 e^{34}+e^{37}+e^{48}$&&$\kappa_1,(\kappa_2=-\lambda_2)\in\R$\\
$\G_{8,52}$&$\omega= e^{15}+\kappa e^{24}+e^{26}+e^{37}+e^{48}$&&$\kappa\in\R$\\
$\G_{8,53}$&$\omega= e^{15}+\kappa_1 e^{23}+e^{26}+\kappa_2 e^{34}+e^{37}+e^{48}$&&$\kappa_1,(\kappa_2=-\lambda_2)\in\R$\\
$\G_{8,54}$&$\omega= e^{15}+e^{26}+\kappa e^{34}+e^{37}+e^{48}$&&$(\kappa=\lambda_1)\in\R$\\
$\G_{8,55}$&$\omega= e^{15}+e^{26}+e^{37}+e^{48}$&&\\
$\G_{8,56}$&$\omega= e^{15}+\kappa e^{23}+e^{26}+e^{37}+e^{48}$&&$(\kappa=-b)\in\R$\\
$\G_{8,57}$&$\omega=\kappa_1 e^{12}+ e^{15}+\kappa_2 e^{23}+e^{26}+\kappa_3e^{34}+e^{37}+e^{48}$&&$\kappa_1,(\kappa_2=-\frac{2a}{\mu+1}-b),\kappa_3\in\R,~\kappa_1=0$ if $\mu\neq-9$\\
$\G_{8,58}$&$\omega= e^{15}+\kappa e^{23}+e^{26}+e^{37}+e^{48}$&&$(\kappa=b-c)\in\R$\\
$\G_{8,59}$&$\omega=\kappa_1 e^{12}+ e^{15}+\kappa_2 e^{23}+e^{26}+\kappa_3e^{34}+e^{37}+e^{48}$&&$\kappa_1,(\kappa_2=\frac{a}{2(\mu+1)}-b),\kappa_3\in\R,~\kappa_1=0$ if $\mu\neq-9$\\
$\G_{8,60}$&$\omega= e^{15}+\kappa e^{23}+e^{26}+e^{37}+e^{48}$&&$(\kappa=b-c)\in\R$\\
$\G_{8,61}$&$\omega=\kappa_1 e^{12}+ e^{15}+\kappa_2 e^{23}+e^{26}+\kappa_3e^{34}+e^{37}+e^{48}$&&$\kappa_1,(\kappa_2=-\frac{2a}{\mu+1}-b),\kappa_3\in\R,~\kappa_1=0$ if $\mu\neq-9$\\
$\G_{8,62}$&$\omega= e^{15}+\kappa e^{23}+e^{26}+e^{37}+e^{48}$&&$(\kappa=b-c)\in\R$\\
$\G_{8,63}$&$\omega= e^{15}+\kappa e^{23}+e^{26}+e^{37}+e^{48}$&&$(\kappa=-c)\in\R$\\
$\G_{8,64}$&$\omega= e^{15}+\kappa e^{23}+e^{26}+e^{37}+e^{48}$&&$(\kappa=b-c)\in\R$\\
$\G_{8,65}$&$\omega=\kappa_1 e^{13}+ e^{15}+\kappa_2 e^{23}+e^{26}+e^{37}+e^{48}$&&$(\kappa_1=-c,\kappa_2=-\lambda_1)\in\R$\\
$\G_{8,66}$&$\omega=e^{15}+\kappa_1 e^{23}+e^{26}+\kappa_2 e^{34}+e^{37}+e^{48}$&&$(\kappa_1=c-d),\kappa_2\in\R$\\
$\G_{8,67}$&$\omega=e^{15}+\kappa_1 e^{23}+e^{26}+\kappa_2 e^{34}+e^{37}+e^{48}$&&$(\kappa_1=c-d),\kappa_2\in\R$\\
$\G_{8,68}$&$\omega=e^{15}+\kappa_1 e^{23}+e^{26}+\kappa_2 e^{34}+e^{37}+e^{48}$&&$(\kappa_1=-b),\kappa_2\in\R$\\
$\G_{8,69}$&$\omega=e^{15}+\kappa e^{23}+e^{26}+e^{37}+e^{48}$&&$(\kappa=b-c)$\\
$\G_{8,70}$&$\omega=e^{15}+\kappa_1 e^{23}+e^{26}+\kappa_2e^{34}+e^{37}+e^{48}$&&$(\kappa_1=b-c),\kappa_2\in\R,~\kappa_2=0$ if $\mu\neq-1$\\
$\G_{8,71}$&$\omega=e^{15}+\kappa e^{23}+e^{26}+e^{37}+e^{48}$&&$(\kappa=c-d)\in\R$\\
$\G_{8,72}$&$\omega=e^{15}+\kappa e^{23}+e^{26}+e^{37}+e^{48}$&&$(\kappa=c-d)\in\R$\\
$\G_{8,73}$&$\omega=e^{15}+\kappa e^{23}+e^{26}+e^{37}+e^{48}$&&$(\kappa=c-d-\mu b)\in\R$\\
$\G_{8,74}$&$\omega=\kappa_1 e^{13}+e^{15}+\kappa_2 e^{23}+e^{26}+e^{37}+e^{48}$&&$(\kappa_1=-b,\kappa_2=c+2a-\lambda)\in\R$\\
$\G_{8,75}$&$\omega=e^{15}+\kappa e^{23}+e^{26}+e^{37}+e^{48}$&&$(\kappa=b-c)\in\R$\\
$\G_{8,76}$&$\omega=e^{15}+\kappa e^{23}+e^{26}+e^{37}+e^{48}$&&$(\kappa=a-c)\in\R$\\
$\G_{8,77}$&$\omega=\kappa_1e^{13}+e^{15}+\kappa_2 e^{23}+e^{26}+e^{37}+e^{48}$&&$(\kappa_1=-\kappa_2=d)\in\R$\\
$\G_{8,78}$&$\omega=e^{15}+\kappa e^{23}+e^{26}+e^{37}+e^{48}$&&$(\kappa=b-c)\in\R$\\
$\G_{8,79}$&$\omega=e^{15}+\kappa e^{23}+e^{26}+e^{37}+e^{48}$&&$(\kappa=b-c)\in\R$\\
$\G_{8,80}$&$\omega=\kappa_1e^{13}+e^{15}+\kappa_2 e^{23}+e^{26}+e^{37}+e^{48}$&&$(\kappa_1=c,\kappa_2=d-\lambda)\in\R$\\
$\G_{8,81}$&$\omega=e^{15}+\kappa e^{23}+e^{26}+e^{37}+e^{48}$&&$(\kappa=b-c)\in\R$\\
$\G_{8,82}$&$\omega=e^{15}+\kappa e^{23}+e^{26}+e^{37}+e^{48}$&&$(\kappa=c-d)\in\R$\\
$\G_{8,83}$&$\omega=\kappa_1e^{13}+e^{15}+\kappa_2 e^{23}+e^{26}+\kappa_3 e^{34}+e^{37}+e^{48}$&&$(\kappa_1=c,\kappa_2=d-\lambda,\kappa_3)\in\R$,~$\kappa_3=0$ if $\mu_1=\frac{2}{3}(\mu^2-1)+\mu$\\
$\G_{8,84}$&$\omega=e^{15}+\kappa_2 e^{23}+e^{26}+\kappa_3 e^{34}+e^{37}+e^{48}$&&$(\kappa=b-c)\in\R$\\
$\G_{8,85}$&$\omega=\kappa_1e^{12}+e^{15}+\kappa_2 e^{23}+e^{26}+\kappa_3 e^{34}+\kappa_3e^{34}+e^{37}+e^{48}$&&$\kappa_1,\kappa_2,(\kappa_2=-c)\in\R$,~$\kappa_1=0$ if $\mu\neq-2$,~$\kappa_3=0$ if $\mu\neq\frac{1}{2}$ or $-1$\\
$\G_{8,86}$&$\omega=\kappa_1e^{12}+e^{15}+\kappa_2 e^{23}+e^{26}+e^{37}+e^{48}$&&$\kappa_1,(\kappa_2=b-c)\in\R$, $\kappa_1=0$ if $3\mu_1^2-3\mu_1-2\neq0$\\
$\G_{8,87}$&$\omega=e^{15}+\kappa e^{23}+e^{26}+e^{37}+e^{48}$&&$(\kappa=c-d)\in\R$\\
$\G_{8,88}$&$\omega=e^{15}+\kappa_1 e^{23}+e^{26}+\kappa_2e^{34}+e^{37}+e^{48}$&&$(\kappa=c-d),\kappa_2\in\R$,~$\kappa_2=0$ if $2\mu^2+3\mu+1\neq0$\\
$\G_{8,89}$&$\omega=\kappa_1e^{13}+e^{15}+\kappa_2 e^{23}+e^{26}+e^{37}+e^{48}$&&$(\kappa_1=b,\kappa_2=c-d)\in\R$\\
$\G_{8,90}$&$\omega=\kappa_1e^{13}+e^{15}+\kappa_2 e^{23}+e^{26}+e^{37}+e^{48}$&&$(\kappa_1=-c,\kappa_2=a-d)\in\R$\\
$\G_{8,91}$&$\omega=e^{15}+\kappa e^{23}+e^{26}+e^{37}+e^{48}$&&$(\kappa=c-\lambda)\in\R$\\
$\G_{8,92}$&$\omega=\kappa_1e^{12}+e^{15}+\kappa e^{23}+e^{26}+e^{37}+e^{48}$&&$\kappa_1,(\kappa_2=-\frac{\mu c+b(\mu+1)}{\mu})$, $\kappa_1=0$ if $\mu\neq-\frac{2}{3}$\\
$\G_{8,93}$&$\omega=e^{15}+\kappa e^{23}+e^{26}+e^{37}+e^{48}$&&$(\kappa=-d-\frac{8c}{5})\in\R$\\
$\G_{8,94}$&$\omega=\kappa_1e^{12}+e^{15}+\kappa_2 e^{23}+e^{26}+e^{37}+e^{48}$&&$\kappa_1,(\kappa_2=-\frac{\mu c+a+b(\mu+1)}{\mu})$, $\kappa_1=0$ if $\mu\neq-\frac{2}{3}$\\
$\G_{8,95}$&$\omega=e^{15}+\kappa e^{23}+e^{26}+e^{37}+e^{48}$&&$(\kappa=b-d-\frac{1}{5}(3a+8c))\in\R$
\\\hline
\end{longtable}
			\end{small}	
			}
\setcounter{table}{2}
\end{pr}
\begin{proof}
The general outline of the proof is as follows. We consider a flat Lie algebra $(\h, \nabla)$, selected from those described in Proposition~\ref{4flat} (Table~\ref{4flatcomplete}). This flat Lie algebra has a basis $B = \{e_1, e_2, e_3,e_4\}$. Let $(\G_\nabla, \alpha, \omega)$ be the Lagrangian extension of the flat Lie algebra $(\h, \nabla)$ with respect to $\alpha \in Z^2_{L,\rho}(\h, \h^\ast)$. Let $\sigma \in \mathcal{C}^1(\h, \h^\ast)$ and $\alpha' \in Z^2_{L,\rho}(\h, \h^\ast)$ be such that $\alpha' = \alpha - \partial_\rho \sigma$. In this case, $\G_{\nabla, \alpha} \cong \G_{\nabla, \alpha'}$, and we focus on the Lie algebra $\G_{\nabla, \alpha'}$ in what follows. This Lie algebra is isomorphic to a family of Lie algebras described in Proposition~\ref{prSymbracket} (Table~\ref{Symbracket}), and their Lie brackets are given by~(\ref{Bra1}) and~(\ref{Bra2}). By this procedure, we collect all Lie algebras that are isomorphic to $\G_{\nabla, \alpha'}$. This step is described in the proof of Proposition~\ref{prSymbracket}. We now proceed to classify symplectic forms on $\G_{\nabla, \alpha'}$ that admit $\h^\ast$ as a Lagrangian ideal.  This is achieved by establishing a one-to-one correspondence between Lagrangian extensions over flat Lie algebras and the cohomology group $H^2_{L,\rho}(\h, \h^\ast)$ (see Proposition~\ref{isosymp}). There is a natural (not necessarily injective) map from $H^2_{L,\nabla}(\h, \h^\ast)$ to $H^2_\rho(\h, \h^\ast)$. We illustrate this with two examples.

We consider the first example, which corresponds to the first algebra given in Proposition~$\ref{prSymbracket}$ and Table~$\ref{Symbracket}$, and which is treated in detail in the proof of Proposition~5. Let $(\mathfrak{a}_1, \nabla)$ (see Table~$\ref{4flatcomplete}$) be a flat Lie algebra, and let $(\G_{\nabla,\alpha}, \omega)$ be the Lagrangian extension of this flat Lie algebra. According to the proof of Proposition~$\ref{prSymbracket}$, $(\G_{\nabla,\alpha}, \omega)$ is isomorphic to the Lie algebra $(\G_{\nabla,\alpha'}, \omega)$ whose Lie brackets are given in the proof of the same proposition. The isomorphism between the Lie algebras $(\G_{\nabla,\alpha}, \omega)$ and $(\G_{\nabla,\alpha'}, \omega)$ is given by the following map:
\begin{align}
	\Psi: \G_{\nabla,\alpha} \longrightarrow \G_{\nabla,\alpha'}, \quad (x,\xi) \mapsto (x, \xi + \sigma(x)),
\end{align}
where $\sigma$ is given under the assumptions in~\eqref{assmpp}.

Next, we compute the Lagrangian extension cohomology group $H^2_{L,\rho}(\mathfrak{a}_1, \mathfrak{a}_1^\ast)$ for the flat Lie algebra $(\mathfrak{a}_1, \nabla)$. Let $\sigma_L \in \mathcal{C}^1_L(\mathfrak{a}_1, \mathfrak{a}_1^\ast)$ be given by
	\begin{align}
	\sigma(e_i)&=\mathop{\resizebox{1.3\width}{!}{$\sum$}}^4\limits_{j=1}\sigma_{ij}e^j,\quad\sigma_{ij}\in\R,
\end{align}
where $\sigma_{ij} = \sigma_{ji}$ for all $i,j = 1,2,3,4$. 
We then compute the differential of $\sigma$, denoted by $\partial_\rho \sigma$, with respect to the dual representation $\rho$ of $\nabla$. It is given by
\begin{align}
	\big(\partial_\rho \sigma\big)(x,y)
	= \rho(x)\sigma(y) - \rho(y)\sigma(x) - \sigma\big([x,y]\big),
	\qquad \text{for all } x,y \in \mathfrak{a}_1.
\end{align}
Consequently, we obtain
\begin{align*}
	\big(\partial_\rho\sigma_L\big)(e_1,e_2)&=\sigma_{11}e^2,&\big(\partial_\rho\sigma_L\big)(e_1,e_3)&=\sigma_{11}e^3,&\big(\partial_\rho\sigma_L\big)(e_1,e_4)&=\sigma_{11}e^4,\\
	\big(\partial_\rho\sigma_L\big)(e_2,e_3)&=-\sigma_{13}e^2+\sigma_{12}e^3,&\big(\partial_\rho\sigma_L\big)(e_2,e_4)&=-\sigma_{14}e^2+\sigma_{12}e^4,&\big(\partial_\rho\sigma_L\big)(e_3,e_4)&=-\sigma_{14}e^3+\sigma_{13}e^4.
\end{align*}
Define $\alpha'' :=\alpha-\partial_\rho \sigma_L$. It follows that
\begin{align}
	H^2_{L,\rho}(\mathfrak{a}_1, \mathfrak{a}_1^\ast) = [\alpha''].
\end{align}
According to Proposition~\ref{isosymp}, the isomorphism class of the symplectic form $\omega_{[\alpha]_L}$ with a Lagrangian ideal is uniquely determined by the cohomology class $[\alpha]_L \in H^2_{L,\rho}( \mathfrak{a}_1,  \mathfrak{a}_1^\ast)$. Moreover, the symplectic form can be written as
\begin{align*}
	\omega_{[\alpha]_L}(x + \xi, y)
	&= \Lambda(x,y) + \omega(x,\xi),
	\quad \text{for all } x,y \in  \mathfrak{a}_1,\ \xi \in  \mathfrak{a}_1^\ast,
\end{align*}
where
\[
\Lambda = (\sigma - \sigma_L) - {}^t(\sigma - \sigma_L)
\in \mathrm{Hom}\bigl(\wedge^2  \mathfrak{a}_1, \mathbb{R}\bigr),
\]
and where $\alpha' = \alpha - \partial_\rho \sigma$ and $\alpha'' = \alpha - \partial_\rho \sigma_L$ for some $\sigma \in C^1( \mathfrak{a}_1,  \mathfrak{a}_1^\ast)$ and some $\sigma_L \in C^1_L( \mathfrak{a}_1,  \mathfrak{a}_1^\ast)$. In this case, it is clear that $\Lambda = 0$. Thus, there is a single isomorphism class of symplectic forms, and consequently the map
\[
H^2_{L,\rho}(\mathfrak{a}_1, \mathfrak{a}_1^\ast)
\longrightarrow
H^2_{\rho}(\mathfrak{a}_1, \mathfrak{a}_1^\ast)
\]
is injective. By identifying the dual basis $\{e^1, e^2, e^3, e^4\}$ with $\{e_5, e_6, e_7, e_8\}$, the symplectic form defined by $\omega(\xi, x) = \xi(x)$ for all $x \in \mathfrak{a}_1$ and $\xi \in \mathfrak{a}_1^\ast$ becomes
\begin{align}
	\omega = e^{15} + e^{26} + e^{37} + e^{48}.
\end{align}
The second Lie algebra $\G_{8,2}$, which is a Lagrangian extension of the flat Lie algebra $(\mathfrak{a}_2, \nabla)$ given in Proposition~\ref{4flat} (Table~\ref{4flatcomplete}), also falls into the case where the map
\[
H^2_{L,\rho}(\mathfrak{a}_2, \mathfrak{a}_2^\ast)
\longrightarrow
H^2_{\rho}(\mathfrak{a}_2, \mathfrak{a}_2^\ast)
\]
is injective.

We now turn to the second example, considering the flat Lie algebra $( \mathfrak{a}_3, \nabla)$ described in Proposition~\ref{4flat} (Table~\ref{4flatcomplete}). The flat torsion-free connection $\nabla$ is defined by
\begin{align}
	\nabla_{e_2}e_4&=e_1,\quad\nabla_{e_3}e_3=e_1,\quad\nabla_{e_3}e_4=e_2,\quad\nabla_{e_4}e_3=e_2,\quad\nabla_{e_4}e_4=e_3.
\end{align}
	Next, we compute the ordinary Lie algebra cohomology group $H^2_\rho(\mathfrak{a}_3,\mathfrak{a}_3^\ast)$
associated with the representation $\rho$.  

Let $(\G_{\nabla,\alpha},\omega)$ be the Lagrangian extension of the flat Lie	algebra $(\mathfrak{a}_3,\nabla)$ with respect to
$\alpha \in Z^2_{\rho}(\mathfrak{a}_3,\mathfrak{a}_3^\ast)$.
Let $B^\ast=\{e^1,e^2,e^3,e^4\}$ be the dual basis of $B$, which generates
$\mathfrak{a}_3^\ast$.
Let $\alpha \in C^2(\mathfrak{a}_3,\mathfrak{a}_3^\ast)$ be given by
\begin{align}
	\alpha(e_i,e_j)&=	\mathop{\resizebox{1.3\width}{!}{$\sum$}}^4\limits_{k=1}\alpha_{ijk}e^k,\quad 1\leq i<j\leq 4,~~\alpha_{ijk}\in\R.
\end{align}
Therefore, $\alpha \in Z^2_{\rho}(\mathfrak{a}_3,\mathfrak{a}_3^\ast)$
if and only if $\alpha$ has the following form:
\begin{align}
	\alpha(e_i,e_j)&=	\mathop{\resizebox{1.3\width}{!}{$\sum$}}^4\limits_{k=1}\alpha_{ijk}e^k,\quad 1\leq i<j\leq 4,~~\alpha_{ijk}\in\R.
\end{align}
such that 
\begin{align*}
	\alpha_{121}&=0, &\alpha_{122}&=0, &\alpha_{123}&=\alpha_{141}, &\alpha_{131}&=0, &\alpha_{132}&=\alpha_{141},\\
	\alpha_{133}&=\alpha_{142}, &\alpha_{231}&=0, &\alpha_{232}&=\alpha_{241}, &\alpha_{233}&=\alpha_{242}-\alpha_{341}.
\end{align*}
Moreover, $\alpha\in Z^2_\rho( \mathfrak{a}_3, \mathfrak{a}_3^\ast)$ satisfies the Jacobi identity; that is,
\begin{align}
	\mathop{\resizebox{1.3\width}{!}{$\sum$}}\limits_{\mathrm{cycl}}\alpha(x,y)(z) = 0
\end{align}
for all $x,y,z \in \mathfrak{a}_3$, if and only if
\begin{align*}
	\alpha_{124}&=\alpha_{142}-\alpha_{241}, &\alpha_{134}&=\alpha_{143}-\alpha_{341}, &\alpha_{234}&=\alpha_{243}-\alpha_{342}.
\end{align*}
Let now $\sigma \in \mathcal{C}^1(\mathfrak{a}_3,\mathfrak{a}_3^\ast)$ be given, with respect to the basis $\{e^1, e^2, e^3, e^4\}$, by
\begin{align}
	\sigma(e_i)&=\mathop{\resizebox{1.3\width}{!}{$\sum$}}^4\limits_{j=1}\sigma_{ij}e^j,\quad\sigma_{ij}\in\R.
\end{align}
Next, we compute the differential of $\sigma$, denoted $\partial_\rho \sigma$, with respect to the dual representation $\rho$ of $\nabla$, that is,
\begin{align*}
	\big(\partial_\rho \sigma\big)(x,y)
	&= \rho(x)\sigma(y) - \rho(y)\sigma(x) - \sigma\big([x,y]\big)\\
	&=\rho(x)\sigma(y) - \rho(y)\sigma(x) 
\end{align*}
for all $ x,y \in \mathfrak{a}_3$. Therefore, we obtain
\begin{align*}
	\big(\partial_\rho\sigma\big)(e_1,e_2)&=\sigma_{11}e^4,&\big(\partial_\rho\sigma\big)(e_1,e_3)&=\sigma_{11}e^3+\sigma_{12}e^4,\\
	\big(\partial_\rho\sigma\big)(e_1,e_4)&=\sigma_{11}e^2+\sigma_{12}e^3+\sigma_{13}e^4,
	&\big(\partial_\rho\sigma\big)(e_2,e_3)&=\sigma_{21}e^3+(\sigma_{22}-\sigma_{31})e^4,\\
	\big(\partial_\rho\sigma\big)(e_2,e_4)&=\sigma_{21}e^2+\sigma_{22}e^3+(\sigma_{23}-\sigma_{41}) e^4,&\big(\partial_\rho\sigma\big)(e_3,e_4)&=\sigma_{31}e^2+(\sigma_{32}-\sigma_{41})e^3+(\sigma_{33}-\sigma_{42})e^4.
\end{align*}
	Define $\alpha' :=\alpha-  \partial_\rho \sigma$. Then $\alpha'$ is given by
\begin{align*}
	\alpha'(e_1,e_2)&=\alpha_{141}e^3+(\alpha_{142}-\alpha_{241}-\sigma_{11})e^4,\\ \alpha'(e_1,e_3)&=\alpha_{141}e^2+(\alpha_{142}-\sigma_{11})e^3+(\alpha_{143}-\alpha_{341}-\sigma_{12})e^4,\\
	\alpha'(e_1,e_4)&=\alpha_{141}e^1+(\alpha_{142}-\sigma_{11})e^2+(\alpha_{143}-\sigma_{12})e^3+(\alpha_{144}-\sigma_{13})e^4\\\alpha'(e_2,e_3)&=\alpha_{241}e^2+(\alpha_{242}-\alpha_{341}-\sigma_{21})e^2+(\alpha_{243}-\alpha_{342}+\sigma_{31}-\sigma_{22})e^4,\\
	\alpha'(e_2,e_4)&=\alpha_{241}e^1+(\alpha_{242}-\sigma_{21})e^2+(\alpha_{243}-\sigma_{22})e^3+(\alpha_{244}+\sigma_{41}-\sigma_{23})e^4, \\
	\alpha'(e_3,e_4)&=\alpha_{341}e^1+(\alpha_{342}-\sigma_{31})e^2+(\alpha_{343}+\sigma_{41}-\sigma_{32})e^3+(\alpha_{344}+\sigma_{42}-\sigma_{33})e^4.
\end{align*}
Since $\sigma \in \mathcal{C}^1(\mathfrak{a}_3,\mathfrak{a}_3^\ast)$ can be chosen arbitrarily, we consider the following assumptions:
\begin{align}\label{assmpp2}
	\begin{split}
	\sigma_{11}&=\alpha_{142}-\alpha_{241}, \quad\sigma_{12}=\alpha_{143}-\alpha_{341}, \quad\sigma_{13}=\alpha_{144}, \quad\sigma_{21}=\alpha_{242}-\alpha_{341},\\
	\sigma_{22}&=\alpha_{243}+\alpha_{342}-\alpha_{342}, \quad\sigma_{31}=\alpha_{342}, \quad\sigma_{23}=\alpha_{244}+\sigma_{41}, \quad\sigma_{32}=\alpha_{343}+\sigma_{41},\\
	\sigma_{33}&=\alpha_{344}+\sigma_{42}.
	\end{split}
\end{align}
Hence, $\alpha'$ takes the form
\begin{align*}
	\alpha'(e_1,e_2)&=\alpha_{141}e^3, &\alpha'(e_1,e_3)&=\alpha_{141}e^2+\alpha_{241}e^3,\\
\alpha'(e_1,e_4)&=\alpha_{141}e^1+\alpha_{241}e^2+\alpha_{341}e^3, &\alpha'(e_2,e_3)&=\alpha_{241}e^2,\\
	\alpha'(e_2,e_4)&=\alpha_{241}e^1+\alpha_{341}e^2, &
	\alpha'(e_3,e_4)&=\alpha_{341}e^1.
\end{align*}
In other words, we have shown that
\begin{align}
	H^2_\rho(\mathfrak{a}_3,\mathfrak{a}_3^\ast) = [\alpha'].
\end{align}
According to Proposition~\ref{cohooro}, the following map
\begin{align}
	\Psi: \G_{\nabla,\alpha} \longrightarrow \G_{\nabla,\alpha'}, \quad (x,\xi) \mapsto (x, \xi + \sigma(x))
\end{align}
defines an isomorphism of Lie algebras, where $\sigma$ is given under the assumptions in~\eqref{assmpp2}.

It remains now to compute the Lie brackets associated with the cocycle $\alpha'$, using the relations~\eqref{Bra1} and~\eqref{Bra2}. Then, the Lie brackets of the Lie algebra $\G_{\nabla,\alpha'}$ are given by
\begin{align*}
	[e_1,e_2]&=\alpha_{141}e^3, &[e_1,e_3]&=\alpha_{141}e^2+\alpha_{241}e^3,\\
	[e_1,e_4]&=\alpha_{141}e^1+\alpha_{241}e^2+\alpha_{341}e^3, &[e_2,e_3]&=\alpha_{241}e^2,\\
	[e_2,e_4]&=\alpha_{241}e^1+\alpha_{341}e^2, &
	[e_3,e_4]&=\alpha_{341}e^1.\\
	[e_2,e^1]&=-e^4, &[e_3,e^1]&=-e^3,\\
	[e_3,e^2]&=-e^4, &[e_4,e^1]&=-e^2,\\
	[e_4,e^2]&=-e^3, &[e_4,e^3]&=-e^4\\
	&\alpha_{ijk}\in\R.
\end{align*}
Now, using a change of variables, adopting the following notation
\begin{align*}
	\alpha_{141}&=a, &\alpha_{241}&=b, &\alpha_{341}&=c, &a,b,c\in\R&
\end{align*}
and identifying the dual basis 
$\{e^1, e^2, e^3, e^4\}$ with $\{e_5, e_6, e_7, e_8\}$, the previous Lie algebra structures become
\begin{align*}
	[e_1,e_2]&=ae_7, &[e_1,e_3]&=ae_6+be_7,\\
	[e_1,e_4]&=ae_5+be_6+ce_7, &[e_2,e_3]&=be_6,\\
	[e_2,e_4]&=be_5+ce_6, &
	[e_3,e_4]&=ce_5.\\
	[e_2,e_5]&=-e_8, &[e_3,e_5]&=-e_7,\\
	[e_3,e_6]&=-e_8, &[e_4,e_5]&=-e_6,\\
	[e_4,e_6]&=-e_7, &[e_4,e_7]&=-e_8\\
	&a,b,c\in\R.
\end{align*}
Here, $a,b,c,d \in \mathbb{R}$ and $\lambda_j, \eta_j \in \mathbb{R}$.  This algebra corresponds to the  algebra $\G_{8,3}$ in Table~\ref{Symbracket}.

We now turn to the computation of the Lagrangian extension cohomology group
$H^2_{L,\rho}(\mathfrak{a}_3, \mathfrak{a}_3^\ast)$ associated with the flat Lie algebra $(\mathfrak{a}_3, \nabla)$.
Let $\sigma_L \in \mathcal{C}^1_L(\mathfrak{a}_3, \mathfrak{a}_3^\ast)$ be a $1$-cochain defined by
\begin{align}
	\sigma(e_i) = \mathop{\resizebox{1.3\width}{!}{$\sum$}}^4\limits_{j=1} \sigma_{ij}\, e^j,
	\qquad \sigma_{ij} \in \mathbb{R}.
\end{align}
where $\sigma_{ij} = \sigma_{ji}$ for all $i,j \in \{1,2,3,4\}$.
We next compute the differential of $\sigma_L$, denoted by $\partial_\rho \sigma_L$, with respect to the dual representation~$\rho$ induced by~$\nabla$.
As a result, we obtain
\begin{align*}
	\big(\partial_\rho \sigma_L\big)(e_1,e_2) &= \sigma_{11} e^4, &
	\big(\partial_\rho \sigma_L\big)(e_1,e_3) &= \sigma_{11} e^3+\sigma_{12}e^4, \\
	\big(\partial_\rho \sigma_L\big)(e_1,e_4) &= \sigma_{11} e^2+\sigma_{12}e^3+\sigma_{13}e^4, &\big(\partial_\rho \sigma_L\big)(e_2,e_3) &= \sigma_{12} e^3 + (\sigma_{22}-\sigma_{13} )e^4, \\
	\big(\partial_\rho \sigma_L\big)(e_2,e_4) &= \sigma_{12} e^2 + \sigma_{22} e^3+(\sigma_{23}-\sigma_{14})e^4, &
	\big(\partial_\rho \sigma_L\big)(e_3,e_4) &= \sigma_{13} e^2 +( \sigma_{23}-\sigma_{14}) e^3+(\sigma_{33}-\sigma_{24})e^4.
\end{align*}
Setting $\alpha'' := \alpha - \partial_\rho \sigma_L$, it follows that
\begin{align}
	H^2_{L,\rho}(\mathfrak{a}_3, \mathfrak{a}_3^\ast) = [\alpha''].
\end{align}

By Proposition~\ref{isosymp}, the isomorphism class of a symplectic form $\omega_{[\alpha]_L}$ admitting a Lagrangian ideal is uniquely determined by the cohomology class
$[\alpha]_L \in H^2_{L,\rho}(\mathfrak{a}_3, \mathfrak{a}_3^\ast)$.
Furthermore, the associated symplectic form can be expressed as
\begin{align*}
	\omega_{[\alpha]_L}(x + \xi, y)
	&= \Lambda(x,y) + \omega(x,\xi),
	\quad \text{for all } x,y \in \mathfrak{a}_3,\ \xi \in \mathfrak{a}_3^\ast,
\end{align*}
where
\[
\Lambda = (\sigma - \sigma_L) - {}^t(\sigma - \sigma_L)
\in \mathrm{Hom}\bigl(\wedge^2 \mathfrak{a}_3, \mathbb{R}\bigr),
\]
and where $\alpha' = \alpha - \partial_\rho \sigma$ and $\alpha'' = \alpha - \partial_\rho \sigma_L$ for some
$\sigma \in C^1(\mathfrak{a}_3, \mathfrak{a}_3^\ast)$ and
$\sigma_L \in C^1_L(\mathfrak{a}_3, \mathfrak{a}_3^\ast)$. It remains to determine the explicit form of $\Lambda$, which is given by
\begin{align}
	\Lambda&=(\alpha_{143}-\alpha_{242})e^{12}+(\alpha_{144}-\alpha_{342})e^{13}+(\alpha_{244}-\alpha_{343})e^{23},\quad\alpha_{ijk}\in\R.
\end{align}
Consequently the natural map
\[
H^2_{L,\rho}(\mathfrak{a}_3, \mathfrak{a}_3^\ast)
\longrightarrow
H^2_{\rho}(\mathfrak{a}_3, \mathfrak{a}_3^\ast)
\]
is not injective. Let us consider the following change of variables:
\begin{align*}
	\alpha_{143} - \alpha_{242} &= \kappa_1, &
	\alpha_{144} - \alpha_{342} &= \kappa_2, &
	\alpha_{244} - \alpha_{343} &= \kappa_3.
\end{align*}
Moreover, by identifying the dual basis $\{e^1, e^2, e^3, e^4\}$ with $\{e_5, e_6, e_7, e_8\}$, the symplectic form defined by
$\omega(\xi, x) = \xi(x)$ for all $x \in \mathfrak{a}_1$ and $\xi \in \mathfrak{a}_1^\ast$
can be written as
\begin{align}
	\omega_{\kappa_1,\kappa_2,\kappa_3} =\kappa_1e^{12}+\kappa_2e^{13}+\kappa_3e^{23}+ e^{15} + e^{26} + e^{37} + e^{48},\quad\kappa_1,\kappa_2,\kappa_3\in\R.
\end{align}
Since $\kappa_1$, $\kappa_2$, and $\kappa_3$ are parameters of the cohomology group, it follows that for distinct choices
$(\kappa_1,\kappa_2,\kappa_3) \neq (\kappa_1',\kappa_2',\kappa_3')$,
the corresponding symplectic structures
$\omega_{\kappa_1,\kappa_2,\kappa_3}$ and
$\omega_{\kappa_1',\kappa_2',\kappa_3'}$
are not symplectically isomorphic.
This completes the proof.

\end{proof}

\begin{remark}
The values of the parameters indeed determine non-isomorphic cases. 
More precisely, the structure constants of the Lie algebra $($see Table~$\ref{Symbracket})$ and the coefficients of the symplectic structure $($see Table~$\ref{Symform})$ correspond, respectively, to elements of the cohomology groups 
$H^2_\rho(\h, \h^\ast)$ and $H^2_{L,\rho}(\h,\h^\ast)$. 
Therefore, two different sets of parameters give rise to non-isomorphic symplectic Lie algebras.
\end{remark}

\begin{theo}
Let $(G,\omega)$ be a connected, simply connected  symplectic nilpotent real Lie group of dimension $8$ admitting a Lagrangian normal subgroups. Then, $G$ is isomorphic to exactly one of the Lie groups whose Lie algebra appears in Table~$\ref{Symbracket}$. 
\end{theo}

A class of nilpotent Lie algebras that are particularly important are filiform nilpotent Lie algebras, cf.~\cite[Chapter~2, \S
6.2]{Oni}. We recall that a nilpotent Lie algebra $\G$ is filiform precisely when its nilpotency class $p$ satisfies $p = n - 1$, where $n = \dim \G$. The classification of eight-dimensional symplectic filiform complex Lie algebras, 
along with a description of all their symplectic structures, was established in \cite{GJK}. Let $(\G, \omega)$ be a $2n$-dimensional symplectic filiform Lie algebra. 
Then $(\G, \omega)$ admits a Lagrangian ideal, which is precisely the term 
$\mathcal{C}^{n-1}(\G)$ of its central descending series (see, Theorem 7.2.1, \cite{Bause}). Together with Propositions $\ref{prSymbracket}$ and  $\ref{prSymform}$, we therefore have:

\begin{theo}\label{Filifrom}
Let $(G, \omega)$ be a connected, simply connected symplectic filiform real Lie group of dimension~$8$. Then, $G$ is isomorphic to one of the Lie groups whose Lie algebras and corresponding symplectic forms are given by the following list:
\setcounter{table}{3}	
{\renewcommand*{\arraystretch}{1.5}
\captionof{table}{Eight-dimensional symplectic filiform real Lie algebras.}\label{SympFilifrom}
\begin{small} % Plus petit que small
\setlength{\tabcolsep}{10pt} % RÃ©duire l'espace entre colonnes
\begin{longtable}{@{}lllllllc@{}} % @{} rÃ©duit les marges
			\hline
		Algebra&Symplectic form&&&&Remarks\\\hline
$\G_{8,80}^{a,b,c,d,\lambda,\mu,\mu_1}$&$\omega=\kappa_1e^{13}+e^{15}+\kappa_2 e^{23}+e^{26}+e^{37}+e^{48}$&&$(\kappa_1=c,\kappa_2=d-\lambda)\in\R$&&$b(\mu_1-1)(\mu+1)\neq0$\\
$\G_{8,90}^{a,b,c,d,\lambda}$&$\omega=\kappa_1e^{13}+e^{15}+\kappa_2 e^{23}+e^{26}+e^{37}+e^{48}$&&$(\kappa_1=-c,\kappa_2=a-d)\in\R$&&$\lambda\neq0$\\
$\G_{8,91}^{a,b,c,d,\lambda}$&$\omega=e^{15}+\kappa e^{23}+e^{26}+e^{37}+e^{48}$&&$(\kappa=c-\lambda)\in\R$&&$b-d\neq0$\\
$\G_{8,93}^{a,b,c,d,\lambda}$&$\omega=e^{15}+\kappa e^{23}+e^{26}+e^{37}+e^{48}$&&$(\kappa=-d-\frac{8c}{5})\in\R$&&$b\neq0$\\
$\G_{8,95}^{a,b,c,d,\lambda}$&$\omega=e^{15}+\kappa e^{23}+e^{26}+e^{37}+e^{48}$&&$(\kappa=b-d-\frac{1}{5}(3a+8c))\in\R$&&$b\neq0$

\\\hline
\end{longtable}
			\end{small}	
			
			}
		
\end{theo}

\section{Appendix}\label{Appendix}
\subsection{Complete left-invariant affine structures on four-dimensional nilpotent Lie groups}
The following list presents the different classes of
four-dimensional geodesically complete flat Lie algebras; see~\cite{Hosh}.

As is well known, there are three 4-dimensional nilpotent Lie algebras up to
isomorphism (see, for example, \cite{chao} or \cite{dix}). We will denote these as follows.
\begin{enumerate}
	\item $\mathrm{A}\cong\R^4=\langle x_1,x_2,x_3,x_4|-\rangle$;\quad abelian
	\item $\mathrm{H}=\langle x_1,x_2,x_3,x_4|~~[x_2,x_3]=x_1\rangle$;\quad $Heisenberg\oplus\R$,
	\item $\mathrm{T}=\langle x_1,x_2,x_3,x_4|~~[x_2,x_3]=x_1,~~[x_3,x_4]=x_2\rangle$;
\end{enumerate}
		\begin{align*}
	1.\quad e_1\cdot e_4&=e_2, &e_2\cdot e_3&=e_1, &e_2\cdot e_4&=-e_3, &e_3\cdot e_3&=e_2  &&  &(\mathrm{T})&\\
	 e_4\cdot e_1&=e_2 && && && && &&\\
	2.\quad e_1\cdot e_4&=e_2, &e_2\cdot e_3&=e_1, &e_2\cdot e_4&=-e_3, &e_3\cdot e_3&=e_2,  &&  &(\mathrm{T})&\\
	e_4\cdot e_1&=e_2, &e_4\cdot e_4&=e_1 && && && &&\\
	3.\quad e_2\cdot e_2&=e_1, &e_3\cdot e_3&=e_1, &e_4\cdot e_4&=e_1  && &&  &(\mathrm{A})&\\
	4.\quad e_2\cdot e_2&=e_1, &e_3\cdot e_3&=e_1, &e_4\cdot e_4&=-e_1 && &&  &(\mathrm{A})&\\
	5.\quad e_2\cdot e_2&=e_1, &e_2\cdot e_3&=e_1, &e_3\cdot e_2&=-e_1, &e_3\cdot e_4&=e_1,   &&  &(\mathrm{H})&\\
	e_4\cdot e_3&=e_1 && && && && &&\\
	6.\quad e_2\cdot e_3&=e_1, &e_3\cdot e_2&=-e_1,  &e_3\cdot e_4&=e_1, &e_4\cdot e_3&=e_1  &&  &(\mathrm{H})&\\
	7.\quad e_2\cdot e_2&=e_1, &e_2\cdot e_3&=e_1, &e_3\cdot e_2&=-e_1, &e_3\cdot e_3&=t\,e_1,   &t\in\R&  &(\mathrm{H})&\\
	e_4\cdot e_4&=e_1 && && && && &&\\
	8.\quad e_2\cdot e_2&=e_1, &e_2\cdot e_3&=e_1, &e_3\cdot e_3&=t\,e_1, &e_4\cdot e_4&=-e_1,  &t\in\R^+&  &(\mathrm{H})&\\
	9.\quad e_2\cdot e_3&=e_1, &e_3\cdot e_3&=e_2, &e_4\cdot e_4&=e_2 &&  && &(\mathrm{H})&\\
	10.\quad e_2\cdot e_3&=e_1, &e_3\cdot e_3&=e_2, &e_4\cdot e_4&=e_1+e_2 &&  &&  &(\mathrm{H})&\\
	11.\quad e_2\cdot e_3&=e_1, &e_3\cdot e_3&=e_2, &e_3\cdot e_4&=e_1, &e_4\cdot e_4&=t\, e_1+e_2,  &t\in\R^+&  &(\mathrm{H})&&\\
	12.\quad e_2\cdot e_3&=e_1, &e_3\cdot e_3&=e_2, &e_4\cdot e_4&=-e_2, &&  &&  &(\mathrm{H})&\\
	13.\quad e_2\cdot e_3&=e_1, &e_2\cdot e_4&=e_1, &e_3\cdot e_3&=e_2, &e_4\cdot e_4&=-e_2  &&   &(\mathrm{H})&\\
	14.\quad e_2\cdot e_3&=e_1, &e_2\cdot e_4&=e_1, &e_3\cdot e_4&=e_1, &e_3\cdot e_3&=e_2,   &&  &(\mathrm{H})&\\
	e_4\cdot e_4&=-e_2 && && && && &&\\
	15.\quad e_2\cdot e_3&=e_1, &e_2\cdot e_4&=e_1,  &e_3\cdot e_3&=e_2, &e_3\cdot e_4&=e_1,    &&  &(\mathrm{H})&\\
	e_4\cdot e_4&=e_1-e_2  && && && && &&\\
	16.\quad e_2\cdot e_3&=e_1, &e_3\cdot e_3&=e_2, &e_4\cdot e_4&=e_1-e_2  &&  &&  &(\mathrm{H})&\\
	17.\quad e_2\cdot e_3&=e_1, &e_3\cdot e_3&=e_2, &e_3\cdot e_4&=e_1,  &e_4\cdot e_4&=t\,e_1-e_2  &t\in\R&  &(\mathrm{H})&\\
	18.\quad e_2\cdot e_3&=e_1, &e_2\cdot e_4&=e_1, &e_3\cdot e_3&=e_2,  &e_3\cdot e_4&=e_2,   & &  &(\mathrm{T})&\\
	e_4\cdot e_2&=-2\,e_1, &e_4\cdot e_3&=-e_2 &e_4\cdot e_4&=\lambda\,e_2 &&  &\lambda\in\R& &&\\
	19.\quad e_2\cdot e_3&=e_1, &e_2\cdot e_4&=e_1, &e_3\cdot e_3&=e_2,  &e_3\cdot e_4&=e_2, & &  &(\mathrm{T})&\\
	 e_4\cdot e_2&=-2\,e_1, &e_4\cdot e_3&=e_1-e_2, & e_4\cdot e_4&=\lambda\,e_2&&  &\lambda\in\R& &&\\
	20.\quad e_2\cdot e_3&=e_1, &e_2\cdot e_4&=e_1, &e_3\cdot e_3&=e_1+e_2,  &e_3\cdot e_4&=e_2,   & &  &(\mathrm{T})&\\
	e_4\cdot e_2&=-2\,e_1, &e_4\cdot e_3&=t\,e_1-e_2, &	e_4\cdot e_4&=\lambda\,e_2 && &\lambda,t\in\R& &&\\
	21.\quad e_2\cdot e_3&=e_1, &e_3\cdot e_2&=-2\,e_1, &e_3\cdot e_4&=e_2,  &e_4\cdot e_3&=-e_2  & &  &(\mathrm{T})&\\
	22.\quad e_2\cdot e_3&=e_1, &e_3\cdot e_2&=-2\,e_1, &e_3\cdot e_4&=e_2,  &e_4\cdot e_3&=-e_2,   & &  &(\mathrm{T})&\\
	e_4\cdot e_4&=e_1 && && && && &&\\
	23.\quad e_2\cdot e_3&=e_1, &e_2\cdot e_4&=e_1, &e_3\cdot e_2&=-2\,e_1,  &e_3\cdot e_4&=e_2,   & &  &(\mathrm{T})&\\
	e_4\cdot e_2&=-2\,e_1, &e_4\cdot e_3&=-e_2, &e_4\cdot e_4&=e_1 && && && \\
	24.\quad e_2\cdot e_3&=e_1, &e_4\cdot e_4&=e_2,  && && & &  &(\mathrm{H})&\\
	25.\quad e_2\cdot e_3&=e_1, &e_3\cdot e_4&=e_1, &e_4\cdot e_4&=e_2 &&   & &  &(\mathrm{H})&\\
	26.\quad e_2\cdot e_3&=e_1, &e_4\cdot e_3&=e_1, &e_4\cdot e_4&=e_2 &&   & &  &(\mathrm{H})&\\
	27.\quad e_2\cdot e_3&=e_1, &e_3\cdot e_4&=e_1, &e_4\cdot e_3&=e_1, &e_4\cdot e_4&=e_2  & &  &(\mathrm{H})&\\
	28.\quad e_2\cdot e_4&=e_1, &e_3\cdot e_3&=e_1, &e_4\cdot e_4&=e_2 &&   & &  &(\mathrm{H})&\\
	29.\quad e_2\cdot e_4&=e_1, &e_4\cdot e_3&=e_1, &e_4\cdot e_4&=e_2 &&  & &  &(\mathrm{H})&\\
	30.\quad e_2\cdot e_3&=t\,e_1, &e_3\cdot e_3&=e_1, &e_4\cdot e_2&=e_1, &e_4\cdot e_4&=e_2   &t\in\R &  &(\mathrm{H} ~\text{if}~ t=1)&\\
	& && && && && &(\mathrm{T} ~\text{if}~ t\neq1)&\\
	31.\quad e_2\cdot e_3&=t\,e_1, &e_3\cdot e_4&=e_1, &e_4\cdot e_2&=e_1, &e_4\cdot e_4&=e_2   &t\in\R &  &(\mathrm{H})&\\
	32.\quad e_3\cdot e_2&=e_1, &e_3\cdot e_4&=e_2, &e_4\cdot e_4&=e_3 &&   & &  &(\mathrm{T})&
\end{align*}
\begin{align*}
		33.\quad e_3\cdot e_2&=e_1, &e_3\cdot e_4&=e_2, &e_4\cdot e_3&=e_1, &e_4\cdot e_4&=e_3   & &  &(\mathrm{T})&\\
		34.\quad e_3\cdot e_2&=e_1, &e_3\cdot e_4&=e_2, &e_4\cdot e_3&=-e_1, &e_4\cdot e_4&=e_3  & &  &(\mathrm{T})&\\
	35.\quad e_2\cdot e_4&=-e_1, &e_3\cdot e_4&=e_2, &e_4\cdot e_2&=e_1, &e_4\cdot e_4&=e_3  & &  &(\mathrm{T})&\\
	36.\quad e_2\cdot e_4&=2\,e_1, &e_3\cdot e_3&=e_1, &e_3\cdot e_4&=e_2, &e_4\cdot e_2&=-e_1,   & &  &(\mathrm{T})&\\
	e_4\cdot e_3&=e_1, &e_4\cdot e_4&=e_3 && && && &&\\
		37.\quad e_2\cdot e_4&=(1-t)e_1, &e_3\cdot e_3&=e_1, &e_3\cdot e_4&=e_2, &e_4\cdot e_2&=t\,e_1,    &t\in\R &  &(\mathrm{H}~\text{if}~t=\tfrac{1}{2})&\\
		e_4\cdot e_4&=e_3 && && && && &(\mathrm{T}~\text{if}~t\neq\tfrac{1}{2})&\\
		38.\quad e_2\cdot e_4&=e_1, &e_3\cdot e_2&=e_1, &e_3\cdot e_3&=e_1, &e_3\cdot e_4&=e_2   &t\in\R &  &(\mathrm{T})&\\
		  e_4\cdot e_3&=t\,e_1&e_4\cdot e_4&=e_3 &&  && && && && 	\\
		  \\
		  39.\quad e_3\cdot e_3&=e_1, &e_3\cdot e_4&=e_2, &e_4\cdot e_2&=e_1, &e_4\cdot e_3&=\mu\,e_2   &\mu\in\R &  &&\\
		  e_4\cdot e_4&=e_3 &\text{Cases}& &\mu=1&:(\mathrm{H})  &\mu\neq1&:(\mathrm{T}) && &&\\
		  40.\quad e_2\cdot e_4&=e_1, &e_3\cdot e_3&=\tfrac{1+\mu}{2} e_1, &e_3\cdot e_4&=e_2, &e_4\cdot e_2&=\tfrac{3\mu-1}{2} e_1   &\mu\in\R &  &&\\
		  e_4\cdot e_3&=e_1+\mu\,e_2, &e_4\cdot e_4&=e_3 &\text{Cases}& &\mu=1&:(\mathrm{H})  &\mu\neq1&:(\mathrm{T}) &&\\
		   41.\quad e_2\cdot e_4&=e_1, &e_3\cdot e_3&=t\,e_1, &e_3\cdot e_4&=e_2, &e_4\cdot e_2&=(t-1+\mu) e_1   &\mu,t\in\R &  &&\\
		   e_4\cdot e_3&=\mu\,e_2, &e_4\cdot e_4&=e_3 &t&\neq\tfrac{1+\mu}{2} && && &&\\
		   \text{Cases}& &\mu=1& &(\mathrm{A})& ~\text{if}~ t=1, &(\mathrm{H})& ~\text{if}~ t\neq 1 && &&\\
		   & &\mu\neq 1& &(\mathrm{H})& ~\text{if}~ t=2-\mu, &(\mathrm{T})& ~\text{if}~ t\neq 1 && &&\\
		   42.\quad e_4\cdot e_2&=e_1, &e_4\cdot e_3&=e_2, &e_4\cdot e_4&=e_3 && &&  &(\mathrm{T})&\\
		   43.\quad e_3\cdot e_4&=e_1, &e_4\cdot e_2&=e_1, &e_4\cdot e_3&=e_2, &e_4\cdot e_4&=e_3  &&  &(\mathrm{T})&\\
		   44.\quad e_2\cdot e_4&=e_1, &e_3\cdot e_3&=-e_1, &e_4\cdot e_2&=t\,e_1, &e_4\cdot e_3&=e_2  &t\in\R&  &(\mathrm{H}~\text{if}~t=1)&\\
		   e_4\cdot e_4&=e_3 && && && && &(\mathrm{T}~\text{if}~t\neq 1)&\\
		   45.\quad e_2\cdot e_4&=e_1, &e_3\cdot e_3&=-e_1, &e_3\cdot e_4&=e_1, &e_4\cdot e_2&=t\,e_1  &t\in\R&  &(\mathrm{H}~\text{if}~t=1)&\\
		   e_4\cdot e_3&=e_2 &e_4\cdot e_4&=e_3 && && && &(\mathrm{T}~\text{if}~t\neq 1)&\\
		   46.\quad e_3\cdot e_4&=e_2, &e_4\cdot e_3&=-e_2 && &&  &&  &(\mathrm{H})&\\
		   47.\quad e_3\cdot e_3&=e_1, &e_3\cdot e_4&=e_1, &e_4\cdot e_3&=-e_1 &&  &&  &(\mathrm{H})&\\
		   48.\quad e_3\cdot e_3&=e_1, &e_3\cdot e_4&=e_2, &e_4\cdot e_3&=-e_2 &&  &&  &(\mathrm{H})&\\
		   49.\quad e_3\cdot e_3&=e_1, &e_3\cdot e_4&=e_2, &e_4\cdot e_3&=-e_2 &e_4\cdot e_4&=e_1  &&  &(\mathrm{H})&\\
		   50.\quad e_3\cdot e_3&=e_1, &e_3\cdot e_4&=e_2, &e_4\cdot e_3&=-e_2 &e_4\cdot e_4&=-e_1  &&  &(\mathrm{H})&\\
		   51.\quad e_3\cdot e_3&=e_1, &e_3\cdot e_4&=t\,e_1, &e_4\cdot e_3&=-t\,e_1 &e_4\cdot e_4&=e_1  &t\in\R^+&  &(\mathrm{A}~\text{if}~t=0)&\\
		   & && && && && &(\mathrm{H}~\text{if}~t>0)&\\
		   52.\quad e_3\cdot e_3&=e_1, &e_3\cdot e_4&=t\,e_1, &e_4\cdot e_3&=-t\,e_1 &e_4\cdot e_4&=-e_1  &t\in\R^+&  &(\mathrm{A}~\text{if}~t=0)&\\
		   & && && && && &(\mathrm{H}~\text{if}~t>0)&\\
		   53.\quad e_3\cdot e_3&=e_1, &e_3\cdot e_4&=(1+t)e_2, &e_4\cdot e_3&=(1-t)e_2 &e_4\cdot e_4&=e_1  &t\in\R^+&  &(\mathrm{A}~\text{if}~t=0)&\\
		   & && && && && &(\mathrm{H}~\text{if}~t>0)&\\
		    54.\quad e_3\cdot e_3&=e_1, &e_4\cdot e_4&=e_2 && &&  &&  &(\mathrm{A})&\\
		    55.\quad e_3\cdot e_3&=e_1, &e_3\cdot e_4&=e_1+t\,e_2 &e_4\cdot e_3&=-e_1-t\,e_2, &e_4\cdot e_4&=e_2  &t\in\R&  &(\mathrm{H})&\\
		    56.\quad e_3\cdot e_3&=e_1, &e_3\cdot e_4&=e_1+e_2, &e_4\cdot e_3&=-e_1+e_2 && &&  &(\mathrm{H})&\\
		    57.\quad e_3\cdot e_3&=e_1, &e_3\cdot e_4&=(1+t)e_2, &e_4\cdot e_3&=(1-t)e_2&& &t\in\R^+&  &(\mathrm{A}~\text{if}~t=0)&\\
		    & && && && && &(\mathrm{H}~\text{if}~t>0)&\\
		    58.\quad e_4\cdot e_3&=e_1, &e_4\cdot e_4&=e_3 && && &&  &(\mathrm{H})&\\
		     59.\quad e_3\cdot e_4&=e_1, &e_4\cdot e_3&=e_2, &e_4\cdot e_4&=e_3 && &&  &(\mathrm{H})&\\
		     60.\quad e_3\cdot e_4&=e_1, &e_4\cdot e_3&=te_1, &e_4\cdot e_4&=e_3 && &t\in\R&  &(\mathrm{A}~\text{if}~t=1)&\\
		     & && && && && &(\mathrm{H}~\text{if}~t\neq1)&
	\end{align*}
	\begin{align*}
			61.\quad e_4\cdot e_4&=e_1, && && &   & &  &(\mathrm{A})&\\
			62.\quad e_i\cdot e_j&=0 &\text{for all}~i,j=1,2,3,4& && &   & &  &(\mathrm{A})&\\
	\end{align*}
\subsection{Identification of Isomorphic Flat Lie Algebras}
We now proceed to eliminate those Lie algebras that are isomorphic, based on the list provided above. As noted above Definition~\ref{Defofnabla}, the flat torsion-free connections labeled 
$\mathbf{15}$, $\mathbf{23}$, $\mathbf{26}$, and $\mathbf{25}$ in \textsc{\cite{Hosh}} 
are, in fact, isomorphic to those represented by $\mathbf{13}$, $\mathbf{22}$, 
$\mathbf{24}$, and $\mathbf{24}$, respectively. 
The remaining connections are not isomorphic under the equivalence relation~$(\ref{relationequiva})$.

We now exhibit, for each case, the explicit Lie algebra isomorphism. We now examine the flat Lie algebra labeledy $\mathbf{13}$, whose associated flat torsion-free connection is given by
\begin{align}
	\nabla_{e_2}e_3&=e_1,\quad\nabla_{e_2}e_4=e_1,\quad\nabla_{e_3}e_3=e_2,\quad\nabla_{e_4}e_4=-e_2.
\end{align}
The underlying Lie algebra induced by this connection has the following nonvanishing Lie brackets:
\begin{align}
	\mathrm{H}_{13} : [e_2,e_3]=e_1,\quad [e_2,e_4]=e_1.
\end{align}
It is clear that the Lie algebra $\mathrm{H}_{13}$ is isomorphic to the Heisenberg Lie algebra $\ell$. 
One easily verifies that the map
\[
\Psi : \mathrm{H}_{13} \longrightarrow \ell
\]
defined by
\[
e_1 \mapsto -e_3, \quad e_2 \mapsto -e_1, \quad e_3 \mapsto  e_2-e_4, \quad e_4 \mapsto e_2
\]
is a Lie algebra isomorphism. It therefore remains to recover the corresponding flat torsion-free connection on $\ell$ by applying Definition~\ref{Defofnabla}. 
We then obtain
\begin{align}\label{H12}
	\nabla_{e_1} e_2 &= e_3, & \nabla_{e_2} e_2 &= e_1, & \nabla_{e_2} e_4 &= e_1, & \nabla_{e_4} e_2 &= e_1.
\end{align}
This flat torsion-free connection is listed in Table~\ref{4flatcomplete} under the algebra $\ell_9$. Consider now the flat torsion-free connection labeled by $\mathrm{H}_{15}$, which is given by
\begin{align}
	\nabla_{e_2} e_3 &= e_1, \quad \nabla_{e_2} e_4 = e_1,\quad\nabla_{e_3}e_3=e_2,\quad \nabla_{e_3} e_4=e_1,\quad\nabla_{e_4}e_4=e_1-e_2.
\end{align}
The underlying Lie algebra associated with this connection has the following nonvanishing bracket:
\begin{align}
	\mathrm{H}_{15} : [e_2,e_3] = e_1,\quad [e_2,e_4] = e_1,\quad [e_3,e_4] = e_1.
\end{align}  
It is straightforward to see that $\mathrm{H}_{15}$ is isomorphic to the Heisenberg Lie algebra $\ell$.  
Indeed, the map
\[
\Psi : \mathrm{H}_{15} \longrightarrow \ell
\]
defined by
\[
e_1 \mapsto e_3, \quad e_2 \mapsto e_1+e_3, \quad e_3 \mapsto  e_1+e_2+e_3, \quad e_4 \mapsto e_2-e_4
\]
provides a Lie algebra isomorphism.  
It therefore remains to determine the corresponding flat torsion-free connection on $\ell$ via Definition~\ref{Defofnabla}. We then obtain
\begin{align}\label{H15}
	\nabla_{e_1} e_2 &= e_3, & \nabla_{e_2} e_2 &= e_1, & \nabla_{e_2} e_4 &= e_1, & \nabla_{e_4} e_2 &= e_1.
\end{align}
which, in fact, coincides with the connection defined in (\ref{H12}).

Consider the flat Lie algebra labeled~$22$ (see the list above). Its associated flat, torsion-free connection is given by
\begin{align}\label{H22}
	\nabla_{e_2} e_3 &= e_1, & \nabla_{e_3} e_2 &= -2\,e_1, & \nabla_{e_3} e_4 &= e_2, & \nabla_{e_4} e_3 &= -e_2, & \nabla_{e_4} e_4 &= e_1.
\end{align}
The underlying Lie algebra associated with this connection has the following nonvanishing Lie brackets:
\begin{align}
	\mathrm{H}_{22} :\quad 
	[e_2,e_3] = 3\,e_1,\qquad 
	[e_3,e_4] = 2\,e_1.
\end{align}
It is straightforward to verify that $\mathrm{H}_{22}$ is isomorphic to the  Lie algebra~$\mathfrak{t}$.
Indeed, the linear map
\[
\Psi : \mathrm{H}_{22} \longrightarrow \mathfrak{t}
\]
defined by
\[
e_1 \mapsto \tfrac{1}{36} e_3, \qquad
e_2 \mapsto -\tfrac{1}{12}  e_2, \qquad
e_3 \mapsto e_4, \qquad
e_4 \mapsto -\tfrac{1}{6}  e_1
\]
is a Lie algebra isomorphism.  It therefore remains to recover the corresponding flat, torsion-free connection on~$\mathfrak{t}$ by applying Definition~\ref{Defofnabla}.
We then obtain
\begin{align}\label{H22}
	\nabla_{e_1} e_1 &= e_3, &
	\nabla_{e_1} e_4 &=-\tfrac{1}{2} e_2, &
	\nabla_{e_2} e_4 &= -\tfrac{1}{2} e_3, &
	\nabla_{e_4} e_1 &= \tfrac{1}{2} e_2,&
	\nabla_{e_4} e_2 &= \tfrac{2}{3} e_3.&
\end{align}
This flat, torsion-free connection is listed in Table~\ref{4flatcomplete} under the Lie algebra~$\mathfrak{t}_7$.

Consider now the flat, torsion-free connection labeled by~$23$, which is given by
\begin{align}
	\nabla_{e_2} e_3 &= e_1, \quad
	\nabla_{e_2} e_4 = e_1, \quad
	\nabla_{e_3} e_2 =-2\,e_1, \quad
	\nabla_{e_3} e_4 = e_2, \quad
	\nabla_{e_4} e_2 =-2\,e_1\quad
	\nabla_{e_4} e_3 =-e_2\quad
	\nabla_{e_4} e_4 =e_1.
\end{align}
The underlying Lie algebra associated with this connection has the following nonvanishing Lie brackets:
\begin{align}
	\mathrm{H}_{23} :\quad
	[e_2,e_3] = 3\,e_1, \qquad
	[e_2,e_4] = 3\,e_1, \qquad
	[e_3,e_4] = 2\,e_2.
\end{align}
One readily verifies that the Lie algebra $\mathrm{H}_{23}$ is isomorphic to the  Lie algebra~$\mathfrak{t}$.
Indeed, the linear map
\[
\Psi : \mathrm{H}_{23} \longrightarrow \mathfrak{t}
\]
defined by
\[
e_1 \mapsto \tfrac{1}{36} e_3, \qquad
e_2 \mapsto -\tfrac{1}{12}  e_2, \qquad
e_3 \mapsto e_4, \qquad
e_4 \mapsto -\tfrac{1}{6}  e_1+e_4
\]
is a Lie algebra isomorphism.
Applying Definition~\ref{Defofnabla}, we recover the induced flat, torsion-free connection on~$\mathfrak{t}$, which is given by
\begin{align}
	\nabla_{e_1} e_1 &= e_3, &
	\nabla_{e_1} e_4 &=-\tfrac{1}{2} e_2, &
	\nabla_{e_2} e_4 &= -\tfrac{1}{2} e_3, &
	\nabla_{e_4} e_1 &= \tfrac{1}{2} e_2,&
	\nabla_{e_4} e_2 &= \tfrac{2}{3} e_3.&
\end{align}
This connection coincides with the one previously defined in~\eqref{H22}.

Let us consider the flat Lie algebra labeled~$24$, whose associated flat, torsion-free connection is given by

\begin{align}
	\nabla_{e_2} e_3 &= e_1, \qquad
	\nabla_{e_4} e_4 =e_1. 
\end{align}
The underlying Lie algebra corresponding to this connection has the following nonvanishing Lie brackets:
\begin{align}
	\mathrm{H}_{24} :\quad 
	[e_2,e_3] = e_1.
\end{align}
One readily verifies that the Lie algebra $\mathrm{H}_{24}$ is isomorphic to the Heisenberg Lie algebra~$\ell$. Indeed, the linear map
\[
\Psi : \mathrm{H}_{24} \longrightarrow \ell
\]
defined by
\[
e_1 \mapsto  e_3, \qquad
e_2 \mapsto e_1, \qquad
e_3 \mapsto e_2, \qquad
e_4 \mapsto e_4
\]
is a Lie algebra isomorphism. It therefore remains to recover the corresponding flat, torsion-free connection on~$\ell$ by applying Definition~\ref{Defofnabla}.
This yields
\begin{align}\label{H24}
	\nabla_{e_1} e_2 &= e_3,\qquad
	\nabla_{e_4} e_4 = e_1.
\end{align}
This flat, torsion-free connection is listed in Table~\ref{4flatcomplete} under the Lie algebra~$\ell_{13}$.  Consider now the flat, torsion-free connection labeled~$26$, which is defined by
\begin{align}\label{H26}
	\nabla_{e_2} e_3 &= e_1,\qquad
	\nabla_{e_4} e_3 = e_1,\qquad \nabla_{e_4} e_4 = e_2.
\end{align}
The underlying Lie algebra corresponding to this connection has the following nonvanishing Lie brackets:
\begin{align}
	\mathrm{H}_{26} :\quad
	[e_2,e_3] = e_1, \qquad
	[e_3,e_4] = -e_1.
\end{align}
One readily verifies that the Lie algebra $\mathrm{H}_{26}$ is isomorphic to the Lie algebra~$\ell$.
Indeed, the linear map
\[
\Psi : \mathrm{H}_{26} \longrightarrow \ell
\]
defined by
\[
e_1 \mapsto  e_3, \qquad
e_2 \mapsto e_1, \qquad
e_3 \mapsto e_2, \qquad
e_4 \mapsto e_1+e_4
\]
provides a Lie algebra isomorphism.
Applying Definition~\ref{Defofnabla}, we recover the induced flat, torsion-free connection on~$\ell$, which is given by
\begin{align}
	\nabla_{e_1} e_2 &= e_3,\qquad
	\nabla_{e_4} e_4 = e_2.
\end{align}
This connection coincides with the one previously defined in~\eqref{H24}.

Consider now the  flat Lie algebra labeled~$25$, whose associated flat, torsion-free connection is given by
\begin{align}
	\nabla_{e_2} e_3 &= e_1,\qquad
	\nabla_{e_4} e_3 = e_1,\qquad 	\nabla_{e_4} e_4 = e_2.
\end{align}
The underlying Lie algebra corresponding to this connection has the following nonvanishing Lie brackets:
\begin{align}
	\mathrm{H}_{25} :\quad
	[e_2,e_3] = e_1, \qquad
	[e_3,e_4]=-e_1.
\end{align}
One easily checks that the Lie algebra $\mathrm{H}_{25}$ is isomorphic to the Heisenberg Lie algebra~$\ell$. 
Specifically, the linear map
\[
\Psi : \mathrm{H}_{25} \longrightarrow \ell
\]
given by
\[
e_1 \mapsto e_3, \qquad
e_2 \mapsto e_1, \qquad
e_3 \mapsto e_2, \qquad
e_4 \mapsto e_1+e_4
\]
provides an isomorphism of Lie algebras. 
Applying Definition~\ref{Defofnabla}, we then recover the induced flat, torsion-free connection on~$\ell$:
\begin{align}\label{H25}
	\nabla_{e_1} e_2 &= e_3, \qquad
	\nabla_{e_4} e_4 = e_1.
\end{align}
This connection coincide with the one given in \eqref{H24}.

In summary, after eliminating redundancies, we have removed the following: $\mathbf{15}$, which is isomorphic to $\mathbf{13}$; $\mathbf{23}$, which is isomorphic to $\mathbf{22}$; $\mathbf{26}$, which is isomorphic to $\mathbf{24}$; and $\mathbf{25}$, which is also isomorphic to $\mathbf{24}$. The remaining flat Lie algebras are pairwise non-isomorphic. 

After identifying the underlying Lie algebras of the flat Lie algebras listed above, it is straightforward to check that they are not isomorphic by using the automorphism groups described in the next subsection. 
Specifically, two flat, torsion-free connections on a Lie algebra (for example, the Heisenberg Lie algebra~$\ell$) are isomorphic if and only if there exists a Lie algebra automorphism
\[
\Psi : \ell \longrightarrow \ell
\] 
such that
\begin{align}\label{relationequiva}
	\nabla_x^2 = \Psi \circ \nabla^1_{\Psi^{-1}(x)} \circ \Psi^{-1}, \quad \text{for all } x \in \ell.
\end{align}

This identification is very simple: instead of directly searching for isomorphisms between Lie algebras, one can work with the automorphism group, which is easy to determine. 
As a consequence, we obtain a complete classification of geodesically complete, flat Lie algebras of dimension~4. 
The remaining work consists merely of a cohomological computation to obtain the classification of eight-dimensional symplectic Lie algebras with Lagrangian normal subgroups, via Lagrangian extensions of flat Lie algebras.

\subsection{Automorphism Groups of Four-Dimensional Nilpotent Lie Groups}

\begin{align*}
	\mathrm{Aut}(\mathfrak{a})
	= \{ M \in \mathcal{M}_n(\mathbb{R}) \mid \det(M) \neq 0 \}
	= \mathrm{GL}_n(\mathbb{R}).
\end{align*}

\begin{align*}
	\mathrm{Aut}(\ell)
	= \left\{
	\begin{pmatrix}
		a_{11} & a_{12} &0 &0\\
		a_{21} & a_{22} &0&0 \\
		a_{31} & a_{32} & a_{11}a_{22}-a_{12}a_{21} & a_{34} \\
		a_{41} & a_{42} & 0& a_{44}
	\end{pmatrix}
	\;\middle|\;
	a_{44}(a_{11}a_{22}-a_{12}a_{21} )\neq0
	\right\}.
\end{align*}

\begin{align*}
	\mathrm{Aut}(\mathfrak{t})
	= \left\{
	\begin{pmatrix}
		a_{11} & 0 &0 & a_{14} \\
		a_{21} & a_{11}a_{44} & 0 & a_{24} \\
		a_{31} & a_{21}a_{44} & a_{11}a_{44}^2 & a_{34} \\
		0 &0& 0& a_{44}
	\end{pmatrix}
	\;\middle|\;
	a_{11}a_{44}\neq0
	\right\}.
\end{align*}

\end{document}